\documentclass[11pt,reqno]{amsart}
\usepackage[margin=1in]{geometry}
\usepackage{amsmath,amssymb,amsfonts,amsthm}
\usepackage{mathrsfs}
\usepackage{xcolor}
\usepackage{physics}      % \norm, \abs, \comm
\usepackage{bm}
\usepackage[scr=rsfs,cal=boondox]{mathalfa}
\usepackage{url}
\usepackage[colorlinks=true,linkcolor=blue!60!black,citecolor=blue!60!black,urlcolor=blue!60!black]{hyperref}
\theoremstyle{plain}
\newtheorem{theorem}{Theorem}

\newtheorem{lemma}[theorem]{Lemma}
\newtheorem{corollary}[theorem]{Corollary}
\theoremstyle{definition}

\newtheorem{remark}[theorem]{Remark}
\newtheorem{definition}[theorem]{Definition}
\newtheorem{notation}[theorem]{Notation}
\newcommand{\p}{\partial}
\newcommand{\R}{\mathbb{R}}
\newcommand{\lref}[1]{Lemma \ref{#1}}
\newcommand{\Dt}{D_t}
\newcommand{\scrC}{\mathscr{C}}
\newcommand{\rref}[1]{Remark \ref{#1}}
\newcommand{\tref}[1]{Theorem \ref{#1}}
\newcommand{\nref}[1]{Notation \ref{#1}}
\newcommand{\cref}[1]{Corollary \ref{#1}}
\newcommand{\tdr}{\tilde{r}}
\newcommand{\tdu}{\tilde{u}}
\newcommand{\tdp}{\tilde{\pi}}
\newcommand{\hatL}{\hat{L}}
\newcommand{\tdL}{\tilde{L}}
\newcommand{\blue}[1]{#1}   % journal-revision colouring removed for the arXiv version

\newcommand{\dist}[1]{\mathrm{dist}(#1,\Gamma_t)}
\newcommand{\stp}{\bm{\partial}}
\newcommand{\tdf}{\tilde{f}}
\newcommand{\tdg}{\tilde{g}}
\newcommand{\tdh}{\tilde{h}}
\newcommand{\tds}{\tilde{s}}
\newcommand{\tdv}{\tilde{v}}
\newcommand{\tdz}{\tilde{z}}

\begin{document}

\title[The linearized Israel--Stewart equations with a physical vacuum boundary]{The linearized Israel--Stewart equations with a physical vacuum boundary}
\author{Runzhang Zhong}
\address{Department of Mathematics, University of Missouri, Columbia, MO 65211, USA}
\email{rz47r@missouri.edu}  
\thanks{The author gratefully acknowledges support from NSF grant DMS-2107701.}

\begin{abstract}
In this article, we consider the Israel-Stewart equations of relativistic viscous fluid dynamics with bulk viscosity. We investigate the evolution of the equations linearized about solutions that satisfy the physical vacuum boundary condition and establish local well-posedness of the corresponding Cauchy problem.
\end{abstract}

\maketitle

\section{Introduction}\label{sec1}

The study of relativistic fluids comprises a large and active field of study in mathematics and physics. Many important physical systems in astrophysics, high-energy physics, and cosmology are described with the formalism of relativistic hydrodynamics. See, e.g., \cite{rezzolla2013relativistic} for an overview of the topic and its applications. Relativistic hydrodynamics is also a fertile source of mathematical problems, see e.g. \cite[Chapter IX]{choquet2008general} and the reviews \cite{abbrescia2023relativistic,disconzi2024recent}.

On the mathematical side, most of the community's focus has been on studies of relativistic ideal fluids, i.e., fluids without viscosity, which are described by the well-known relativistic Euler equations. While many important problems remain regarding the mathematical properties of the relativistic Euler equations (see the two recent reviews \cite{abbrescia2023relativistic,disconzi2024recent} for a discussion of recent progress and open problems), there are also compelling reasons to investigate relativistic fluids \emph{with viscosity.} Indeed, there are two important situations where viscosity is needed for an accurate description: 
\begin{itemize}
    \item The quark-gluon plasma. The quark-gluon plasma is an exotic state of matter that forms at extreme temperatures and densities achieved in modern-day particle colliders such as the Large Hadron Collider and the Relativistic Heavy Ion Collider. During a brief period of time after the collision, matter deconfines and quarks and gluons enter into a plasma-like state. Study of this plasma-like state provides researchers with a unique opportunity to understand properties of matter at extreme conditions. In fact, the discovery of the quark-gluon plasma was named by the American Physical Society as one of the ten most important discoveries in physics in the first decade of this century \cite{Chodos-APSNewsQGP-2010}. Careful analysis of data from these experiments reveals that the quark-gluon plasma behaves like a relativistic fluid with viscosity \cite{Romatschke:2017ejr}. 
    \item Neutron star mergers. The recent detection of gravitational waves produced in the merger of binary neutron stars \cite{TheLIGOScientific:2017qsa} opened the door to a whole range of previously inaccessible phenomena and has the potential to revolutionize our understanding of nature \cite{Dexheimer:2020zzs,MUSES:2023hyz}. While for a long time theoretical studies of mergers of neutron stars assumed that they could be reasonably modeled by the Einstein-Euler system, recent state-of-the-art numerical simulations convincingly show that viscous effects can be relevant for such systems \cite{Alford:2017rxf}. In particular, effects of viscosity on the gravitational-wave signal emitted by mergers of neutron stars are expected to be within the sensitivity range of the next-generation of gravitational wave detectors \cite{Most:2018eaw,Most:2019onn}.
\end{itemize}

Despite their importance in physics, the mathematical study of theories of relativistic fluids with viscosity has traditionally been slow. The last decade, however, has seen increasing interest from the mathematical community in theories of relativistic viscous fluids, with several rigorous results finally established, including local \cite{bemfica2019causality,Bemfica:2020zjp,Bemfica-Disconzi-Graber-2021,Bemfica-Disconzi-Rodriguez-Shao-2021,Freistuhler-Temple-2014,Freistuhler-Temple-2017,Freistuhler-Temple-2018,Freistuhler-2020} and global \cite{Sroczinski-2019,Sroczinski-2020,Freistuhler-Sroczinski-2021,Freistuhler-Reintjes-Sroczinski-2022,Freistuhler-Sroczinski-2024-arxiv} well-posedness, formation of singularities \cite{Freistuhler-2023-arxiv-2,Disconzi-Hoang-Radosz-2023}, existence of viscous shocks \cite{Freistuhler-2021,pellhammer2022oscillating,Pellhammer-2023}, and the study of the Einstein constraint equations \cite{Disconzi-Isenberg-Maxwell-2024-arxiv} for viscous matter.

In this paper, we initiate the study of relativistic \emph{viscous} hydrodynamics in a free-boundary setting, thus providing a further step on the mathematical development of relativistic fluids with viscosity. In the absence of viscosity, the study of free-boundary relativistic fluids goes back to the '30s with the seminal works of Tolman, Oppenheimer, and Volkoff \cite{Tolman:1939jz,Tolman:1934za,Oppenheimer:1939ne}, wherein they derived the now-famous TOV equations that describe a general relativistic static star, and Oppenheimer and Snyder \cite{Oppenheimer:1939ue}, who introduced the first model of gravitational collapse and subsequent formation of a black hole in general relativity, in what is now known as the Oppenheimer-Snyder model. Subsequent works on relativistic free-boundary ideal fluids include \cite{Lindblom-1988,Rendall-Schmidt-1991,Makino-1998,LeFloch-Ukai-2009,Makino-2016,Makino-2017,Makino-2018,Heilig-1995,Makino-2019-arxiv,Makino-2019-2-arxiv,Hadzic-Lin-2021,Brauer-Karp-2011,Brauer-Karp-2014,Friedrich-1998,Oliynyk-2012-1,Ginsberg-2019,Ginsberg-Lindblad-2023,Oliynyk-2017,Oliynyk-2019-arxiv,Miao-Shahshahani-Wu-2021,Miao-Shahshahani-2024,Jang-LeFloch-Masmoudi-2016,Trakhinin-2009,Hadzic-Shkoller-Speck-2019,
Rendall-1992,Disconzi-2015,Hao-Huo-Miao-2024-arxiv,Makino-2021}. See section 1.6 of \cite{disconzi2022relativistic} for a review of previous works on free-boundary relativistic fluids without viscosity.

Inclusion of viscosity is important in relativistic free-boundary models in order to describe dissipative effects in star evolution, especially  in the case of neutron stars (see \cite{Duez:2018jaf,Duez:2004nf}, references therein, and the above discussion), and, more recently, accretion disks near black holes \cite{Chandra:2015iza,Foucart:2017axc,Cordeiro:2023ljz}, where viscous effects can be important to understand some measurements made by the Event Horizon Telescope \cite{EHT-Astro-2022}. This work is the first one to mathematically address a model of relativistic viscous hydrodynamics in a free-boundary context. More precisely, we investigate equations of Israel-Stewart type originally introduced in \cite{Denicol:2012cn} and that are widely used in applications devoted to viscous effects in the quark-gluon plasma \cite{Romatschke:2017ejr} and neutron star mergers \cite{Shibata:2017xht,Shibata:2017jyf}. We focus on the case when the fluid satisfies the well-known physical vacuum boundary condition, which in the absence of viscosity is needed to ensure that the fluid-vacuum interface can accelerate. We remark that when viscosity is present, further conditions, in comparison to the non-viscous case, are needed to ensure that the free boundary moves with bounded non-zero acceleration. In other words, our first task below is to determine what the physical vacuum boundary condition means in the viscous context, i.e., what assumptions are needed in order to guarantee that the free boundary can accelerate.

As said, the model we will investigate was originally introduced in \cite{Denicol:2012cn}. This model is sometimes referred to as the DNMR model after its authors, although for historical reasons it is more often called a Israel-Stewart, Israel-Stewart-like, or Israel-Stewart-Mueller model due to its similarities with a model of relativistic viscous fluids introduced by Israel and Stewart in the '70s \cite{Israel:1976tn,Israel:1979wp,Stewart-1977,Israel-Stewart-1979,Stewart-Stewart-1976}, which in turn was inspired by a model of non-relativistic viscous fluids introduced earlier by Mueller \cite{Mueller-1967}. The DNMR or Israel-Stewart model contains a wide range of fluid variables, including three types of viscosity, namely, bulk and shear viscosity and heat conductivity. The local well-posedness of these equations has been obtained only in the special case when bulk viscosity is the only viscous contribution. Thus, it would be premature to study it in a free boundary context with all viscous contributions, and we restrict ourselves to the situation where only bulk viscosity is present. 

In comparison with previous works on ideal fluids with a physical vacuum boundary, the main new difficulty we encounter is that viscosity in Israel-Stewart models is governed by a new variable satisfying an additional equation of motion, see the third equation in \eqref{E:BasicPre}. This is a transport equation sourced by the divergence of the fluid's velocity, and a simple derivative counting shows that direct application of transport estimates leads to a derivative loss. In the case of a physical vacuum boundary, this problem is compounded by the need to control the fluid variables in weighted spaces that have to be introduced due to the fluid's degeneracy on the free boundary. We overcome this difficulty by uncovering some special structures at the level of the linearized equations. 

It turns out that the linearized equations by themselves already present a rich and complex structure, and their local well-posedness cannot be derived from standard techniques. Thus, here we focus on establishing local well-posedness for the linearized Israel-Stewart equations with a physical vacuum boundary and bulk viscosity. A recent article providing a priori estimates for linearized relativistic equations with a physical vacuum boundary and an ideal gas equation of state is also included here \cite{luczak2024priori}. The solution to the linearized equations will be the basis for our solution to the full nonlinear problem that we will present in an upcoming work \cite{Zhong-in_prep}. In order to shorten the presentation, we will focus on aspects of the problem that are genuinely new compared to the case of an ideal fluid. Thus, we will be brief when arguments are a simple modification of those employed in the case of an ideal fluid with a physical vacuum boundary, mostly notably to those found in \cite{Jang_Masmoudi_2014,disconzi2022relativistic,ifrim2023compressible}.
\subsection{The Equations of Motion}
In this article, we consider the equations of Israel-Stewart type with bulk viscosity proposed in \cite{Denicol:2012cn}, which is the most widely used model of relativistic fluids with viscosity \cite{Romatschke:2017ejr}. We consider the situation where the viscous contribution only comes from bulk viscosity, thus ignoring 
the dissipation due to shear and \blue{particle diffusion}. The fluid's motion can be characterized by an energy-momentum tensor
\begin{equation}\label{EnergyMomentum}
    T_{\mu\nu} = \varrho u_\mu u_\nu + (p + \Pi)\Delta_{\mu\nu},
\end{equation} 
and a conserved quantity called rest-mass current
\begin{equation*}
    J^\mu = n u^\mu,
\end{equation*}
where $\varrho$ is the energy density, $\Pi$ is the bulk viscosity, $n$ is the baryon density, $p$ is the pressure subject to the equation of state
\begin{equation*}
    p = p(\varrho,n),
\end{equation*}
$u$ is the $4$-velocity of the fluid subject to the normalization condition
\begin{equation}\label{Normalization}
    u^\mu u_\mu = -1,
\end{equation}
and the $2$-tensor $\Delta_{\mu\nu}=g_{\mu\nu} + u_\mu u_\nu$ is the projection onto the space orthogonal to $u$. Additionally, 
$g$ is the space-time metric, which we will take to be the Minkowski metric in that this already contains many difficulties in the study of a free-boundary fluids which we introduce below.

The evolution of the fluid is given by the system
\begin{equation}\label{E:BasicPre}
    \left\{ \begin{aligned}
    & \p_\mu T^{\mu\nu} = 0,
    \\
    & \p_\mu J^\mu = 0 ,
    \\
    & \tau_\Pi u^\mu \p_\mu \Pi + \Pi  + \lambda \Pi^2 + \zeta \p_\mu u^\mu = 0,
\end{aligned}\right.
\end{equation}
where the first equation is the conservation of energy and momentum, the second equation is the conservation of mass and the third equation is the additional relaxation-type equation satisfied by the bulk scalar $\Pi$, in which $\tau_\Pi = \tau_\Pi(\varrho)$ is the bulk relaxation time,  $\lambda=\lambda(\varrho)$ is a transport coefficient and $\zeta=\zeta(\varrho)$ is the coefficient of bulk viscosity.

Throughout and including the above, we adopt the rectangular spacetime coordinate $\{x^\alpha\}_{\alpha = 0,1,2,3}$ of $\R^{1+3}$, where we identify $x^0$ with the time coordinate $t$. We use $x=(x^1,x^2,x^3)$ to denote the spatial coordinates. Einstein sum convention is employed. The vector field $\p_\alpha$ is a short-hand expression for $\frac{\p}{\p x^\alpha}$. Greek indices will vary from $0$ to $3$, and Latin indices will vary from $1$ to $3$. Indices are raised and lowered with the Minkowski metric $g$.

\subsection{Physical Vacuum Boundary}
It was shown in \cite{bemfica2019causality} that the system is strongly hyperbolic if the following causality condition is satisfied
\begin{equation}\label{Causality_condition}
    \biggl[\frac{\zeta}{\tau_\Pi} + n\left(\frac{\p p}{\p n}\right)_\varrho\biggr]\frac{1}{\varrho + p + \Pi} \leq 1 - \left(\frac{\p p}{\p \varrho}\right)_n,
\end{equation}
where $\varrho + p + \Pi>0$ and $\tau_\Pi>0$. One can use the standard theory of FOSH developed by Friedrichs-Lax-Kato \cite{friedrichs1954symmetric,lax2005weak,kato1975cauchy} to obtain the local well-posedness.

In this article, we will consider the motion of a relativistic viscous gas and the physical situation where vacuum states are allowed, i.e., the density $\varrho$ is allowed to vanish. \blue{In this scenario, the boundary of the fluid is not prescribed but allowed to move with the fluid's motion, which makes the problem a free-boundary problem. }The gas is located in the following region where the density is non-vanishing, i.e., the following moving domain at time $t$,
\begin{equation*}
    \Omega_t = \{(t,x) \in \R^3 \vert \varrho(t,x)>0\},
\end{equation*}
whose boundary denoted by $\Gamma_t:=\p\Omega_t$ is called the vacuum boundary at time $t$, which is advected by the fluid velocity $u$. So, the dynamics of the fluid is confined within
\begin{equation*}
    \mathscr{D}_T := \bigcup_{t\in[0,T]}\Omega_t
\end{equation*}
and the fluid's free boundary at time $t$ is defined as
$$\Gamma:=\bigcup_{t\in[0,T]}\Gamma_t,$$
where $T>0$ is the time existence of the solutions.

\blue{For a gas, the correct boundary condition that comes from physical consideration is that the pressure and the density vanish on the free boundary $\Gamma_t$, i.e.,
$$p|_\Gamma = \varrho|_\Gamma = 0.$$
Due to the absence of matter ($\varrho=0$) on the free boundary $\Gamma_t$, it is natural to also impose that the bulk viscosity $\Pi$ vanishes on $\Gamma_t$, i.e.,
$$\Pi|_\Gamma = 0.$$}

We will consider the gas where the pressure $p$ is determined by the density $\varrho$ through the following barotropic equation of state,
\begin{equation}\label{Equation_of_state}
    p(\varrho) = \varrho^{\kappa + 1},
\end{equation}
where $\kappa > 0$. Such gas is called a barotropic gas. With the equation of state \eqref{Equation_of_state}, the variable $n$ gets decoupled from the system \eqref{E:BasicPre}, which, with the energy-momentum tensor $T$ given by \eqref{EnergyMomentum}, can be written as the following Cauchy problem
\begin{equation}\label{E:Basic}
    \left\{ \begin{aligned}
    & u^\mu\p_\mu \varrho + (\varrho^{\kappa + 1} + \varrho + \Pi)\p_\mu u^\mu  = 0,
    \\
    & (\varrho^{\kappa + 1} + \varrho + \Pi) u^\mu\p_\mu u^\alpha + (\kappa + 1)\varrho^\kappa\Delta^{\alpha\mu}\p_\mu \varrho + \Delta^{\alpha\mu}\p_\mu\Pi = 0,
    \\
    & \tau_\Pi u^\mu \p_\mu \Pi + \Pi + \lambda \Pi^2 + \zeta \p_\mu u^\mu = 0,\\
    & u^\mu u_\mu = -1,\\
    & (\varrho, u,\Pi)|_{t=0}=(\rho_0,u_0,\Pi_0),
\end{aligned}\right.
\end{equation}
where the first two equations in the system \eqref{E:Basic} is derived from the projections of the first equation in the system \eqref{E:BasicPre} onto the space parallel and orthogonal to $u$ and $\rho_0,u_0,\Pi_0$ are initial data for the system. Notice that the system \eqref{E:Basic} becomes a degenerate hyperbolic due to the vanishing of both $\varrho$ and $\Pi$. In this case, the classical theories for FOSH no longer apply. 

\blue{Here, we emphasize that assuming a barotropic equation of state leads to the decoupling of baryon density. A more comprehensive analysis would be needed to account for a non-barotropic equation of state. To justify this assumption, we note two key points: first, barotropic equations of state are commonly employed in the literature (see \cite{bemfica2021nonlinear} and references therein); second, in the free-boundary setting, the problem for existence of a non-barotropic equation of state remains open, even in the simpler setting of a relativistic non-viscous fluid. As such, it is premature to address this in the present work.}

\blue{Our goal is to model a physical scenario in which the boundary $\Gamma_t$ can move with non-zero acceleration. This is particularly important for capturing phenomena like the motion of a rotating star.} It turns out that the decay rates of $\varrho$ and $\Pi$ play a crucial role in the problem. 

\blue{Now, we determine the appropriate decay rates for the variables $\varrho$ and $\Pi$ to ensure that the boundary $\Gamma_t$ can move with a finite non-zero acceleration. }From the second equation of \eqref{E:Basic}, we obtain that the fluid evolves with the $4$-acceleration $a$ near the free boundary $\Gamma_t$,
\begin{equation}\label{Four_acceleration}
    a^\alpha := u^\mu\p_\mu u^\alpha = -\frac{(\kappa + 1)\varrho^\kappa\Delta^{\alpha\mu}\p_\mu\varrho + \Delta^{\alpha\mu}\p_\mu\Pi}{\varrho^{\kappa + 1} + \varrho + \Pi}.
\end{equation}
\begin{notation}
    For functions $f,g$, the notation $f\sim g$ means there exists a constant $C>0$ such that 
    $$\frac{1}{C}g\leq f\leq Cg.$$
\end{notation}
\vspace{0.5em}
We assume the decay rates of the variables $\varrho$ and $\Pi$ near the free boundary is the following
\begin{equation}\label{DecayRatesAssumption}
    \varrho(t,x)\sim \dist{x}^\sigma,\Pi(t,x)\sim\dist{x}^\delta,
\end{equation}
where $\sigma,\delta>0$ and the function $\dist{\bullet}$ is the distance from the free boundary $\Gamma_t$ at time $t$. Plugging \eqref{DecayRatesAssumption} into the $4$-acceleration formula \eqref{Four_acceleration}, we find that near the free boundary $\Gamma_t$,
\begin{equation*}
    a(t,x)\sim\begin{cases}
    \dist{x}^{\sigma\kappa-1} + \dist{x}^{\delta-\sigma -1} \,\,& \sigma\leq\delta,
    \\
    \infty \,\,&\sigma > \delta.
\end{cases}
\end{equation*}
Here, the second condition $\sigma>\delta$ results in an infinite acceleration $a$ on the free boundary $\Gamma_t$, which is not very physical and leads to singular behavior of the fluid. 

\blue{
To ensure a finite nonzero acceleration, the first condition $\sigma\leq\delta$ of the $4$-acceleration $a$ can be further divided into the following $2$ cases, either
\begin{equation}\label{Physical_vacuum_boundary_condition}
    \begin{cases}
        \sigma = \frac{1}{\kappa},\\
        \delta \geq \sigma + 1,
    \end{cases}
\end{equation}
or
\begin{equation*}
    \begin{cases}
        \sigma > \frac{1}{\kappa},\\
        \delta = \sigma + 1.
    \end{cases}
\end{equation*}
We expect a realistic dynamics in the first case \eqref{Physical_vacuum_boundary_condition}, since this case reduces to the physical vacuum boundary condition of the relativistic ideal fluid (when the bulk viscosity $\Pi=0$) in \cite{disconzi2022relativistic}. We also notice that if we take $\sigma > \frac{1}{\kappa}$ in condition $\eqref{Physical_vacuum_boundary_condition}$, the boundary $\Gamma_t$ has a zero acceleration. Such zero acceleration would not allow us to model scenarios such as a rotating star.

In particular, we point out that in this article we will treat the case where $\delta = \frac{1}{\kappa}+2$. The choice is motivated by the following reasons: first, the case where $\delta>2+\frac{1}{\kappa}$ can be addressed using a similar argument; second, the case where $\frac{1}{\kappa}+1\leq\delta<\frac{1}{\kappa}+2$ leads to a very singular behavior of the coefficients of the linearized equation and the energy estimate cannot be closed with our method.

In our analysis, it is convenient to have variables that are comparable to the distance function $\dist{\bullet}$. So, we introduce new variables
\begin{equation*}
    r = \varrho^\kappa, \pi = \Pi^{\frac{\kappa}{2\kappa + 1}}.
\end{equation*}
As a result, the new variables $r$ and $\pi$ have the decay rates
\begin{equation}\label{RealPhysicalVacuumBoundary}
    r(t,x),\pi(t,x)\sim\dist{x}.
\end{equation}
The decay rate assumption \eqref{RealPhysicalVacuumBoundary} should be viewed as a constraint---namely, a condition imposed on the initial data that can be shown to be preserved under the evolution.

This work presents the first study of relativistic viscous fluids in the free-boundary setting. Since no prior techniques have been developed in this context, we must introduce certain technical assumptions. Specifically, throughout the article, we assume that the relaxation time $\tau_{\Pi}$ is a constant, chosen for convenience as $\tau_{\Pi} = \frac{1}{2+\frac{1}{\kappa}}$. The coefficient of bulk viscosity, $\zeta(r^\frac{1}{\kappa})$, is assumed to be a smooth function satisfying $\zeta(r^\frac{1}{\kappa}) \sim r^{2+\frac{1}{\kappa}}$ near the free boundary $\Gamma_t$, with its higher-order derivatives $\left(\frac{d}{dr}\right)^k\left(\frac{\zeta(r^\frac{1}{\kappa})}{r^{2+\frac{1}{\kappa}}}\right)$ remaining bounded for large k. Additionally, we assume that $\lambda = \lambda(r^\frac{1}{\kappa})$ is a smooth function with the higher-order derivatives $\left(\frac{d}{dr}\right)^k\left(\lambda(r^\frac{1}{\kappa})r^{2+\frac{1}{\kappa}}\right)$ remaining bounded for large $k$. These assumptions are consistent with the causality condition \eqref{Causality_condition}.
\vspace{0.5em}
\begin{remark}
    Here, we note that such $\zeta$ and $\lambda$ do exist, for example, they can be chosen as $\zeta(r^\frac{1}{\kappa})=r^{2+\frac{1}{\kappa}}$ and $\lambda(r^\frac{1}{\kappa}) = r^{\gamma-\frac{1}{\kappa}}$, where $\gamma$ is sufficiently large. To simplify the notation, we define $$b(r):=\frac{\zeta(r^\frac{1}{\kappa})}{r^{2+\frac{1}{\kappa}}}, c(r):=\lambda(r^\frac{1}{\kappa})r^{2+\frac{1}{\kappa}}.$$ Given the assumption on $\zeta$ and $\lambda$, it follows that $b(r)\sim 1$, with its derivatives $b^{(k)}(r)$ bounded, and the derivatives of $c(r)$ also remains bounded. These conditions are imposed because, in the higher-order estimates, we will differentiate the equation, and the derivatives could act on the terms $\frac{\zeta(r^\frac{1}{\kappa})}{r^{2+\frac{1}{\kappa}}}$ and $\lambda(r^\frac{1}{\kappa})r^{2+\frac{1}{\kappa}}$, potentially introducing $\frac{1}{r}$ factor. Then, the terms generated blow up near the free boundary. We want to avoid this situation in our energy estimates.
\end{remark}
\vspace{0.5em}}

Using $(r,u,\pi)$ as our new unknowns, we write \eqref{E:Basic} as the following,
\begin{subequations}\label{The_nonlinear_system}
    \begin{align}
        \label{rEquation}
    & u^\mu\p_\mu r + \kappa r a_1 \p_\mu u^\mu = 0,
    \\
    \label{uEquation}
    &  u^\mu\p_\mu u^\alpha + \left(1+\frac{1}{\kappa}\right)a_2\Delta^{\alpha\mu}\p_\mu r + \left(2+\frac{1}{\kappa}\right)ra_3\Delta^{\alpha\mu}\p_\mu\pi = 0,
    \\
    \label{pEquation}
    & \blue{u^\mu\p_\mu\pi + \pi + rc(r)a_0^{3+\frac{1}{\kappa}} + ra_4\p_\mu u^\mu = 0,}\\
    \label{EulerianNormalization}
    & u^\mu u_\mu = -1,\\
    \label{EulerianBoundaryCondition}
    & r|_{\Gamma_t}=0,\pi|_{\Gamma_t} = 0,\\
    \label{EulerianInitialData}
    & (r, u,\pi)|_{t=0}=(r_0,u_0,\pi_0),
    \end{align}
\end{subequations}
where
\begin{align*}
    a_0 &= \frac{\pi}{r}, \\
    a_1 &= r + 1 + r^2a_0^{2+\frac{1}{\kappa}},\\
    a_2 &= \frac{1}{a_1},\\
    a_3 &= a_0^{1+\frac{1}{\kappa}}a_2,\\
    \blue{a_4}&\blue{=\frac{b(r)}{a_0^{1+\frac{1}{\kappa}}}.}
\end{align*}
\blue{\subsection{Main Result}Our main result establishes the local well-posedness of the linearized system of the system \eqref{The_nonlinear_system}, under the assumption that the background solutions $(r,u,\pi)$ satisfy appropriate bounds. For the precise formulation of the linearized equations and the definition of the function spaces used in our analysis, we refer the reader to Section 2.

Before stating the theorem, we establish some notations that will be used throughout the paper.

\vspace{0.5em}
\begin{notation}
    \begin{itemize}
        \item We use the notation $\scrC(\bullet,\bullet,\cdots)$ to represent a constant that depends continuously on its arguments and is non-decreasing with respect to them. This notation frequently appears in the energy estimate, in which the value of $\scrC$ may vary from line to line.
        \item The symbol $\p$ denotes a generic spatial derivative and the symbol $\stp$ denotes a generic spacetime derivative.
        \item The symbol $W^{k,p}$ denotes the standard Sobolev space with regularity $k>0$ and exponent $p>0$ equipped with the norm $\norm{\bullet}_{W^{k,p}}$.
        \item The symbol \(\lesssim\) is used in the standard sense, defined as:
        \[
            A \lesssim B \iff A \leq \mathcal{C} B,
        \]
        for some constant $\scrC>0$. In certain contexts, the constant $\scrC$ depends on the background solutions. To emphasize this dependency, we may write $A \lesssim \mathcal{C} B$, explicitly highlighting the constant.
    \end{itemize}
\end{notation}
\vspace{0.5em}
Now, we state the main result. The theorem establishes short-time existence and uniqueness of solutions to the linearized system, along with a priori energy estimates that provide control over the solution.
\vspace{0.5em}
\begin{theorem}
    Let $(r,u,\pi)$ be a sufficiently regular solution to the system \eqref{The_nonlinear_system} defined on $[0,T]$ such that it satisfies the physical vacuum boundary condition \eqref{RealPhysicalVacuumBoundary}. There exists $T^*=T^*(\norm{(r,u,\pi)}_{W^{1,\infty}})>0$ such that for any $(\tilde{f},\tilde{g},\tilde{h})\in\widetilde{\mathcal{H}}^{2l}$, there exists a unique solution $(\tdr,\tdu,\tdp)\in\mathcal{H}^{2l}$ to the Cauchy problem \eqref{The_linearized_system} satisfying the following estimate,
    \begin{equation}
        \norm{(\tdr(t),\tdu(t),\tdp(t))}_{{\mathcal{H}}^{2l}}\leq \mathscr{C}(\norm{(r,u,\pi)}_{W^{2l+1,\infty}})\left(\norm{(\tilde{r}_0,\tilde{u}_0,\tilde{\pi}_0)}_{{\mathcal{H}}^{2l}}+\int_0^t\norm{(\tdf,\tdg,\tdh)}_{\mathcal{H}^{2l}}\,\mathrm{d}s\right), 
    \end{equation}
    for any $t\in [0,T^*]$.
\end{theorem}
\vspace{0.5em}

The proof relies on a duality argument, similar to the approach in \cite{Jang_Masmoudi_2014}. To apply the duality argument, we first establish a weighted energy estimate for the solution of the linearized equations. The basic energy estimate is derived in Section 2, in which we leverage the special cancellations that arise from the good weighted structure of the linearized equations.

To obtain higher-order estimate, we differentiate the linearized system with weighted elliptic operators that preserve the good weighted structure. The definition and properties of these weighted elliptic operator are detailed in Section 3. We also derive the higher-order estimates and conclude the section with the proof the main theorem via the duality argument. 
}

\section{Basic energy estimate for the  linearized equations}\label{LinearizedEquations} In this section, we formulate the linearized problem of system \eqref{The_nonlinear_system} and establish the basic energy estimate for the linearized equations, assuming suitable conditions on the background solution about which we linearize. This energy estimate is a crucial component in proving the local well-posedness of the linearized system.

\subsection{The Linearized Equations}\label{SS:TheLinearizedEquation} We consider a one-parameter family $(r_\lambda,u_\lambda,\pi_\lambda)$ of solutions to the system \eqref{The_nonlinear_system} such that $(r_\lambda,u_\lambda,\pi_\lambda)|_{\lambda=0}=(r,u,\pi)$. Then, the variables $(\tilde{r},\tilde{u},\tilde{\pi}) := \frac{d}{d\lambda}|_{\lambda = 0}(r_\lambda,u_\lambda,\pi_\lambda)$ defined on the moving domain $\Omega_t$ solve the following linearized problem,
\begin{subequations}\label{The_linearized_system}
    \begin{align}
        \label{trEquation}
        & u^\mu \partial_\mu \tilde{r} + \kappa r a_1 \partial_\mu \tilde{u}^\mu 
        + \partial_\mu r \tilde{u}^\mu + V_1 \tilde{r} + r^2 Z_1 \tilde{\pi} = \tdf, \\
        \label{tuEquation}
        & u^\mu \partial_\mu \tilde{u}^\alpha + \left(1 + \frac{1}{\kappa}\right) a_2 \Delta^{\alpha\mu} \partial_\mu \tilde{r} 
        + \left(2 + \frac{1}{\kappa}\right) r a_3 \Delta^{\alpha\mu} \partial_\mu \tilde{\pi} \nonumber \\
        & \quad + (V_2)^\alpha \tilde{r} + (W_2)^\alpha_\beta \tilde{u}^\beta + (Z_2)^\alpha \tilde{\pi} = \tdg^\alpha, \\
        \label{tpEquation}
        & u^\mu \partial_\mu \tilde{\pi} + r a_4 \partial_\mu \tilde{u}^\mu 
        + V_3 \tilde{r} + (W_3)_\alpha \tilde{u}^\alpha + Z_3 \tilde{\pi} = \tdh, \\
        \label{EulerianOrthogonality}
        & u^\mu \tilde{u}_\mu = 0, \\
        \label{LinearizedInitialData}
        & (\tilde{r}, \tilde{u}, \tilde{\pi}) \big|_{t=0} = (\tilde{r}_0, \tilde{u}_0, \tilde{\pi}_0),
        \end{align}
\end{subequations}
where $V,W,Z$ are coefficients linear in $\stp(r,u,\pi)$ with coefficients that are smooth functions of $(r,u,\pi)$ given as the following,
\begin{equation*}
    \begin{aligned}
        V_1 &= \kappa \left[2r + 1 + \left(1 - \frac{1}{\kappa}\right) r^2 a_0^{2 + \frac{1}{\kappa}}\right] \partial_\mu u^\mu, \\
        (V_2)^\alpha &= 
            \left\{\left(\frac{1}{\kappa} r a_0^{2 + \frac{1}{\kappa}} - 1\right) a_2 
            \left[\left(1 + \frac{1}{\kappa}\right) \partial_\mu r + \left(2 + \frac{1}{\kappa}\right) r a_0^{1 + \frac{1}{\kappa}} \partial_\mu \pi \right] \right. \\ 
            &\quad \left. - \frac{1}{\kappa}\left(2 + \frac{1}{\kappa}\right) a_0^{1 + \frac{1}{\kappa}} \partial_\mu \pi \right\} a_2 \Delta^{\alpha \mu}, \\
        \blue{V_3 }&= \blue{a_0^{3+\frac{1}{\kappa}}\left(c(r)+rc'(r)-(3+\frac{1}{\kappa})c(r)\right)+\frac{b(r)}{a_0^{1+\frac{1}{\kappa}}}\p_\mu u^\mu + r\frac{b'(r)}{a_0^{1+\frac{1}{\kappa}}}\p_\mu u^\mu}\\
        &\quad\blue{+\left(1+\frac{1}{\kappa}\right)\frac{b(r)}{a_0^{1+\frac{1}{\kappa}}}\p_\mu u^\mu},\\
        W_1 &= 0, \\
        (W_2)^{\alpha \beta} &= \left\{ \partial_\mu u^\alpha g^{\mu \beta} + a_2 \left(u^\mu g^{\alpha \beta} + u^\alpha g^{\mu \beta}\right) 
        \left[\left(1 + \frac{1}{\kappa}\right) \partial_\mu r + \left(2 + \frac{1}{\kappa}\right) r a_0^{1 + \frac{1}{\kappa}} \partial_\mu \pi\right] \right\}, \\
        (W_3)^\beta &= g^{\mu \beta} \partial_\mu \pi, \\
        Z_1 &= \kappa \left(2 + \frac{1}{\kappa}\right) a_0^{1 + \frac{1}{\kappa}} \partial_\mu u^\mu, \\
        (Z_2)^\alpha &= 
            -\left(2 + \frac{1}{\kappa}\right) a_0^{\frac{1}{\kappa}} a_2 \Delta^{\alpha \mu} \Biggl\{\left(1 + \frac{1}{\kappa}\right) r a_0 a_2 \partial_\mu r + \left(2 + \frac{1}{\kappa}\right) r^2 a_0^{2 + \frac{1}{\kappa}} a_2 \partial_\mu \pi \\
            &\quad - \left(1 + \frac{1}{\kappa}\right) \partial_\mu \pi \Biggr\}, \\
        \blue{Z_3} &= \blue{1 + (3+\frac{1}{\kappa})c(r)a_0^{2+\frac{1}{\kappa}} - \left(1 + \frac{1}{\kappa}\right) \frac{b(r)}{a_0^{2 + \frac{1}{\kappa}}} \partial_\mu u^\mu}.
    \end{aligned}
\end{equation*}
Here, we introduce the source terms $(\tdf, \tdg, \tdh)$ to make the equations inhomogeneous, as the proof remains the same whether or not these source terms are included. \blue{Since each member of the one-parameter family  $(r_\lambda, u_\lambda, \pi_\lambda)$ does not necessarily share the same domain---unlike when working in a Lagrangian frame---there are no natural boundary conditions imposed on the linearized variables  $(\tilde{r}, \tilde{u}, \tilde{\pi})$. However, the behavior of the linearized variables on the boundary is incorporated into our choice of function spaces. This will be reflected later in the energy estimates, where all the boundary terms vanish due to the weight vanishing on the boundary.}

\blue{As the nonlinear system \eqref{The_nonlinear_system}, the linearized system \eqref{The_linearized_system} is also degenerate on the boundary. In order to deal with the degenerate character of the problem, we adopt a framework of weighted energies and function spaces rather than the standard Sobolev framework. In particular, we bring attention to the third term $\p_\mu r\tdu^\mu$ on the LHS of equation \eqref{trEquation}. At first glance, it may seem like a lower-order term based on a simple derivative count. However, in our weighted space framework, this term is not lower order with respect to our energy since it does not have the correct weight to be controlled. In our energy estimates, we must employ special cancellations that arise from a good weighted structure of the linearized system to handle this term.}

\subsection{Function Spaces}\blue{As mentioned above, the linearized system \eqref{The_linearized_system} is degenerate on the boundary due to the presence of the degenerate factor $r$, which is comparable to the distance $\dist{\bullet}$. The degenerate nature, along with the natural energy admitted by the linearized system, motivates the use of weighted Sobolev spaces, which we define here.
}
\vspace{0.5em}
\begin{definition}
    The weighted $L^2$ space with weight $\omega\geq0$, denoted by $L^2(\omega)$, is the space of all distributions $f$ such that
    \begin{equation*}
        \norm{f}_{L^2(\omega)}^2 = \int \omega\abs{f}^2\,\mathrm{d}x<\infty.
    \end{equation*}
\end{definition}

The weight $\omega$ will be played out by $r$, which is a solution to \eqref{The_nonlinear_system} satisfying the physical vacuum condition \eqref{Physical_vacuum_boundary_condition}.\blue{ As such, $r$ is a defining function for the moving domain $\Omega_t$ and is comparable to the distance $\dist{\bullet}$ from the boundary $\Gamma_t$. }

\blue{First, we define the base weighted space in which the linearized velocity $\tdu$ live. Motivated by the orthogonality condition \eqref{EulerianOrthogonality}, which states that the background velocity $u$ is orthogonal to the linearized velocity $\tdu$ with respect to the Minkowski metric $g$, we introduce the following weighted space of vector fields,
$$L^2_{\perp}(r^\frac{1}{\kappa}):=\{\tdu\in L^2(r^\frac{1}{\kappa})\mid u^\mu\tdu_\mu=0\}.$$
We treat $L^2_\perp$ as a closed subspace of the space $L^2(r^\frac{1}{\kappa})$ of vector fields.}
The base space for the solutions $(\tdr,\tdu,\tdp)$ is then defined as 
\begin{equation*}
    \mathcal{H}:=\mathcal{H}(\Omega_t) = L^2(r^\frac{1-\kappa}{\kappa}(t,\bullet))\times L^2_\perp(r^\frac{1}{\kappa}(t,\bullet))\times L^2(r^\frac{1}{\kappa}(t,\bullet)).
\end{equation*}
\blue{
The base space for the source terms $(\tdf,\tdg,\tdh)$ is defined as 
$$\widetilde{\mathcal{H}}= L^2(r^\frac{1-\kappa}{\kappa}(t,\bullet))\times L^2(r^\frac{1}{\kappa}(t,\bullet))\times L^2(r^\frac{1}{\kappa}(t,\bullet))$$

Above, the first and third components consist of real-valued functions, while the second component consists of vector fields.}

\blue{Here, we define the base energy of the linearized system. The energy admitted by the linearized system is the following,
\begin{equation}\label{BasicEnergy}
    E(t)=\frac{1}{2} \int_{\Omega_t} r^{\frac{1}{\kappa}-1}\left[\tilde{r}^2+\frac{\kappa^2}{\kappa+1} r a_1^2 \tilde{u}_\alpha \tilde{u}^\alpha+\frac{(2\kappa+1)\kappa}{\kappa+1}r\frac{a_1a_0^{2+\frac{2}{\kappa}}}{b(r)}\tdp^2\right] dx.
\end{equation}}

\blue{We note that since $\tdu$ is orthogonal to $u$ and $u$ is assumed to be time-like by the normalization condition \eqref{Normalization}, it follows that $\abs{\tdu}_g\geq 0$, which in turn ensures $E(t)\geq 0$. Here, the energy involve the norm of the four-velocity rather than just the three-velocity, even though $\tdu^0$ can be determined in terms of three-velocity by the orthogonal condition. This is because we will construct a full four-velocity belonging to the corresponding function spaces, rather than constructing the spatial components first and then solving for $\tdu^0$ via orthogonality.}

\blue{
Additionally, we highlight the power $\frac{1}{\kappa}-1$ arises naturally from the structure of the linearized equations. The fact that $\tdu$ and $\tdp$ carry an additional power of $r$ in their weight, compared to $\tdr$, comes from the following observations: the divergence of $\tdu$ terms in \eqref{trEquation} and \eqref{tpEquation} are weighted by $r$, the derivative of $\tdr$ in \eqref{tuEquation} is unweighted, and the derivative of $\tdp$ in \eqref{tuEquation} is weighted by $r$.

Later, we will show that the energy $E(t)=E[(\tdr,\tdu,\tdp)](t)$, viewed as a functional of the linearized variables, is equivalent to the norm of the base space. More precisely, we have $E[\tdr,\tdu,\tdp]\sim \norm{(\tdr,\tdu,\tdp)}_\mathcal{H}^2$. This directly follows from \lref{BasicEllInftyControl}.}

Now, we define high-order weighted Sobolev spaces.
\vspace{0.5em}
\begin{definition}
    We define $H^{j,\sigma}:=H^{j,\sigma}(\Omega_t)$, where $j\geq 0$ is an integer and $\sigma >-\frac{1}{2}$, to be the space of all distributions in $\Omega_t$ such that
    \begin{equation*}
    \norm{f}_{H^{j,\sigma}}^2:=\sum_{\abs{\alpha}\leq j}\norm{r^\sigma\p^\alpha f}_{L^2(\Omega_t)}^2<\infty.
\end{equation*}
\end{definition}
The weighted spaces $H^{j,\sigma}$ satisfy the following Hardy-type embedding. We refer readers to \cite{ifrim2023compressible} for the proofs.

\vspace{0.5em}
\begin{lemma}\label{WeightedEmbedding}
    Assume that $j_1>j_2 \geq 0$ and $\sigma_1>\sigma_2>-\frac{1}{2}$ with $j_1-j_2=\sigma_1-\sigma_2$. Then we have
    \begin{equation*}
        H^{j_1, \sigma_1} \subset H^{j_2, \sigma_2}.
    \end{equation*}
\end{lemma}
\vspace{0.5em}

As a corollary of the above lemma we have embeddings of the weighted Sobolev spaces into standard Sobolev spaces:

\vspace{0.5em}
\begin{corollary}
    Assume that $\sigma>0$ and $\sigma \leq j$. Then we have
    \begin{equation*}
        H^{j, \sigma} \subset H^{j-\sigma}.
    \end{equation*}
\end{corollary}
\vspace{0.5em}

Then, by the standard Sobolev embedding theorem, we have

\vspace{0.5em}
\begin{corollary}
    \begin{equation*}
        H^{j, \sigma} \subset C^s, \quad 0 \leq \blue{s < j-\sigma-\frac{3}{2}}.
    \end{equation*}
\end{corollary}
\vspace{0.5em}

\blue{Now, we define the high-order space $\mathcal{H}^{2l}$ for the regular solutions. As the base space, the linearized velocity $\tdu$ lives in a weighted space of vector fields orthogonal to $u$, given by
\begin{equation*}
    H^{2l,\frac{1}{2\kappa}}_\perp:=\{\tdu\in H^{2l,\frac{1}{2\kappa}}\mid u^\mu\tdu_\mu = 0 \}.
\end{equation*}
Thus, the solutions $(\tdr,\tdu,\tdp)$ with $2l$ regularity are in the space
\begin{equation*}
    \mathcal{H}^{2l} = H^{2l,\frac{1}{2\kappa}+l-\frac{1}{2}}\times H^{2l,\frac{1}{2\kappa}+l}_\perp \times H^{2l,\frac{1}{2\kappa}+l}.
\end{equation*}
The source terms $(\tdf,\tdg,\tdh)$ with $2l$ regularity are in the space,
\begin{equation*}
    \widetilde{\mathcal{H}}^{2l} = H^{2l,\frac{1}{2\kappa}+l-\frac{1}{2}}\times H^{2l,\frac{1}{2\kappa}+l} \times H^{2l,\frac{1}{2\kappa}+l}.
\end{equation*}
Above, as before, the first and third components consist of real-valued functions, while the second component consists of vector fields.

Notice the special relation between derivatives and weights in the definition of the higher-order space: each time we increase the derivatives by $2$, we also introduce an additional power of $r$ in the weight. This follows from the structure of the system when considering the second-order evolution of $\tdr$ and the divergence of $\tdu$. At the leading order, these quantities are governed by a wave-like operator similar to
$$\Dt^2-r\Delta,$$
which suggests we define the higher-order space based on powers of $r\Delta$.

We emphasize that the precise wave-like operator actually has a more complicated form than simply $r\Delta$. However, we present this simplified form to highlight that one has a degenerate elliptic operator, where the degeneracy arises from the factor of $r$. (See Section 3 for the precise form of the weighted elliptic operator.)}

\subsection{Basic Energy Estimate} Here, we establish the basic energy estimate of the solutions $(\tilde{r},\tilde{u},\tilde{\pi})$ to the linearized equations \eqref{The_linearized_system} under suitable assumptions on the background solutions $(r,u,\pi)$. Throughout the proof, we will use the following differential formula for the moving domain $\Omega_t$,

\begin{equation}\label{MovingDomainFormula}
    \frac{\mathrm{d}}{\mathrm{d} t} \int_{\Omega_t} f dx=\int_{\Omega_t} \frac{1}{u^0} \mathrm{D}_t f dx+\int_{\Omega_t}f\, \partial_i\left(\frac{u^i}{u^0}\right) dx .
\end{equation}
Above the derivative $D_t :=u^\mu\p_\mu$ is the material derivative parallel to the $4$-velocity $u$.

\blue{To justify the form of this formula, we note that $\frac{u^i}{u^0}$ represents the physical $3$-velocity (see \cite{rezzolla2013relativistic}, Sec 7.1), which also motivates our definition of the flow map in the following:}
\vspace{0.5em}
\blue{\begin{notation}\label{LagrangianCoordinate}
    We let $y=(y^1,y^2,y^3)$ be the coordinate on the initial surface $\Omega_0$ induced by the rectangular spacetime coordinate $\{x^\alpha\}_{\alpha = 0,1,2,3}$. We define the map $\eta: [0,T]\times\Omega_0 \to \R^{1+3}$ to be the flow map of the velocity $\left\{\frac{u^i}{u^0}\right\}_{i=1,2,3}$, $\eta=\eta(t,y)$ solve the following equation,
    \begin{equation}\label{Flow_map}
        \left\{ \begin{aligned}
        & \p_t \eta^i = \left(\frac{u^i}{u^0}\right)\circ\eta
        \\
        & \eta(0,y) = y
    \end{aligned}\right..
    \end{equation}
    It follows that,
        \begin{equation*}
        \{t\}\times\eta(t,\Omega_0) = \{t\}\times \Omega_t.
    \end{equation*}
    We use Capital Latin indices $K,L,...$, etc. ranging from $1$ to $3$ to denote the indices of Lagrangian spatial coordinates. We use $\p_K$ to denote the derivative $\frac{\p}{\p y^K}$ with respect to the Lagrangian coordinate. 
\end{notation}}
\vspace{0.5em}

\blue{To establish the basic energy estimate, we need a uniform $L^\infty$ control of the coefficient $a_0=\frac{\pi}{r}$ over a short time interval. Here, we point out that in the following lemmas of controlling $a_0$ and its derivatives, the standard Sobolev norms for $(r,u,\pi)$ are only in spatial derivatives. That is, we have $L^\infty$ control only over the spatial derivatives of $(r,u,\pi)$. However, in the estimates, time derivatives also appear due to terms like the spacetime divergence of $u$ in \eqref{rEquation} and \eqref{pEquation}. To address this, we solve the time derivatives $\p_t(r,u,\pi)$ 
algebraically in terms of the spatial derivative $\p(r,u,\pi)$, with coefficients that are smooth functions of $(r,u,\pi)$ by \eqref{The_nonlinear_system}. This allows us to control the time derivatives using the norms of the spatial derivatives. (See \rref{TimeInSpace}.)

\vspace{0.5em}
\begin{lemma}\label{BasicEllInftyControl}
    Let $(r,u,\pi)$ be a solution to the system \eqref{The_nonlinear_system} defined on the time interval $[0,T]$. Then, there exists $T^*=T^*(\norm{(r,u,\pi)}_{L_t^\infty W^{1,\infty}_x})>0$ such that
    \begin{align*}
        \sup_{t\in[0,T^*]}\norm{\frac{\pi}{r}}_{L^\infty(\Omega_t)}&\leq 2\norm{\frac{\pi_0}{r_0}}_{L^\infty(\Omega_0)},\\
        \sup_{t\in[0,T^*]}\norm{\frac{r}{\pi}}_{L^\infty(\Omega_t)}&\leq 2\norm{\frac{r_0}{\pi_0}}_{L^\infty(\Omega_0)}.
    \end{align*}
\end{lemma}
\begin{proof}
     we prove the statement by a bootstrap argument. Assume that
     \begin{equation}\label{Bootstrap}
         \begin{aligned}
            \sup_{t\in[0,T^*]}\norm{\frac{\pi}{r}}_{L^\infty(\Omega_t)}&\leq 100\norm{\frac{\pi_0}{r_0}}_{L^\infty(\Omega_0)},\\
            \sup_{t\in[0,T^*]}\norm{\frac{r}{\pi}}_{L^\infty(\Omega_t)}&\leq 100\norm{\frac{r_0}          {\pi_0}}_{L^\infty(\Omega_0)}.
         \end{aligned}
     \end{equation}
     We first compute the evolution of $a_0 = \frac{\pi}{r}$ with equations \eqref{rEquation} and \eqref{pEquation},
     \begin{equation*}
         \begin{aligned}
             \Dt a_0 &= \frac{\Dt\pi}{r} - \frac{1}{r}a_0\Dt r\\
             &= \frac{1}{r}\left(-\pi - rc(r)a_0^{3+\frac{1}{\kappa}}-r\frac{b(r)}{a_0^{1+\frac{1}{\kappa}}}\p_\mu u^\mu+\kappa r a_1a_0\p_\mu u^\mu\right)\\
             &=\left[\kappa\p_\mu u^\mu(r+1)-1\right]a_0 + (\kappa r^2\p_\mu u^\mu - c(r))a_0^{3+\frac{1}{\kappa}}-b(r)\p_\mu u^\mu\frac{1}{a_0^{1+\frac{1}{\kappa}}}.
         \end{aligned}
     \end{equation*}
     Integrating along the flow line, we have 
     \begin{align*}
         \left(\frac{\pi}{r}\circ\eta\right)(t,y) &= \frac{\pi_0}{r_0}(y) + \int_0^t \biggl(\frac{1}{u^0}\biggl(\left[\kappa\p_\mu u^\mu(r+1)-1\right]a_0 \\
         &\quad + (\kappa r^2\p_\mu u^\mu - c(r))a_0^{3+\frac{1}{\kappa}}-\frac{b(r)}{a_0^{1+\frac{1}{\kappa}}}\p_\mu u^\mu\biggr)\biggr)(s,\eta(s,y))ds.
     \end{align*}
     Using the bootstrap assumption \eqref{Bootstrap} and solving the time derivatives in terms of the spatial derivatives as in \rref{TimeInSpace}, we have
     \begin{equation}\label{AfterBootstrapI}
         \begin{aligned}
             \abs{\left(\frac{\pi}{r}\right)(t,\eta(t,y))}&\leq \norm{\frac{\pi_0}{r_0}}_{L^\infty(\Omega_0)} \\
             &\quad + t\scrC\left(\norm{(r,u,\pi)}_{L_t^\infty W^{1,\infty}(\Omega_t)},\norm{\frac{\pi_0}{r_0}}_{L^\infty(\Omega_0)},\norm{\frac{r_0}{\pi_0}}_{L^\infty(\Omega_0)}\right).
         \end{aligned}
     \end{equation}
     Similarly, we compute the evolution for $\frac{1}{a_0}$,
     \begin{equation*}
         \begin{aligned}
             \Dt\left(\frac{1}{a_0}\right) &= -\frac{1}{a_0^2}\Dt a_0\\
             &= -\frac{1}{a_0^2}\left(\left[\kappa\p_\mu u^\mu(r+1)-1\right]a_0 + (\kappa r^2\p_\mu u^\mu - c(r))a_0^{3+\frac{1}{\kappa}}-b(r)\p_\mu u^\mu\frac{1}{a_0^{1+\frac{1}{\kappa}}}\right)\\
             &=b(r)\p_\mu u^\mu\frac{1}{a_0^{3+\frac{1}{\kappa}}}-\left[\kappa\p_\mu u^\mu(r+1)-1\right]\frac{1}{a_0} - (\kappa r^2\p_\mu u^\mu -c(r))a_0^{1+\frac{1}{\kappa}}.
         \end{aligned}
     \end{equation*}
     Using a similar argument as above, we have the following estimate,
     \begin{equation}\label{AfterBootstrapII}
         \begin{aligned}
             \abs{\left(\frac{r}{\pi}\right)(t,\eta(t,y))}&\leq \norm{\frac{r_0}{\pi_0}}_{L^\infty(\Omega_0)}\\
              &\quad + t\scrC\left(\norm{(r,u,\pi)}_{L_t^\infty W^{1,\infty}(\Omega_t)},\norm{\frac{\pi_0}{r_0}}_{L^\infty(\Omega_0)},\norm{\frac{r_0}{\pi_0}}_{L^\infty(\Omega_0)}\right).
         \end{aligned}
     \end{equation}
     Taking $L^\infty$ in $y$ and $\sup_{t\in[0,T^*]}$ of the estimates \eqref{AfterBootstrapI} and \eqref{AfterBootstrapII}, we obtain
     \begin{equation*}
         \begin{aligned}
             \sup_{t\in[0,T^*]}\norm{\frac{\pi}{r}}_{L^\infty(\Omega_t)}&\leq \norm{\frac{\pi_0}{r_0}}_{L^\infty(\Omega_0)} \\
             &\quad + T^*\scrC\left(\norm{(r,u,\pi)}_{L_t^\infty W^{1,\infty}(\Omega_t)},\norm{\frac{\pi_0}{r_0}}_{L^\infty(\Omega_0)},\norm{\frac{r_0}{\pi_0}}_{L^\infty(\Omega_0)}\right),\\
             \sup_{t\in[0,T^*]}\norm{\frac{r}{\pi}}_{L^\infty(\Omega_t)}&\leq \norm{\frac{r_0}{\pi_0}}_{L^\infty(\Omega_0)} \\
             &\quad+ T^*\scrC\left(\norm{(r,u,\pi)}_{L_t^\infty W^{1,\infty}(\Omega_t)},\norm{\frac{\pi_0}{r_0}}_{L^\infty(\Omega_0)},\norm{\frac{r_0}{\pi_0}}_{L^\infty(\Omega_0)}\right).
         \end{aligned}
     \end{equation*}
     Now, we close the bootstrap by taking $T^*$ as the following,
     \begin{equation*}
         T^*\leq\frac{\min{\left(\norm{\frac{\pi_0}{r_0}}_{L^\infty(\Omega_0)},\norm{\frac{r_0}{\pi_0}}_{L^\infty(\Omega_0)}\right)}}{\scrC\left(\norm{(r,u,\pi)}_{L_t^\infty W^{1,\infty}(\Omega_t)},\norm{\frac{\pi_0}{r_0}}_{L^\infty(\Omega_0)},\norm{\frac{r_0}{\pi_0}}_{L^\infty(\Omega_0)}\right)}.
     \end{equation*}
\end{proof}

In the following lemma, we show that the physical vacuum boundary condition can be propagated by the flow.

\vspace{0.5em}
\begin{lemma}
    Let $(r,u,\pi)$ be a solution to the system \eqref{The_nonlinear_system} defined on the time interval $[0,T]$. Let $r_0$ and $\pi_0$ be the initial data satisfying the physical vacuum boundary condition \eqref{RealPhysicalVacuumBoundary}. Then, $r$ and $\pi$ satisfy the physical vacuum boundary condition.
\end{lemma}
\begin{proof}
    Here, we only prove that the physical vacuum boundary condition can be propagated for $r$. The same argument applies to $\pi$.

    We integrate \eqref{rEquation} along the flow line and obtain,
    \begin{equation*}
        r(t,\eta(t,y)) = r_0(y) + \int_0^t\left(-\frac{1}{u^0}\kappa r a_1\p_\mu u^\mu\right)(s,\eta(s,y))ds.
    \end{equation*}
    Taking absolute value of the above equation, by a direct estimate, we have
    \begin{equation*}
        \abs{r(t,\eta(t,y))}\leq r_0(y) + \int_0^t\scrC\left(\norm{(r,u,\pi)}_{L_t^\infty W^{1,\infty}}\right)\abs{r(s,\eta(s,y))}ds.
    \end{equation*}
    By the Gronwall's inequality, we have
    \begin{equation*}
        \abs{r(t,\eta(t,y))}\leq r_0(y)\exp{T^*\norm{(r,u,\pi)}_{L_t^\infty W^{1,\infty}}}\sim\operatorname{dist}(y,\Gamma_0),
    \end{equation*}
    where $T^*>0$ is the time existence coming from \lref{BasicEllInftyControl}.

    To show the lower bound, we can use a similar argument, which is a straightforward extension of the standard Gronwall's inequality. From \eqref{rEquation} and \eqref{Flow_map}, we have
    \begin{equation*}
        \p_t(r\circ\eta) = \left(-\frac{1}{u^0}\kappa r a_1 \p_\mu u^\mu\right)\circ\eta\geq - \scrC\left(\norm{(r,u,\pi)}_{L^\infty_tW^{1,\infty}}\right)r\circ\eta.
    \end{equation*}
    Let $A:= \exp{\scrC\left(\norm{(r,u,\pi)}_{L^\infty_tW^{1,\infty}}\right)t}$. It follows directly that $A(0)=1$. Multiplying the above inequality by $A$, we obtain
    \begin{equation*}
        \p_t(Ar\circ\eta)\geq 0.
    \end{equation*}
    Integrating in time, we have
    \begin{equation*}
    \exp{\scrC\left(\norm{(r,u,\pi)}_{L^\infty_tW^{1,\infty}}\right)t}r\circ\eta\geq r_0(y)
    \end{equation*}
    Moving the exponential to the RHS and taking $t=T^*$, we have
    $$r(t,\eta(t,y))\geq \exp{-\scrC\left(\norm{(r,u,\pi)}_{L^\infty_tW^{1,\infty}}\right)T^*}r_0(y)\sim\operatorname{dist}(y,\Gamma_0).$$

    We denote by $\p_y$ the generic spatial derivative with respect to the Lagrangian coordinate defined in \nref{LagrangianCoordinate}, and $\p_x$ the generic spatial derivative with respect to the Eulerian coordinate on the time slice $\Omega_t$. Taking $\p_y$ of the equation \eqref{Flow_map}, we have
    \begin{equation}\label{FlowMapDerivative}
        \p_t\p_y\eta = (\p_x u)\circ\eta \p_y\eta.
    \end{equation}
    This implies
    \begin{equation*}
        \abs{\p_t\p_y\eta}\leq\scrC(\norm{\p_x u}_{L^1_tL^\infty_x})\abs{\p_y\eta}
    \end{equation*}
    By an application of Gronwall inequality and a similar argument as in the above argument where we establish the lower bound of $r\circ\eta$, we have
    \[\exp{-\scrC(\norm{\p_x u}_{L^1_tL^\infty_x})T^*}\lesssim\p_y\eta\lesssim \exp{\scrC(\norm{\p_x u}_{L^1_tL^\infty_x})T^*}.\]

    Now, differentiating in time the determinant $\det{\p_y\eta}$ and using the cyclic property of the trace, by the evolution \eqref{FlowMapDerivative}, we obtain
    \begin{align*}
        \frac{\p}{\p t}\det\p_y\eta &= \det \p_y\eta\operatorname{tr}\left((\p_y\eta)^{-1}\frac{\partial}{\partial t}\p_y\eta\right)\\
        &=\det\p_y\eta\operatorname{tr}\left((\p_y\eta)^{-1}(\p_x u)\circ\eta\p_y\eta\right)\\
        &=\det\p_y\eta\operatorname{tr}((\p_x u)\circ\eta),
    \end{align*}
    which leads to the bound
    \[\abs{\p_t\det\p_y\eta}\leq \scrC(\norm{\p_x u}_{L^1_tL^\infty_x})\abs{\det\p_y\eta}.\]
    Applying the similar Gronwall inequality argument, we obtain
    \[\exp{-\scrC(\norm{\p_x u}_{L^1_tL^\infty_x})T^*}\lesssim\det\p_y\eta\lesssim \exp{\scrC(\norm{\p_x u}_{L^1_tL^\infty_x})T^*}.\]

    By the formula for the inversion of matrix, we have
    \[
        \p_x\eta^{-1} = \frac{\operatorname{adj} \p_y\eta\circ\eta^{-1}}{\det \p_y\eta\circ\eta^{-1}},
    \]
    where $\operatorname{adj}(A)$ denotes the adjoint of the matrix $A$. This leads to a bound
    \begin{equation}\label{FlowMapInverseDerivativeBound}
        \abs{\p_x\eta^{-1}}\lesssim\exp{\scrC(\norm{\p_x u}_{L^1_tL^\infty_x})T^*}.
    \end{equation}

    For any point $y\in\Omega_0$, we pick the point $y_0$ such that $\operatorname{dist}(y,\Gamma_0) = \abs{y-y_0}$. Then, 
    \begin{align*}
        \dist{\eta(t,y)}&\leq \abs{\eta(t,y)-\eta(t,y_0)}\\
        &\lesssim \norm{\p_y\eta}_{L^1_tL^\infty_x}\abs{y-y_0}\\
        &\lesssim \exp{\scrC(\norm{\p_x u}_{L^1_tL^\infty_x})T^*}\operatorname{dist}(y,\Gamma_0).
    \end{align*}

    Similarly, with the bound \eqref{FlowMapInverseDerivativeBound}, we can show that
    \begin{equation*}
        \operatorname{dist}(y,\Gamma_0)\lesssim\exp{\scrC(\norm{\p_x u}_{L^1_tL^\infty_x})T^*}\dist{\eta(t,y)}.
    \end{equation*}
    So, the above shows that we have the equivalence
    \[\dist{\eta(t,y)}\sim \operatorname{dist}(y,\Gamma_0),\]
    which finishes the proof.

\end{proof}

In the higher-order estimates, we need $L^\infty$ control over the derivatives of $a_0 = \frac{\pi}{r}$. The reason is the following. Naturally, when deriving higher-order estimates, we differentiate the system using an appropriate differential operator, which may act on the coefficient $a_0$. Observe that
\begin{equation*}
    \p a_0=\p\left(\frac{\pi}{r}\right)=\frac{\p\pi}{r}-\frac{\pi\p r}{r^2}=\frac{1}{r}(\p\pi-a_0 \p r).
\end{equation*}
The differentiation applied to $a_0$ introduces an additional factor of $\frac{1}{r}$, potentially reducing the power of $r$ in the weighted estimate and thus compromising the good weighted structure of the linearized system mentioned above.

To resolve this, we impose appropriate constraints on the initial data. Specifically, the initial data $r_0$ and $\pi_0$ are chosen such that all spatial derivatives $\p^m\left(\frac{\pi_0}{r_0}\right)$ remain bounded. Such initial data indeed exist; for instance, setting $\pi_0=r_0$.
\vspace{0.5em}
\begin{lemma}
    Let $(r,u,\pi)$ be a solution to the system \eqref{The_nonlinear_system} defined on the time interval $[0,T]$. Assume that the initial data $r_0$ and $\pi_0$ satisfy 
    \[\norm{\p^{2k}\left(\frac{\pi_0}{r_0}\right)}_{L^\infty(\Omega_0)},\norm{\p^{2k}\left(\frac{r_0}{\pi_0}\right)}_{L^\infty(\Omega_0)}<\infty,\]
    for $0\leq k \leq l$, where $l$ is a large integer. Then, there exists a time existence $T^*=T^*(\norm{(r,u,\pi)}_{L_t^\infty W^{1,\infty}_x})>0$ such that
    \begin{equation}
        \begin{aligned}
            \sup_{t\in[0,T^*]}\norm{\p^{2k}\left(\frac{\pi}{r}\right)}_{L^\infty(\Omega_t)} \leq \scrC(\norm{(r,u,\pi)}_{L^\infty_t W^{2k+1,\infty}(\Omega_t)}),\\
            \sup_{t\in[0,T^*]}\norm{\p^{2k}\left(\frac{r}{\pi}\right)}_{L^\infty(\Omega_t)} \leq \scrC(\norm{(r,u,\pi)}_{L^\infty_t W^{2k+1,\infty}(\Omega_t)})
        \end{aligned}
    \end{equation}
    for $0\leq k\leq l$.
\end{lemma}
\begin{proof}
    Here, the time $T^*$ is taken as the same time in \lref{BasicEllInftyControl}. We will prove the statement by an induction. Assume, as the inductive hypothesis, that
    \begin{equation*}        
        \begin{aligned}
            \sup_{t\in[0,T^*]}\norm{\p^{2k}\left(\frac{\pi}{r}\right)}_{L^\infty(\Omega_t)} &\leq \scrC(\norm{(r,u,\pi)}_{L^\infty_t W^{2k+1,\infty}(\Omega_t)}),\\
            \sup_{t\in[0,T^*]}\norm{\p^{2k}\left(\frac{r}{\pi}\right)}_{L^\infty(\Omega_t)} &\leq \scrC(\norm{(r,u,\pi)}_{L^\infty_t W^{2k+1,\infty}(\Omega_t)})
        \end{aligned}
    \end{equation*}
    We now aim to prove the case for $k+1$ .
    
    From the proof of \lref{BasicEllInftyControl}, we have the following evolution for $a_0$,
    \[\Dt a_0 =\left[\kappa\p_\mu u^\mu(r+1)-1\right]a_0 + (\kappa r^2\p_\mu u^\mu - c(r))a_0^{3+\frac{1}{\kappa}}-\frac{b(r)}{a_0^{1+\frac{1}{\kappa}}}\p_\mu u^\mu.\]
    
    Taking $\p^{2k+2}$ of the above equation and then commuting $\p^{2k+2}$ with $\Dt$, after simplifying the equation with \rref{TimeInSpace}, we end up with a transport equation linear in $\p^{2k+2}a_0$ where the coefficients are multi-linear in the lower-order derivatives of $a_0$ and the derivatives of $(r,u,\pi)$, multiplied by a smooth function in $(r,u,\pi,a_0)$. 
    
    By a direct transport estimate using standard Gronwall inequality and the induction hypothesis, we can prove the $k+1$ case. A similar argument applies to $\p^{2k+2}\left(\frac{r}{\pi}\right)$.
\end{proof}}

We also need the following Gronwall-type inequality for the energy estimate. \blue{In the lemma below, the condition  $0 < \alpha < 1$  is crucial. If  $\alpha > 1$ , a similar bound cannot be obtained, as the integrating factor method used in the proof would no longer apply. If  $\alpha = 1$ , a bound on  $d$  follows directly from the standard Gronwall inequality.}

\vspace{0.5em}
\begin{lemma}
    Let $d(t)$ be a non-negative function that satisfies the following integral inequality,
    \begin{equation}\label{PreGronwall}
        d(t)\leq c(t) + \int_0^t a(s)d(s) + b(s)d^\alpha(s)\,ds,
    \end{equation}
    where $a,b,c$ are continuous non-negative functions for $t\geq 0$ and $0<\alpha<1$. Then,
    \begin{equation}\label{GronwallTypeInequality}
        \begin{aligned}
            d(t)&\leq \biggl(c^{1-\alpha}(t)\exp\left((1-\alpha)\int_0^t a(s)\, ds\right) 
            \\
        &\quad + (1-\alpha)\int_0^t b(s) \exp\left((1-\alpha)\int_s^t a(r)\, dr\right)\,ds\biggr)^\frac{1}{1-\alpha}
        \end{aligned}
    \end{equation}
\end{lemma}

\begin{proof}
    Given $\tau>0$, let $c_\tau = c(\tau)$. For $t\in [0,\tau]$, we define
    \begin{equation*}
        \phi(t) = c_\tau + \int_0^t a(s)d(s) + b(s)d^\alpha(s)\,ds.
    \end{equation*}
    From \eqref{PreGronwall}, we have $d(t)\leq \phi$ for $t\in[0,\tau]$. We differentiate $\phi$,
    \begin{equation*}
        \begin{aligned}
            \phi'(t) &= a(t)d(t) + b(t)d^\alpha(t)\\
            &\leq a(t)\phi(t) + b(t)\phi^\alpha(t).
        \end{aligned}
    \end{equation*}

    Dividing both sides by $\phi^\alpha$, this gives us a linear differential inequality in $\phi^{1-\alpha}$,
    \begin{equation*}
        \left(\phi^{1-\alpha}(t)\right)'\leq (1-\alpha)a(t)\phi^{1-\alpha}(t) + (1-\alpha)b(t).
    \end{equation*}

    By an elementary method of integrating factor,
    \begin{equation*}
        (\psi(t)\phi^{1-\alpha}(t))' \leq (1-\alpha)b(t)\psi(t),
    \end{equation*}
    where
    \begin{equation*}
        \psi(t) = \exp\left(\int_0^t(\alpha-1)a(s)\,ds\right).
    \end{equation*}

    Taking integral with respect to $t$, and considering that $\tau$ is arbitrary, we have $\eqref{GronwallTypeInequality}$.
\end{proof}

Now, we establish the basic estimate for the linearized equation. \blue{Due to the finite speed of propagation, the system remains non-degenerate away from the boundary, in which case standard theory applies. Thus, it suffices to consider the system in a small neighborhood of the boundary. By this localization argument, the integrals over $\Omega_t$ in the energy estimates can be regarded as being taken over a small region near the boundary.

Moreover, by the physical vacuum boundary condition, we can assume that $r$ remains small near the boundary, i.e.,
\[
\sup_{t\in[0,T]}\norm{r}_{L^\infty(\Omega_t)}<\epsilon\ll 1.
\]

As a result, whenever a term contains sufficiently many powers of $r$ to be controlled by the energy and has an additional positive power of $r$ to spare, we can extract the extra power from the integral, absorbing it into an $\epsilon$ -dependent term. This allows us to treat the term as a small contribution to the energy.

Due to our choice of function spaces and the presence of the degenerate weight  r  in the system \eqref{The_linearized_system}, throughout the energy estimate, integration by parts in the spatial direction produces no boundary terms.}

\vspace{0.5em}
\blue{\begin{theorem}\label{BasicEnergyEstimateTheorem}
    Let $(r,u,\pi)$ be sufficiently smooth solutions to the system \eqref{The_nonlinear_system} defined on $[0,T]$ satisfying the physical vacuum boundary condition \eqref{RealPhysicalVacuumBoundary}. There exists $T^*=T^*(\norm{(r,u,\pi)}_{L^\infty_tW^{1,\infty}(\Omega_t)})>0$ such that if $(\tdr,\tdu,\tdp)$ are sufficiently smooth solutions to the linearized system \eqref{The_linearized_system}, then the solutions obey the following estimate,
    \begin{equation}\label{BasicEnergyEstimate}
        \begin{aligned}
            \norm{(\tdr(t),\tdu(t),\tdp(t))}_\mathcal{H}&\lesssim \scrC(\norm{(r,u,\pi)}_{L^\infty_tW^{1,\infty}_x(\Omega_t)})\\
            &\quad \times \left(\norm{(\tdr_0,\tdu_0,\tdp_0)}_\mathcal{H}+\int^t_0\norm{(\tdf(s),\tdg(s),\tdh(s))}_\mathcal{H}\,ds\right)
        \end{aligned}
    \end{equation}
    for any $t\in[0,T^*]$.
\end{theorem}}
\begin{proof}
    We multiply \eqref{trEquation} by $r^{\frac{1}{\kappa}-1}$, \eqref{tuEquation} by $\frac{\kappa^2}{\kappa+1}r^\frac{1}{\kappa}a_1^2$ and \eqref{tpEquation} by $\frac{(2\kappa+1)\kappa}{\kappa+1}r^\frac{1}{\kappa}\frac{a_1a_0^{\frac{2}{\kappa}+2}}{b(r)}$, we have 
    \begin{align}
        & r^{\frac{1}{\kappa}-1} u^\mu \partial_\mu \tilde{r} + \kappa r^{\frac{1}{\kappa}} a_1 \partial_\mu \tilde{u}^\mu 
        + r^{\frac{1}{\kappa}-1} \partial_\mu r \tilde{u}^\mu + r^{\frac{1}{\kappa}-1} V_1 \tilde{r} 
        + r^{\frac{1}{\kappa}+1} Z_1 \tilde{\pi} = r^{\frac{1}{\kappa}-1} \tdf,\label{wre}\\
         & \frac{\kappa^2}{\kappa+1} r^{\frac{1}{\kappa}} a_1^2 u^\mu \partial_\mu \tilde{u}^\alpha 
         + \kappa r^{\frac{1}{\kappa}} a_1 \Delta^{\alpha\mu} \partial_\mu \tilde{r} 
         + \frac{(2\kappa+1)\kappa}{\kappa+1} r^{\frac{1}{\kappa}+1} a_0^{1+\frac{1}{\kappa}} a_1 \Delta^{\alpha\mu} \partial_\mu \tilde{\pi}\nonumber\\
         & + \frac{\kappa^2}{\kappa + 1} r^{\frac{1}{\kappa}} a_1^2 (V_2)^\alpha \tilde{r} + \frac{\kappa^2}{\kappa+1} r^{\frac{1}{\kappa}} a_1^2(W_2)^\alpha_\gamma \tilde{u}^\gamma + \frac{\kappa^2}{\kappa+1} r^{\frac{1}{\kappa}} a_1^2 (Z_2)^\alpha \tilde{\pi} = \frac{\kappa^2}{\kappa+1} r^{\frac{1}{\kappa}} a_1^2  \tdg^\alpha,\label{wue} \\
         &\frac{(2\kappa+1)\kappa}{\kappa+1}r^\frac{1}{\kappa}\frac{a_1a_0^{2+\frac{2}{\kappa}}}{b(r)}u^\mu\p_\mu\tdp + \frac{(2\kappa+1)\kappa}{\kappa+1}r^{1+\frac{1}{\kappa}}a_1a_0^{1+\frac{1}{\kappa}}\p_\mu\tdu^\mu + \frac{(2\kappa+1)\kappa}{\kappa+1}r^\frac{1}{\kappa}\frac{a_1a_0^{2+\frac{2}{\kappa}}}{b(r)}V_3\tdr\nonumber\\
         &\quad + \frac{(2\kappa+1)\kappa}{\kappa + 1}r^\frac{1}{\kappa}\frac{a_1a_0^{2+\frac{2}{\kappa}}}{b(r)}(W_3)_\alpha \tdu^\alpha + \frac{(2\kappa+1)\kappa}{\kappa+1}r^\frac{1}{\kappa}\frac{a_1a_0^{2+\frac{2}{\kappa}}}{b(r)}Z_3\tdp = \frac{(2\kappa+1)\kappa}{\kappa+1}r^\frac{1}{\kappa}\frac{a_1a_0^{2+\frac{2}{\kappa}}}{b(r)}\tdh.\label{wpe}
    \end{align}
    We let $\mathcal{L}$ be the operator of the above linear system, then the system can be written as 
    \begin{equation}\label{WeightedLinearizedOperator}
        \mathcal{L}\begin{pmatrix}
            \tdr \\ \tdu \\ \tdp
        \end{pmatrix} = \begin{pmatrix}
            r^{\frac{1}{\kappa}-1} \tdf \\ \frac{\kappa^2}{\kappa+1} r^{\frac{1}{\kappa}} a_1^2  \tdg^\alpha \\ \frac{(2\kappa+1)\kappa}{\kappa+1}r^\frac{1}{\kappa}\frac{a_1a_0^{2+\frac{2}{\kappa}}}{b(r)}\tdh
        \end{pmatrix}.
    \end{equation}
    Now, we multiply \eqref{wre} by $\tdr$, \eqref{wue} by $\tdu_\alpha$ and \eqref{wpe} by $\tdp$, yielding
    \begin{equation*}
    \begin{aligned}
        &\frac{1}{2}r^{\frac{1}{\kappa}-1} u^\mu \partial_\mu \tdr^2 + \kappa r^{\frac{1}{\kappa}} a_1 \partial_\mu \tdu^\mu\tdr 
        + r^{\frac{1}{\kappa}-1} \partial_\mu r \tdu^\mu\tdr + r^{\frac{1}{\kappa}-1} V_1 \tdr^2 
        + r^{\frac{1}{\kappa}+1} Z_1 \tdp\tdr = r^{\frac{1}{\kappa}-1} \tdf\tdr,\\
        &\frac{\kappa^2}{2(\kappa+1)} r^{\frac{1}{\kappa}} a_1^2 u^\mu \partial_\mu (\tdu^\alpha\tdu_\alpha) 
         + \kappa r^{\frac{1}{\kappa}} a_1 \partial_\mu \tdr\tdu^\mu 
         + \frac{(2\kappa+1)\kappa}{\kappa+1} r^{\frac{1}{\kappa}+1} a_0^{1+\frac{1}{\kappa}} a_1 \partial_\mu \tdp\tdu^\mu\nonumber\\
         & + \frac{\kappa^2}{\kappa + 1} r^{\frac{1}{\kappa}} a_1^2 (V_2)^\alpha \tdr\tdu_\alpha + \frac{\kappa^2}{\kappa+1} r^{\frac{1}{\kappa}} a_1^2(W_2)^\alpha_\gamma \tdu^\gamma\tdu_\alpha + \frac{\kappa^2}{\kappa+1} r^{\frac{1}{\kappa}} a_1^2 (Z_2)^\alpha \tdp\tdu_\alpha = \frac{\kappa^2}{\kappa+1} r^{\frac{1}{\kappa}} a_1^2  \tdg^\alpha\tdu_\alpha,\\
         &\frac{(2\kappa+1)\kappa}{2(\kappa+1)}r^\frac{1}{\kappa}\frac{a_1a_0^{2+\frac{2}{\kappa}}}{b(r)}u^\mu\p_\mu\tdp^2 + \frac{(2\kappa+1)\kappa}{\kappa+1}r^{1+\frac{1}{\kappa}}a_1a_0^{1+\frac{1}{\kappa}}\p_\mu\tdu^\mu\tdp + \frac{(2\kappa+1)\kappa}{\kappa+1}r^\frac{1}{\kappa}\frac{a_1a_0^{2+\frac{2}{\kappa}}}{b(r)}V_3\tdr\tdp\nonumber\\
         &\quad + \frac{(2\kappa+1)\kappa}{\kappa + 1}r^\frac{1}{\kappa}\frac{a_1a_0^{2+\frac{2}{\kappa}}}{b(r)}(W_3)_\alpha \tdu^\alpha\tdp + \frac{(2\kappa+1)\kappa}{\kappa+1}r^\frac{1}{\kappa}\frac{a_1a_0^{2+\frac{2}{\kappa}}}{b(r)}Z_3\tdp^2 = \frac{(2\kappa+1)\kappa}{\kappa+1}r^\frac{1}{\kappa}\frac{a_1a_0^{2+\frac{2}{\kappa}}}{b(r)}\tdh\tdp.
    \end{aligned}
    \end{equation*}
    We sum over the three equations above and obtain
    \begin{equation}\label{EnergyIdentity}
        \frac{1}{2}r^{\frac{1}{\kappa}-1} u^\mu \partial_\mu \tdr^2 + \frac{\kappa^2}{2(\kappa+1)} r^{\frac{1}{\kappa}} a_1^2 u^\mu \partial_\mu (\tdu^\alpha\tdu_\alpha)+ \frac{(2\kappa+1)\kappa}{2(\kappa+1)}r^\frac{1}{\kappa}\frac{a_1a_0^{2+\frac{2}{\kappa}}}{b(r)}u^\mu\p_\mu\tdp^2=RHS,
    \end{equation}
    where $RHS$ is the following,
    \begin{align*}
        RHS &= r^{\frac{1}{\kappa}-1} \tdf\tdr + \frac{\kappa^2}{\kappa+1} r^{\frac{1}{\kappa}} a_1^2  \tdg^\alpha\tdu_\alpha + \frac{(2\kappa+1)\kappa}{\kappa+1}r^\frac{1}{\kappa}\frac{a_1a_0^{2+\frac{2}{\kappa}}}{b(r)}\tdh\tdp\\
        &\quad - \underline{(\kappa r^{\frac{1}{\kappa}} a_1 \partial_\mu \tdu^\mu\tdr 
        + r^{\frac{1}{\kappa}-1} \partial_\mu r \tdu^\mu\tdr+\kappa r^{\frac{1}{\kappa}} a_1 \partial_\mu \tdr\tdu^\mu)}\\
        &\quad-\underline{\underline{\left(\frac{(2\kappa+1)\kappa}{\kappa+1} r^{\frac{1}{\kappa}+1} a_0^{1+\frac{1}{\kappa}} a_1 \partial_\mu \tdp\tdu^\mu+\frac{(2\kappa+1)\kappa}{\kappa+1}r^{1+\frac{1}{\kappa}}a_1a_0^{1+\frac{1}{\kappa}}\p_\mu\tdu^\mu\tdp\right)}}\\
        &\quad - \left(r^{\frac{1}{\kappa}-1} V_1 \tdr^2 
        + r^{\frac{1}{\kappa}+1} Z_1 \tdp\tdr\right)\\
        &\quad - \left(\frac{\kappa^2}{\kappa + 1} r^{\frac{1}{\kappa}} a_1^2 (V_2)^\alpha \tdr\tdu_\alpha + \frac{\kappa^2}{\kappa+1} r^{\frac{1}{\kappa}} a_1^2(W_2)^\alpha_\gamma \tdu^\gamma\tdu_\alpha + \frac{\kappa^2}{\kappa+1} r^{\frac{1}{\kappa}} a_1^2 (Z_2)^\alpha \tdp\tdu_\alpha\right)\\
        &\quad - \left(\frac{(2\kappa+1)\kappa}{\kappa+1}r^\frac{1}{\kappa}\frac{a_1a_0^{2+\frac{2}{\kappa}}}{b(r)}V_3\tdr\tdp+ \frac{(2\kappa+1)\kappa}{\kappa + 1}r^\frac{1}{\kappa}\frac{a_1a_0^{2+\frac{2}{\kappa}}}{b(r)}(W_3)_\alpha \tdu^\alpha\tdp + \frac{(2\kappa+1)\kappa}{\kappa+1}r^\frac{1}{\kappa}\frac{a_1a_0^{2+\frac{2}{\kappa}}}{b(r)}Z_3\tdp^2\right).
    \end{align*}
    The terms underlined with a single line form a perfect derivative, modulo some terms that can be controlled in the energy estimate. This structure reflects the good weighted structure of the linearized equation discussed in Section \ref{SS:TheLinearizedEquation}, which allows us to control the non-lower-order term  $\partial_\mu r \tilde{u}^\mu$ using the special cancellation:
    \begin{equation}\label{SingleLineTerm}
        \begin{aligned}
            \kappa r^{\frac{1}{\kappa}} a_1 \partial_\mu \tilde{u}^\mu \tilde{r}+r^{\frac{1}{\kappa}-1} \partial_\mu r \tilde{u}^\mu \tilde{r}&+\kappa r^{\frac{1}{\kappa}} a_1 \partial_\mu \tilde{r} \tilde{u}^\mu=\kappa \partial_\mu\left(r^{\frac{1}{\kappa}} a_1 \tilde{r} \tilde{u}^\mu\right)\\
            &\quad -r^{\frac{1}{\kappa}}\left[(1+\kappa) \partial_\mu r+(2 \kappa+1) r a_0^{1+\frac{1}{\kappa}} \partial_\mu \pi\right] \tilde{r} \tilde{u}^\mu.
        \end{aligned}
    \end{equation}

    The terms underlined with a double line can be combined into a single term using the product rule. In the energy estimate, this term can be controlled simply by applying integration by parts, as it carries sufficient powers of $r$ to be controlled by the energy.
    \begin{equation}\label{DoubleLineTerm}
        \begin{aligned}
            \frac{(2\kappa+1)\kappa}{\kappa+1} r^{\frac{1}{\kappa}+1} a_0^{1+\frac{1}{\kappa}} a_1 \partial_\mu \tdp\tdu^\mu&+\frac{(2\kappa+1)\kappa}{\kappa+1}r^{1+\frac{1}{\kappa}}a_1a_0^{1+\frac{1}{\kappa}}\p_\mu\tdu^\mu\tdp\\
            &\quad =\frac{(2 \kappa+1) \kappa}{\kappa+1} r^{1+\frac{1}{\kappa}}a_0^{\frac{1}{\kappa}+1} a_1 \partial_\mu\left(\tilde{\pi} \tilde{u}^\mu\right)
        \end{aligned}
    \end{equation}
    We differentiate the energy \eqref{BasicEnergy} with the moving domain formula \eqref{MovingDomainFormula},
    \begin{equation*}
        \begin{aligned}
            \frac{dE}{dt}&=\frac{1}{2}\frac{d}{dt}\int_{\Omega_t}r^{\frac{1}{\kappa}-1}\left(\tdr^2+\frac{\kappa^2}{\kappa+1}ra_1^2\tdu_\alpha\tdu^\alpha + \frac{(2\kappa+1)\kappa}{\kappa+1}r\frac{a_1a_0^{2+\frac{2}{\kappa}}}{b(r)}\tdp^2\right)dx\\
            &=\frac{1}{2}\int_{\Omega_t}\frac{1}{u^0}\Dt\left(r^{\frac{1}{\kappa}-1}\left(\tdr^2+\frac{\kappa^2}{\kappa+1}ra_1^2\tdu_\alpha\tdu^\alpha + \frac{(2\kappa+1)\kappa}{\kappa+1}r\frac{a_1a_0^{2+\frac{2}{\kappa}}}{b(r)}\tdp^2\right)\right)dx\\
            &\quad + \frac{1}{2}\int_{\Omega_t}r^{\frac{1}{\kappa}-1}\left(\tdr^2+\frac{\kappa^2}{\kappa+1}ra_1^2\tdu_\alpha\tdu^\alpha + \frac{(2\kappa+1)\kappa}{\kappa+1}r\frac{a_1a_0^{2+\frac{2}{\kappa}}}{b(r)}\tdp^2\right)\p_i\left(\frac{u^i}{u^0}\right)dx\\
            &= (i) + (ii).
        \end{aligned}
    \end{equation*}
    The latter term can be controlled readily, i.e.,
    \[(ii)\lesssim \norm{\partial_i\left(\frac{u^i}{u^0}\right)}_{L^\infty(\Omega_t)}E(t)\lesssim \mathscr{C}(\norm{(r,u,\pi)}_{L^\infty_tW^{1,\infty}_x})E(t)\lesssim\scrC\norm{(\tdr,\tdu,\tdp)}_\mathcal{H}^2.\]

    Now, we control the former term,
    \begin{equation*}
        \begin{aligned}
            (i) &= \frac{1}{2}\int_{\Omega_t}\left(\frac{1}{\kappa}-1\right)\frac{1}{u^0}r^{\frac{1}{\kappa}-2}\Dt r\left(\tdr^2+\frac{\kappa^2}{\kappa+1}ra_1^2\tdu_\alpha\tdu^\alpha + \frac{(2\kappa+1)\kappa}{\kappa+1}r\frac{a_1a_0^{2+\frac{2}{\kappa}}}{b(r)}\tdp^2\right)dx\\
            &\quad + \frac{1}{2}\int_{\Omega_t}\frac{1}{u^0}r^{\frac{1}{\kappa}-1}\left(\Dt\tdr^2 + \frac{\kappa^2}{\kappa+1}ra_1^2\Dt(\tdu_\alpha\tdu^\alpha)+\frac{(2\kappa+1)\kappa}{\kappa+1}r\frac{a_1a_0^{2+\frac{2}{\kappa}}}{b(r)}\Dt\tdp^2\right)dx\\
            &\quad +\frac{1}{2}\int_{\Omega_t}\frac{\kappa^2}{\kappa+1}\frac{1}{u^0}r^{\frac{1}{\kappa}-1}\Dt(ra_1^2)\tdu_\alpha\tdu^\alpha dx\\
            &\quad + \frac{1}{2}\int_{\Omega_t}\frac{(2\kappa+1)\kappa}{\kappa+1}\frac{1}{u^0}r^{\frac{1}{\kappa}-1}\Dt\left(\frac{a_1a_0^{2+\frac{2}{\kappa}}}{b(r)}\right)\tdp^2dx\\
            &=(iii)+(iv)+(v)+(vi).
        \end{aligned}
    \end{equation*}
    We must be cautious when a derivative falls on the powers of $r$, as this could potentially produce terms with insufficient weight. However, this issue does not arise when the derivative is the material derivative $\Dt$ due to \eqref{rEquation}
    \begin{equation}
        \Dt r = -\kappa ra_1\p_\mu u^\mu,
    \end{equation}
    which can recover $1$ power of $r$. In the estimates, time derivatives $\partial_t (r, u, \pi)$ will appear as coefficients. To handle these coefficients, we use \rref{TimeInSpace} to express the time derivatives in terms of the spatial derivatives $\partial (r, u, \pi)$, allowing us to control them as follows:
    \begin{equation*}
        \norm{\p_t(r,u,\pi)}_{L^\infty_tL^\infty_x(\Omega_t)}\lesssim \scrC({\norm{(r,u,\pi)}_{L^\infty_tW^{1,\infty}_x(\Omega_t)}}).
    \end{equation*}

    By a direct estimate, we have
    \[(iii),(v),(vi)\lesssim \scrC({\norm{(r,u,\pi)}_{L^\infty_tW^{1,\infty}_x(\Omega_t)}})\norm{(\tilde{r},\tilde{u},\tilde{\pi})}_\mathcal{H}^2.\]

    For $(iv)$, we will use the identity \eqref{EnergyIdentity} to get
    \begin{equation*}
        \begin{aligned}
            (iv)&=\int_{\Omega_t}\frac{1}{u^0}r^{\frac{1}{\kappa}-1}\tdf\tdr + \frac{1}{u^0}\frac{\kappa^2}{\kappa+1}r^\frac{1}{\kappa}a_1^2 \tdg^\alpha\tdu_\alpha + \frac{1}{u^0}\frac{(2\kappa+1)\kappa}{\kappa+1}r^\frac{1}{\kappa}\frac{a_1a_0^{2+\frac{2}{\kappa}}}{b(r)}\tdh\tdp dx\\
            &\quad -\int_{\Omega_t}\frac{1}{u^0}\underline{(\kappa r^\frac{1}{\kappa}a_1\p_\mu\tdu^\mu\tdr + r^{\frac{1}{\kappa}-1}\p_\mu r\tdu^\mu \tdr + \kappa r^\frac{1}{\kappa}a_1\p_\mu\tdr\tdu^\mu)}dx\\
            &\quad - \int_{\Omega_t}\frac{1}{u^0}\underline{\underline{\left(\frac{(2\kappa+1)\kappa}{\kappa+1}r^{\frac{1}{\kappa}+1}a_0^{1+\frac{1}{\kappa}}a_1\p_\mu\tdp\tdu^\mu + \frac{(2\kappa+1)\kappa}{\kappa+1}r^{1+\frac{1}{\kappa}}a_1a_0^{1+\frac{1}{\kappa}}\p_\mu \tdu^\mu\tdp\right)}}dx\\
            &\quad -\int_{\Omega_t}\frac{1}{u^0}\left(r^{\frac{1}{\kappa}-1}V_1\tdr^2 + r^{\frac{1}{\kappa}+1}Z_1\tdp\tdr\right)dx\\
            &\quad -\int_{\Omega_t}\frac{1}{u^0}\left(\frac{\kappa^2}{\kappa + 1} r^{\frac{1}{\kappa}} a_1^2 (V_2)^\alpha \tdr\tdu_\alpha + \frac{\kappa^2}{\kappa+1} r^{\frac{1}{\kappa}} a_1^2(W_2)^\alpha_\gamma \tdu^\gamma\tdu_\alpha + \frac{\kappa^2}{\kappa+1} r^{\frac{1}{\kappa}} a_1^2 (Z_2)^\alpha \tdp\tdu_\alpha\right)dx\\
            &\quad - \int_{\Omega_t}\frac{1}{u^0}\biggl(\frac{(2\kappa+1)\kappa}{\kappa+1}r^\frac{1}{\kappa}\frac{a_1a_0^{2+\frac{2}{\kappa}}}{b(r)}V_3\tdr\tdp+ \frac{(2\kappa+1)\kappa}{\kappa + 1}r^\frac{1}{\kappa}\frac{a_1a_0^{2+\frac{2}{\kappa}}}{b(r)}(W_3)_\alpha \tdu^\alpha\tdp \\
            &\quad \quad  +\frac{(2\kappa+1)\kappa}{\kappa+1}r^\frac{1}{\kappa}\frac{a_1a_0^{2+\frac{2}{\kappa}}}{b(r)}Z_3\tdp^2\biggr)dx\\
            &=(vii)+(viii)+(ix)+(x).
        \end{aligned}
    \end{equation*}
Here $(x)$ is the sum of the last 3 terms in $(iv)$. By Cauchy--Schwarz inequality, we have
\begin{equation*}
    \begin{aligned}
        (vii)&=\int_{\Omega_t}\frac{1}{u^0}r^{\frac{1}{\kappa}-1}\tdf\tdr + \frac{1}{u^0}\frac{\kappa^2}{\kappa+1}r^\frac{1}{\kappa}a_1^2\tdg^\alpha\tdu_\alpha + \frac{1}{u^0}\frac{(2\kappa+1)\kappa}{\kappa+1}r^\frac{1}{\kappa}\frac{a_1a_0^{2+\frac{2}{\kappa}}}{b(r)}\tdh\tdp dx\\
        &\lesssim \norm{(\tdf,\tdg,\tdh)}_\mathcal{H}\norm{(\tdr,\tdu,\tdp)}_\mathcal{H}.
    \end{aligned}
\end{equation*}
Since $V,W,Z$ are coefficients linear in $\stp(r,u,\pi)$ multiplied by functions smooth in $(r,u,\pi)$, $(x)$ can be controlled by a direct estimate in a similar fashion as above,
\begin{equation*}
    (x)\lesssim\scrC({\norm{(r,u,\pi)}_{L^\infty_tW^{1,\infty}_x(\Omega_t)}})\norm{(\tdr,\tdu,\tdp)}_\mathcal{H}^2
\end{equation*}
For the single-underlined term, using \eqref{SingleLineTerm}, we have
\begin{equation*}
    \begin{aligned}
        (viii) &= -\int_{\Omega_t}\frac{1}{u^0}\underline{(\kappa r^\frac{1}{\kappa}a_1\p_\mu\tdu^\mu\tdr + r^{\frac{1}{\kappa}-1}\p_\mu r\tdu^\mu \tdr + \kappa r^\frac{1}{\kappa}a_1\p_\mu\tdr\tdu^\mu)}dx\\
        &= -\int_{\Omega_t}\frac{\kappa}{u^0}\p_\mu(r^\frac{1}{\kappa}a_1\tdr\tdu^\mu)-\frac{1}{u^0}r^\frac{1}{\kappa}\left((1+\kappa)\p_\mu r + (2\kappa+1)ra_0^{1+\frac{1}{\kappa}}\p_\mu\pi\right)\tdr\tdu^\mu dx\\
        &\lesssim -\int_{\Omega_t}\frac{\kappa}{u^0}\p_\mu(r^\frac{1}{\kappa}a_1\tdr\tdu^\mu)dx + \scrC({\norm{(r,u,\pi)}_{L^\infty_tW^{1,\infty}_x(\Omega_t)}})\norm{(\tdr,\tdu,\tdp)}_\mathcal{H}^2.
    \end{aligned}
\end{equation*}
Above, the last inequality holds because the second term in the second line carries sufficient weight. This term is bounded using Cauchy--Schwarz inequality.

For the double-underlined term, using \eqref{DoubleLineTerm}, we have
\begin{equation*}
    \begin{aligned}
        (ix) &= - \int_{\Omega_t}\frac{1}{u^0}\underline{\underline{\left(\frac{(2\kappa+1)\kappa}{\kappa+1}r^{\frac{1}{\kappa}+1}a_0^{1+\frac{1}{\kappa}}a_1\p_\mu\tdp\tdu^\mu + \frac{(2\kappa+1)\kappa}{\kappa+1}r^{1+\frac{1}{\kappa}}a_1a_0^{1+\frac{1}{\kappa}}\p_\mu \tdu^\mu\tdp\right)}}dx\\
        &= - \int_{\Omega_t}\frac{1}{u^0}\frac{(2 \kappa+1) \kappa}{\kappa+1} r^{1+\frac{1}{\kappa}}a_0^{\frac{1}{\kappa}+1} a_1 \partial_\mu\left(\tilde{\pi} \tilde{u}^\mu\right)dx.
    \end{aligned}
\end{equation*}

We combine all the inequalities we have above and find
\begin{equation*}
    \begin{aligned}
        \frac{dE}{dt}&\lesssim -\int_{\Omega_t}\frac{\kappa}{u^0}\p_\mu(r^\frac{1}{\kappa}a_1\tdr\tdu^\mu)dx - \int_{\Omega_t}\frac{1}{u^0}\frac{(2 \kappa+1) \kappa}{\kappa+1} r^{1+\frac{1}{\kappa}}a_0^{\frac{1}{\kappa}+1} a_1 \partial_\mu\left(\tilde{\pi} \tilde{u}^\mu\right)dx\\
        &\quad + \norm{(\tdf,\tdg,\tdh)}_\mathcal{H}\norm{(\tdr,\tdu,\tdp)}_\mathcal{H} +\scrC({\norm{(r,u,\pi)}_{L^\infty_tW^{1,\infty}_x(\Omega_t)}})\norm{(\tdr,\tdu,\tdp)}_\mathcal{H}^2
    \end{aligned}
\end{equation*}

Now, we take the integral with respect to time $t$,
\begin{align*}
    E(t)&\lesssim E(0) - \int_0^t\int_{\Omega_t}\frac{\kappa}{u^0}\p_\mu(r^\frac{1}{\kappa}a_1\tdr\tdu^\mu)dxds\\
    &\quad -\int_0^t\int_{\Omega_t}\frac{1}{u^0}\frac{(2 \kappa+1) \kappa}{\kappa+1} r^{1+\frac{1}{\kappa}}a_0^{\frac{1}{\kappa}+1} a_1 \partial_\mu\left(\tilde{\pi} \tilde{u}^\mu\right)dxds\\
    &\quad +\int_0^t \norm{(\tdf,\tdg,\tdh)}_\mathcal{H}\norm{(\tdr,\tdu,\tdp)}_\mathcal{H} +\scrC({\norm{(r,u,\pi)}_{L^\infty_tW^{1,\infty}_x(\Omega_t)}})\norm{(\tdr,\tdu,\tdp)}_\mathcal{H}^2ds\\
    &= E(0)+(xi)+(xii)+(xiii)
\end{align*}
For $(xi)$, we integrate by parts with respect to $\p_\mu$,
\begin{align*}
    - \int_0^t\int_{\Omega_t}\frac{\kappa}{u^0}\p_\mu(r^\frac{1}{\kappa}a_1\tdr\tdu^\mu)dxds&= \int^t_0\int_{\Omega_t}\p_\mu\left(\frac{1}{u^0}\right)\kappa(r^\frac{1}{\kappa}a_1\tdr\tdu)dxds\\
    &\quad -\int_{\Omega_t}\frac{1}{u^0}(r^\frac{1}{\kappa}a_1\tdr\tdu)dx\\
    &\quad -\int_{\Omega_t}\frac{1}{u^0}(r^\frac{1}{\kappa}a_1\tdr\tdu)dx\\
    &\lesssim \epsilon(\norm{(\tdr_0,\tdu_0,\tdp_0)}_\mathcal{H}^2+\norm{(\tdr,\tdu,\tdp)}_\mathcal{H}^2)\\
    &\quad + \int_{\Omega_t}\epsilon\scrC({\norm{(r,u,\pi)}_{L^\infty_tW^{1,\infty}_x(\Omega_t)}})\norm{(\tdr,\tdu,\tdp)}_\mathcal{H}^2dx.
\end{align*}

A similar estimate as above applies to $(xii)$, yielding
\begin{equation*}
    \begin{aligned}
        -\int_0^t\int_{\Omega_t}\frac{1}{u^0}\frac{(2 \kappa+1) \kappa}{\kappa+1} r^{1+\frac{1}{\kappa}}a_0^{\frac{1}{\kappa}+1} a_1 \partial_\mu\left(\tilde{\pi} \tilde{u}^\mu\right)dxds&=\epsilon(\norm{(\tdr_0,\tdu_0,\tdp_0)}_\mathcal{H}^2+\norm{(\tdr,\tdu,\tdp)}_\mathcal{H}^2)\\
    &\quad + \int_{\Omega_t}\scrC({\norm{(r,u,\pi)}_{L^\infty_tW^{1,\infty}_x(\Omega_t)}})\norm{(\tdr,\tdu,\tdp)}_\mathcal{H}^2dx.
    \end{aligned}
\end{equation*}

Since $E(t)\sim \norm{(\tdr(t),\tdu(t),\tdp(t))}_\mathcal{H}^2$, from above, we have
\begin{align*}
    \norm{(\tdr(t),\tdu(t),\tdp(t))}_\mathcal{H}^2&\lesssim \norm{(\tdr_0,\tdu_0,\tdp_0)}_\mathcal{H}^2+\epsilon\norm{(\tdr(t),\tdu(t),\tdp(t))}_\mathcal{H}^2+\int_0^t \norm{(\tdf,\tdg,\tdh)}_\mathcal{H}\norm{(\tdr,\tdu,\tdp)}_\mathcal{H} ds\\
    &\quad +\int^t_0\scrC({\norm{(r,u,\pi)}_{L^\infty_tW^{1,\infty}_x(\Omega_t)}})\norm{(\tdr,\tdu,\tdp)}_\mathcal{H}^2ds
\end{align*}

Absorbing $\epsilon\norm{(\tilde{r},\tilde{u},\tilde{\pi})}_\mathcal{H}^2$ to the LHS, we have
\begin{align*}
    \norm{(\tdr(t),\tdu(t),\tdp(t))}_\mathcal{H}^2&\lesssim \norm{(\tdr_0,\tdu_0,\tdp_0)}_\mathcal{H}^2+\int_0^t \norm{(\tdf,\tdg,\tdh)}_\mathcal{H}\norm{(\tdr,\tdu,\tdp)}_\mathcal{H} ds\\
    &\quad+\int^t_0\scrC({\norm{(r,u,\pi)}_{L^\infty_tW^{1,\infty}_x(\Omega_t)}})\norm{(\tdr,\tdu,\tdp)}_\mathcal{H}^2ds
\end{align*}

By the Gr\"{o}nwall-type inequality \eqref{GronwallTypeInequality} with $\alpha=\frac{1}{2}$, we have our basic energy estimate \eqref{BasicEnergyEstimate}.
\end{proof}

\blue{One has a reverse-time estimate that directly follows from the baisc energy estimate, similar to the classical FOSH theory.
\vspace{0.5em}
\begin{corollary}
    Under the same assumption as \tref{BasicEnergyEstimateTheorem}, we have the following estimate,
        \begin{equation}\label{BackwardBasicEnergyEstimate}
        \begin{aligned}
            \norm{(\tdr(t),\tdu(t),\tdp(t))}_\mathcal{H}&\lesssim \scrC(\norm{(r,u,\pi)}_{L^\infty_tW^{1,\infty}_x(\Omega_t)})\\
            &\quad \times \left(\norm{(\tdr(T^*),\tdu(T^*),\tdp(T^*))}_\mathcal{H}+\int^{T^*}_t\norm{(\tdf(s),\tdg(s),\tdh(s))}_\mathcal{H}\,ds\right)
        \end{aligned}
    \end{equation}
    for any $t\in[0,T^*]$.
\end{corollary}
\vspace{0.5em}}
\begin{remark}\label{TimeInSpace}
     In many parts of the estimates, we need to express the time derivatives $\partial_t (r, u, \pi)$ in terms of the spatial derivatives $\partial (r, u, \pi)$ using \eqref{The_nonlinear_system}. Let $\mathbb{L}(r, u, \pi)$ denote the operator associated with the nonlinear system \eqref{The_nonlinear_system}, which can be written as:
    \begin{equation*}
        \mathbb{L}(r,u,\pi) = A^0 \p_t (r,u,\pi)
        + A^i \p_i(r,u,\pi)
        + B(r,u,\pi),
    \end{equation*}
    where \( A^0 = A^0[r,u,\pi], A^i = A^i[r,u,\pi] \), and \( B = B[r,u,\pi] \) are given by
    \begin{equation*}
        \begin{aligned}
            A^0 &= \begin{pmatrix}
                u^0 & \kappa r a_1 & 0 & 0 & 0 & 0 \\
                \left(1 + \frac{1}{\kappa}\right) a_2 \Delta_0^0 & u^0 & 0 & 0 & 0 & \left(2 + \frac{1}{\kappa}\right) r a_3 \Delta_0^0 \\
                \left(1 + \frac{1}{\kappa}\right) a_2 \Delta_1^0 & 0 & u^0 & 0 & 0 & \left(2 + \frac{1}{\kappa}\right) r a_3 \Delta_1^0 \\
                \left(1 + \frac{1}{\kappa}\right) a_2 \Delta_2^0 & 0 & 0 & u^0 & 0 & \left(2 + \frac{1}{\kappa}\right) r a_3 \Delta_2^0 \\
                \left(1 + \frac{1}{\kappa}\right) a_2 \Delta_3^0 & 0 & 0 & 0 & u^0 & \left(2 + \frac{1}{\kappa}\right) r a_3 \Delta_3^0 \\
                0 & r a_4 & 0 & 0 & 0 & u^0
            \end{pmatrix}, \\
            A^i &= \begin{pmatrix}
                u^i & 0 & \kappa r a_1 \delta^i_1 & \kappa r a_1 \delta^i_2 & \kappa r a_1 \delta^i_3 & 0 \\
                \left(1 + \frac{1}{\kappa}\right) a_2 \Delta_0^i & u^i & 0 & 0 & 0 & \left(2 + \frac{1}{\kappa}\right) r a_3 \Delta_0^i \\
                \left(1 + \frac{1}{\kappa}\right) a_2 \Delta_1^i & 0 & u^i & 0 & 0 & \left(2 + \frac{1}{\kappa}\right) r a_3 \Delta_1^i \\
                \left(1 + \frac{1}{\kappa}\right) a_2 \Delta_2^i & 0 & 0 & u^i & 0 & \left(2 + \frac{1}{\kappa}\right) r a_3 \Delta_2^i \\
                \left(1 + \frac{1}{\kappa}\right) a_2 \Delta_3^i & 0 & 0 & 0 & u^i & \left(2 + \frac{1}{\kappa}\right) r a_3 \Delta_3^i \\
                0 & 0 & r a_4 \delta^i_1 & r a_4 \delta^i_2 & r a_4 \delta^i_3 & u^i
            \end{pmatrix}, \\
            B &= \begin{pmatrix}
                0 & 0 & 0 & 0 & 0 & 0 \\
                0 & 0 & 0 & 0 & 0 & 0 \\
                0 & 0 & 0 & 0 & 0 & 0 \\
                0 & 0 & 0 & 0 & 0 & 0 \\
                0 & 0 & 0 & 0 & 0 & 0 \\
                0 & 0 & 0 & 0 & 0 & 1 + \lambda(r^{\frac{1}{\kappa}}) \pi^{2 + \frac{1}{\kappa}}
            \end{pmatrix}.
        \end{aligned}
    \end{equation*}
    The determinant of $A^0$ is the following
    \begin{equation*}
        \det A^0 = (u^0)^6 - (2+\frac{1}{\kappa})r^2a_3a_4\left(-1+(u^0)^2\right)(u^0)^4-\kappa(1+\frac{1}{\kappa})r\left(-1+(u^0)^2\right)(u^0)^4.
    \end{equation*}
    When \( r \) is small, i.e., near the boundary, $\det A^0\neq 0$. Furthermore,
    \begin{equation*}
        \left( A^0 \right)^{-1} = \begin{pmatrix}
            \mathcal{O}(1) & r \mathcal{O}(1) & 0 & 0 & 0 & 0 \\
            \mathcal{O}(1) & \mathcal{O}(1) & 0 & 0 & 0 & r \mathcal{O}(1) \\
            \mathcal{O}(1) & r \mathcal{O}(1) & \mathcal{O}(1) & 0 & 0 & r \mathcal{O}(1) \\
            \mathcal{O}(1) & r \mathcal{O}(1) & 0 & \mathcal{O}(1) & 0 & r \mathcal{O}(1) \\
            \mathcal{O}(1) & r \mathcal{O}(1) & 0 & 0 & \mathcal{O}(1) & r \mathcal{O}(1) \\
            r \mathcal{O}(1) & r \mathcal{O}(1) & 0 & 0 & 0 & \mathcal{O}(1)
        \end{pmatrix},
    \end{equation*}
    where \( \mathcal{O}(1) \) represents some \( L^\infty \) coefficients in the energy estimate that are smooth functions in $(r,u,\pi)$. Omitting these $\mathcal{O}(1)$ coefficients,
    \begin{equation*}
        \begin{aligned}
            \p_t r &\simeq \p r + r \p u + r^2 \p \pi + r^2 \pi, \\
            \p_t u &\simeq \p r + \p u + r \p \pi + r \pi, \\
            \p_t \pi &\simeq r \p r + r \p u + \p \pi + \pi.
        \end{aligned}
    \end{equation*}
\end{remark}

\section{Local well-posedness of linearized equations and its high-order estimate}
In this section, we build on the basic energy estimate established in the previous section for the linearized equations to prove their local well-posedness. The basic energy estimate yields only a weak solution to the system. To construct a more regular solution, we introduce weighted elliptic operators derived from the second-order evolution equations of the main system. \blue{These equations are obtained by applying the material derivatives $\Dt$ to the linearized system \eqref{The_linearized_system}, commuting $\Dt$ with the system and then replacing $\Dt(\tdr,\tdu,\tdp)$ by $\stp(\tdr,\tdu,\tdp)$ using the system \eqref{The_linearized_system}. } The relevant properties of this operator, which are crucial for proving the existence of a regular solution, will also be demonstrated in this section.

\subsection{Bookkeeping Scheme}Here, we introduce a bookkeeping scheme to track the crucial terms and omit the less important terms during the energy estimate of the linearized equations. This allows us to greatly simplify the analysis. In view of the weighted version of the Sobolev embedding Lemma~\ref{WeightedEmbedding}, we make the following definition,
\vspace{0.5em}
\begin{definition}
    The order of terms of the form $r^a\p^b\tilde{r}$, $r^a\p^b\tilde{u}$ and $r^a\p^b\tilde{\pi}$ are defined to be $b-a$, $b-a+\frac{1}{2}$ and $b-a+\frac{1}{2}$, respectively. 
\end{definition}
\vspace{0.5em}

In a higher-order estimate of order $2l$, if a term $r^a\p^b\tdr$ has order $b-a = l$, by the embedding Lemma~\ref{WeightedEmbedding},
\begin{equation*}
    \norm{r^a\p^b\tilde{r}}_{H^{0,\frac{1}{2\kappa}-\frac{1}{2}}}\lesssim \norm{\tilde{r}}_{H^{2l,l+\frac{1}{2\kappa}-\frac{1}{2}}}.
\end{equation*}
Similarly, if a term $r^a\p^b\tdu$ has order $b - a + \frac{1}{2} = l$, then
\begin{align*}
    \norm{r^a\p^b\tilde{u}}_{H^{0,\frac{1}{2\kappa}-\frac{1}{2}}}&\lesssim \norm{\tilde{u}}_{H^{2l,l+\frac{1}{2\kappa}}},\\
    \norm{r^a\p^b\tilde{\pi}}_{H^{0,\frac{1}{2\kappa}-\frac{1}{2}}}&\lesssim \norm{\tilde{\pi}}_{H^{2l,l+\frac{1}{2\kappa}}}.
\end{align*}
We refer to such terms as exact terms, which can be controlled in a higher-order estimate of order $2l$. If $b-a<l$ for $\tilde{r}$, $b-a+\frac{1}{2}<l$ for $\tilde{u}$ and $\tilde{\pi}$, then these terms can also be controlled and will be referred to as lower-order terms. Otherwise, they are referred to as higher-order terms.

\vspace{0.5em}
\begin{notation}\label{Bookkeep}
    Let $A$ and $B$ be $2$ expressions involving terms $r^a\p^b\tdr$, $r^a\p^b\tdu$ and $r^a\p^b\tdp$. We will use the notation
    \begin{equation*}
        A\simeq B
    \end{equation*}
    if both $A$ and $B$ are of the same order and 
    \begin{equation*}
        A = B + \text{L.O.T.},
    \end{equation*}
    where $\text{L.O.T.}$ contains all the terms lower-order than $A$ and $B$. 
\end{notation}
\vspace{0.5em}
For each term in an expression, we will retain the powers of undifferentiated $r$ of the coefficients, as these powers of $r$ are essential for the terms to be controlled. We will omit the parts of the coefficients linear in derivatives $\p^{\geq 1} r,\p^{\geq1}u,\p^{\geq 1}\pi$ of $r,\pi$ of order higher than $1$ with coefficients that are smooth functions of $(r,u,\pi)$. The reason for this is twofold. First, these parts of the coefficients can be treated as $L^\infty$ coefficients in the estimate. Second, when derivatives are applied to terms with these coefficients, they will generate lower-order terms.

As an example of the bookkeeping scheme, from the first equation \eqref{trEquation} of the linearized system in the homogeneous case (i.e., $\tdf=0$):
\begin{equation}\label{Bookkeeping1}
    \Dt\tilde{r} = -\kappa r a_1 \p_\mu\tilde{u}^\mu - \p_\mu r\tilde{u}^\mu - V_1\tilde{r} - r^2Z_1\tilde{\pi} \simeq r\p\tilde{u} + \tilde{u}.
\end{equation}

\vspace{0.5em}
\begin{remark}\label{LinearizedTimeInSpace}
    Since our weighted Sobolev norms only involve spatial derivatives, in the estimate, we need to write all time derivatives in terms of the spatial derivatives. Linearizing the expressions in \rref{TimeInSpace}, we can show that
    \begin{equation*}
        \begin{aligned}
            \p_t\tilde{r} &\simeq \p\tilde{r} + \text{L.O.T.},\\
            \p_t\tilde{u} &\simeq \p\tilde{u} + \text{L.O.T.},\\
            \p_t\tilde{\pi} &\simeq \p\tilde{\pi} + \text{L.O.T.}.
        \end{aligned}
    \end{equation*} 
\end{remark}
\begin{remark}
    Similar to \eqref{Bookkeeping1}, we have
    \begin{equation*}
        \begin{aligned}
            \Dt^2\tilde{u} &\simeq \tdL_2\tdu\simeq r\p^2\tdu,\\
            \Dt\tilde{\pi} &\simeq r\p\tilde{u} + \tilde{u}.
        \end{aligned}
    \end{equation*}
    Counting the order of RHS, $\Dt\tilde{r},\Dt^2\tilde{u},\Dt\tilde{\pi}$ have an order of $\frac{1}{2},1,\frac{1}{2}$. So, we assign order $\frac{1}{2}$ to the material derivative $\Dt$.
\end{remark}

\subsection{Weighted Elliptic Operators}\label{SS:WeightedEllipticOperator}The linearized equations \eqref{The_linearized_system} in its homogeneous form $(\tdf,\tdg,\tdh)=0$ can be rewritten as a second-order evolution that has a wave-like structure. Applying $\Dt$ to the equation \eqref{trEquation} and using the equation \eqref{tuEquation}, with the bookkeeping scheme to capture the principal part, we have
\begin{equation}\label{WaveEvolution}
    \begin{aligned}
        \Dt^2\tilde{r} &\simeq \tilde{L}_1\tilde{r},\\
        \Dt^2\tilde{u}^\alpha &\simeq \left(\tilde{L}_2\tilde{u}\right)^\alpha,
    \end{aligned}
\end{equation}
where 
\begin{equation*}
    \begin{aligned}
        \tilde{L}_1 \tilde{r}&=(\kappa+1) H^{i j}\left(r \partial_i \partial_j \tilde{r}+\frac{1}{\kappa} \partial_i r \partial_j \tilde{r}\right),\\
        \left(\tilde{L}_2 \tilde{u}\right)_\alpha&=(\kappa+1) A_\alpha^i H^{j k}\left(\partial_i\left(r \partial_j \tilde{u}_k\right)+\frac{1}{\kappa} \partial_j r \partial_i \tilde{u}_k\right),
    \end{aligned}
\end{equation*}
with
\begin{equation*}
    \begin{aligned}
        A_\alpha^i&=g_\alpha^i-\frac{u^i}{u^0} g_\alpha^0,\\
        H^{i j}&=\delta^{i j}-\frac{u^i u^j}{\left(u^0\right)^2}.
    \end{aligned}
\end{equation*}

We say the second-order evolution \eqref{WaveEvolution} has a wave-like structure because later we will show $\tilde{L}_1$ and $\tilde{L}_2$ satisfy elliptic estimates. More precisely, the operator $\tilde{L}_2$ is only associated to the divergence of $\tilde{u}$. To gain the full coercivity, we need to pair it with an operator $\tilde{L}_3$ associated with the curl of $\tilde{u}$, which is defined as follows,
\begin{equation*}
    \left(\tilde{L}_3 \tilde{u}\right)_\alpha=(\kappa+1) A_\alpha^i H^{i j} r^{-\frac{1}{\kappa}} \partial_j\left(r^{\frac{1}{\kappa}+1}\left(\partial_k \tilde{u}_i-\partial_i \tilde{u}_k\right)\right).
\end{equation*}

We can check that $\tilde{L}_2\tilde{L}_3\simeq\tilde{L}_3\tilde{L}_2\simeq0$. \blue{Here, we emphasize that the operators $\tdL_2$ and $\tdL_3$ act on linearized $4$-velocity $\tdu$ by taking the divergence and curl of the spatial components $\vec{\tdu}$ of $\tdu$. Additionally, these operators map $\tdu$ to a vector field that is orthogonal to $u$. Therefore, once we establish the elliptic control of the spatial part of $\tdu$, we can use the orthogonal condition \eqref{EulerianOrthogonality} to control $\tdu^0$. (see the proof of \lref{EllipticEstimates}.)}

Although we will not use these weighted elliptic operators directly in connection to the wave equations, they play an important role in the high-order energy estimates and in constructing more regular solutions of the linearized equation. We sometimes employ slightly modified versions $\hat{L}_1,\hat{L}_2 + \hat{L}_3$ of these operators $\tilde{L}_1,\tilde{L}_2+\tilde{L}_3$ for a different purpose, but they share the same principal parts,
\begin{equation}\label{ModifiedEllipitcOperators}
    \begin{aligned}
        \hatL_1 \tdr &= (\kappa + 1)r^{1-\frac{1}{\kappa}}\p_i\left(r^\frac{1}{\kappa}H^{ij}\p_j\tdr\right),\\
        (\hatL_2 \tdu)_\alpha &= (\kappa+1)A_\alpha^i\p_i(r^{1-\frac{1}{\kappa}} \p_j (r^\frac{1}{\kappa} H^{ij} \tdu_k)),\\
        (\hatL_3 \tdu)_\alpha &= (\kappa + 1)A^l_\alpha\Delta_{ml}r^{-\frac{1}{\kappa}}\p_j(r^{\frac{1}{\kappa}+1}H^{jk}H^{im}(\p_k\tdu_i-\p_i\tdu_k)).
    \end{aligned}
\end{equation}

We see the operator $\hat{L}_1$ as an unbounded operator of $H^{0,\frac{1}{2\kappa}-\frac{1}{2}}$ satisfying the following property,
\vspace{0.5em}
\begin{lemma}\label{SelfAdjoint1}
    Let $(r,u,\pi)$ be a sufficiently smooth solution to \eqref{The_nonlinear_system} satisfying the physical vacuum boundary condition \eqref{Physical_vacuum_boundary_condition}. Then, the operator \blue{$-\hat{L}_1$ }defined as an unbounded operator of the Hilbert space $H^{0,\frac{1}{2\kappa}-\frac{1}{2}}$ with the domain
    \begin{equation*}
        \mathcal{D}(\hat{L}_1):=\{f\in H^{0,\frac{1}{2\kappa}-\frac{1}{2}}\vert \hat{L}_1 f\in H^{0,\frac{1}{2\kappa}-\frac{1}{2}}\,\text{in the distributional sense}\}
    \end{equation*}
    is a non-negative, self-adjoint operator.
\end{lemma}
\begin{proof}
    Similar statements and proofs can be found in  \cite{ifrim2023compressible,disconzi2022relativistic}. We first show that $\mathcal{D}(\hat{L}_1^*)=\mathcal{D}(\hat{L}_1)$. It is sufficient to show $\mathcal{D}(\hat{L}_1^*)\subset\mathcal{D}(\hat{L}_1)$. Let $\tilde{s}\in \mathcal{D}(\hat{L}_1^*)$, then
    \begin{equation*}
        (\hatL_1\tdr,\tilde{s})_{H^{0,\frac{1}{2\kappa}-\frac{1}{2}}} = (\tdr,\hatL_1^*\tilde{s})_{H^{0,\frac{1}{2\kappa}-\frac{1}{2}}},\forall\tdr\in\mathcal{D}(\hat{L}_1).
    \end{equation*}
    We want to show $\tilde{s}\in\mathcal{D}(\hat{L}_1)$, i.e., $\hatL_1\tilde{s}\in H^{0,\frac{1}{2\kappa}-\frac{1}{2}}$. For any $\phi\in C^\infty_c$,
    \begin{align*}
        (\hatL_1\tilde{s})(\phi)&=\int \tilde{s}(\kappa+1)\p_j(r^\frac{1}{\kappa}H^{ij}\p_i(r^{\frac{1}{\kappa}-1}\phi)) \,\mathrm{d}x\\
        &=(\tilde{s},\hatL_1(r^{\frac{1}{\kappa}-1}\phi))_{H^{0,\frac{1}{2\kappa}-\frac{1}{2}}}\\
        &=(\hatL_1^*\tilde{s},r^{\frac{1}{\kappa}-1}\phi)_{H^{0,\frac{1}{2\kappa}-\frac{1}{2}}}\\
        &=\int \hatL_1^*\tilde{s}\phi \,\mathrm{d}x.
    \end{align*}
    Since $\hatL_1^*\tilde{s}\in H^{0,\frac{1}{2\kappa}-\frac{1}{2}}$, $\hatL_1\tilde{s}\in H^{0,\frac{1}{2\kappa}-\frac{1}{2}}$ in the distributional sense and thus $\tilde{s}\in\mathcal{D}(\hat{L}_1)$. In view of \cref{DomainOfWeightedEllipticOperators}, the non-negativity and self-adjointness of $-\hatL_1$ can be proven by a simple application of integration by parts.
\end{proof}
\vspace{0.5em}

With a similar argument as the above lemma, we obtain
\vspace{0.5em}
\begin{lemma}\label{SelfAdjoint2}
    Let $(r,u,\pi)$ be a sufficiently smooth solution to \eqref{The_nonlinear_system} satisfying the physical vacuum boundary condition \eqref{Physical_vacuum_boundary_condition}. Then, the operator \blue{$-(\hat{L}_2+\hat{L}_3)$} defined as an unbounded operator of the Hilbert space $H^{0,\frac{1}{2\kappa}}_\perp$ with the domain
    \begin{equation*}
        \mathcal{D}(\hat{L}_2+\hat{L}_3):=\{f\in H^{0,\frac{1}{2\kappa}}_\perp \vert (\hat{L}_2+\hat{L}_3) f\in H^{0,\frac{1}{2\kappa}}_\perp\,\text{in the distributional sense}\}
    \end{equation*}
    is a non-negative, self-adjoint operator.
\end{lemma}
\vspace{0.5em}

\blue{\begin{remark}
    One can verify $(\hatL_2+\hatL_3)\tdu\in H^{0,\frac{1}{2\kappa}}_\perp$ if $\tdu\in \mathcal{D}(\hatL_2+\hatL_3)$. Here, $H^{0,\frac{1}{2\kappa}}_\perp$ should be regarded as a Hilbert subspace inheriting the inner product from $H^{0,\frac{1}{2\kappa}}$. To ensure this, we need to show that $H^{0,\frac{1}{2\kappa}}_\perp$ is closed in $H^{0,\frac{1}{2\kappa}}$, which follows from a straightforward argument.
\end{remark}}

\subsection{Elliptic Estimates}Here, we show the elliptic estimate of the operator $\tilde{L}_1$ and $\tilde{L}_2+\tilde{L}_3$

\vspace{0.5em}
\begin{lemma}\label{EllipticEstimates}
    Let $(r,u,\pi)$ be a sufficiently smooth solution to \eqref{The_nonlinear_system} satisfying the physical vacuum boundary condition \eqref{Physical_vacuum_boundary_condition}. For $\tilde{r}$ sufficiently smooth, we have
    \begin{equation}\label{BasicREllipitc}
        \norm{\tilde{r}}_{H^{2,\frac{1}{2\kappa}-\frac{1}{2}}}\lesssim \norm{\tilde{L}_1\tilde{r}}_{H^{0,\frac{1}{2\kappa}-\frac{1}{2}}} + \norm{\tilde{r}}_{H^{0,\frac{1}{2\kappa}-\frac{1}{2}}}.
    \end{equation}
    For a sufficiently smooth $\tilde{u}$ orthogonal to $u$, we have
    \begin{equation}\label{BasicUElliptic}
        \norm{\tilde{u}}_{H^{2,\frac{1}{2\kappa}}}\lesssim \norm{\left(\tilde{L}_2+\tilde{L}_3\right)\tilde{u}}_{H^{0,\frac{1}{2\kappa}}} + \norm{\tilde{u}}_{H^{0,\frac{1}{2\kappa}}}.
    \end{equation}
\end{lemma}

\begin{proof}
     By a partition of unity, we can localize the estimates to a small ball. We consider the interesting case where the ball is centered at $x_0$ on the free boundary. The elliptic estimate is standard otherwise. We pick a coordinate \blue{$e_3=\frac{(\nabla r)(x_0)}{\abs{(\nabla r)(x_0)}}$ and $e_1,e_2$} an orthonormal basis of the tangent space of the free boundary at $x_0$. We have 
    \begin{equation*}
        \abs{\p_3 r-1},\abs{\p_{i'}r} < \epsilon,
    \end{equation*}
    where $\epsilon>0$ is sufficiently small and $i'=1,2$. We assume the ball is small enough such that $H^{ij}$ is a constant modulo perturbative error that can be absorbed onto the other side. So, we simply assume $H^{ij}$ is a constant in the estimates.

    To prove \eqref{BasicREllipitc}, we first prove
    \begin{equation}\label{REllipticPre1}
        \norm{\tdr}_{H^{2,\frac{1}{2\kappa}+\frac{1}{2}}}\lesssim\norm{\tdL_1\tdr}_{H^{0,\frac{1}{2\kappa}-\frac{1}{2}}} + \norm{\tdr}_{H^{1,\frac{1}{2\kappa}-\frac{1}{2}}}.
    \end{equation}

    Compute $\norm{\tdL_1\tdr}_{H^{0,\frac{1}{2\kappa}-\frac{1}{2}}}^2$ as follows:
    \begin{equation*}
        \begin{aligned}
            \norm{\tdL_1 \tdr}_{H^{0, \frac{1}{2\kappa} - \frac{1}{2}}}^2 &= \int_{\Omega_t} (\kappa+1)^2 H^{ij} H^{kl} \left(r \p_i \p_j \tdr + \frac{1}{\kappa} \p_i r \p_j \tdr\right) \left(r \p_k \p_l \tdr + \frac{1}{\kappa} \p_k r \p_l \tdr\right)dx \\
            &\simeq (i) + (ii) + (iii) + (iv).
        \end{aligned}
    \end{equation*}
    where $(i)$ represents the product of second-order derivative terms, $(ii)$ and $(iii)$ represent to the products of first-order and second-order derivative terms, and $(iv)$ represents the product of first-order derivatives terms.

    For $(i)$, we integrate by parts with respect to $\p_i$,
    \begin{equation*}
        \begin{aligned}
            (i) &= (\kappa+1)^2\int_{\Omega_t}r^{\frac{1}{\kappa}+1}H^{ij}H^{kl}\p_i\p_j\tdr\p_k\p_l\tdr dx\\
            &= -(\kappa+1)^2\int_{\Omega_t}\left(\frac{1}{\kappa}+1\right)r^\frac{1}{\kappa}\p_i rH^{ij}H^{kl}\p_j\tdr\p_k\p_l\tdr dx \\
            &\quad -(\kappa+1)^2\int_{\Omega_t}r^{\frac{1}{\kappa}+1}H^{ij}H^{kl}\p_j\tdr\p_i\p_k\p_l\tdr dx\\
            &= (v) + (vi).
        \end{aligned}
    \end{equation*}

    By Cauchy--Schwarz inequality with $\epsilon>0$, we have
    \begin{equation*}
        \begin{aligned}
            \abs{(v)}&\lesssim\int_{\Omega_t}\abs{H^{ij}H^{kl}\p_i r}\left(r^{\frac{1}{2\kappa}-\frac{1}{2}}\abs{\p_j\tdr}\right)\left(r^{\frac{1}{2\kappa}+\frac{1}{2}}\abs{\p_k\p_l\tdr}\right) dx\\
            &\lesssim \norm{\tdr}^2_{H^{1,\frac{1}{2\kappa}-\frac{1}{2}}} + \epsilon\norm{\tdr}^2_{H^{2,\frac{1}{2\kappa}+\frac{1}{2}}}.
        \end{aligned}
    \end{equation*}

    For $(vi)$, we integrate by parts with respect to $\p_k$,
    \begin{equation*}
        \begin{aligned}
            (vi) &=(\kappa+1)^2\int_{\Omega_t}\left(\frac{1}{\kappa}+1\right)r^\frac{1}{\kappa} \p_k rH^{ij}H^{kl}\p_j \tdr \p_i\p_l\tdr dx \\
            &\quad + (\kappa+1)^2\int_{\Omega_t}r^{\frac{1}{\kappa}+1}H^{ij}H^{kl}\p_j\p_k\tdr\p_i\p_l\tdr \,\mathrm{d}x\\
            &= (vii) + (viii).
        \end{aligned}
    \end{equation*}

     We notice that $(ii),(iii)$ and $(vii)$ can be estimated similarly to $(v)$,
    \begin{equation*}
        \abs{(ii)},\abs{(iii)},\abs{(vii)} \lesssim \norm{\tdr}^2_{H^{1,\frac{1}{2\kappa}-\frac{1}{2}}} + \epsilon \norm{\tdr}^2_{H^{2,\frac{1}{2\kappa}+\frac{1}{2}}}.
    \end{equation*}

    By the positive definiteness of the metric $H$, we find
    \begin{equation*}
        (viii) \gtrsim \norm{r^{\frac{1}{2\kappa}+\frac{1}{2}}\p^2\tdr}_{L^2}^2.
    \end{equation*}

    For $(iv)$, we find
    \begin{equation*}
        \abs{(iv)}\lesssim\norm{\tdr}_{H^{1,\frac{1}{2\kappa}-\frac{1}{2}}}^2.
    \end{equation*}

    Combining all the above inequalities, we have
    \begin{equation*}
        \norm{\tdr}_{H^{2,\frac{1}{2\kappa}+\frac{1}{2}}}\lesssim\norm{\tdL_1\tdr}_{H^{0,\frac{1}{2\kappa}-\frac{1}{2}}} + \norm{\tdr}_{H^{1,\frac{1}{2\kappa}-\frac{1}{2}}}.
    \end{equation*}

    Now, we prove 
    \begin{equation}\label{REllipticPre2}
        \norm{\tdr}_{H^{1,\frac{1}{2\kappa}-\frac{1}{2}}}\lesssim\norm{\tdL_1\tdr}_{H^{0,\frac{1}{2\kappa}-\frac{1}{2}}} + \norm{\tdr}_{H^{0,\frac{1}{2\kappa}-\frac{1}{2}}}.
    \end{equation}

    Consider the following integral and integrate by parts with respect to $\p_i$,
    \begin{equation*}
        \begin{aligned}
            \int_{\Omega_t} r^{\frac{1}{\kappa}-1} \tdL_1 \tdr \, \p_n \tdr &= (\kappa+1)\int_{\Omega_t} r^{\frac{1}{\kappa}-1} \p_n \tdr \, H^{ij} \left(r \p_i \p_j \tdr + \frac{1}{\kappa} \p_i r \p_j \tdr \right) dx \\
            &= -(\kappa + 1) \frac{1}{\kappa} \int_{\Omega_t} r^{\frac{1}{\kappa}-1} \p_i r \, H^{ij} \, \p_n \tdr \, \p_j \tdr dx
            \\
            &\quad- \frac{\kappa+1}{2} \int_{\Omega_t} r^{\frac{1}{\kappa}} \p_n \left(H^{ij} \p_i \tdr \, \p_j \tdr \right) dx \\
            &\quad + (\kappa + 1) \frac{1}{\kappa} \int_{\Omega_t} r^{\frac{1}{\kappa}-1} \p_i r \, H^{ij} \, \p_n \tdr \, \p_j \tdr dx\\
            &= \frac{\kappa + 1}{2 \kappa} \int_{\Omega_t} r^{\frac{1}{\kappa}-1} \p_n r H^{ij} \, \p_i \tdr \, \p_j \tdr dx\\
            &\gtrsim \norm{r^{\frac{1}{2\kappa} - \frac{1}{2}} \p \tdr}^2_{L^2}.
        \end{aligned}
    \end{equation*}

    By applying the Cauchy--Schwarz inequality with $\epsilon$ to the LHS, we have \eqref{REllipticPre2}.

    For the basic elliptic estimate \eqref{BasicUElliptic}, we again first prove the inequality
    \begin{equation*}
        \norm{\vec{\tdu}}_{H^{2,\frac{1}{2\kappa}+1}}\lesssim \norm{\left(\tilde{L}_2+\tilde{L}_3\right)\tilde{u}}_{H^{0,\frac{1}{2\kappa}}} + \norm{\vec{\tdu}}_{H^{1,\frac{1}{2\kappa}}},
    \end{equation*}
    where we use $\vec{\tdu}$ to represent the spatial part of the full linearized $4$-velocity $\tdu$.

    This can be proved similarly to the inequality \eqref{REllipticPre1}. Then, we prove the following with a similar semi-geodesic coordinate change in \cite{disconzi2022relativistic} to reduce the estimate into the Euclidean case in \cite{ifrim2023compressible},
    \begin{equation*}
        \norm{\vec{\tdu}}_{H^{1,\frac{1}{2\kappa}}}\lesssim \norm{\left(\tilde{L}_2+\tilde{L}_3\right)\tdu}_{H^{0,\frac{1}{2\kappa}}} + \norm{\vec{\tdu}}_{H^{0,\frac{1}{2\kappa}}}.
    \end{equation*}

    Combining the 2 inequalities, we have
    \begin{equation*}
        \norm{\vec{\tilde{u}}}_{H^{2,\frac{1}{2\kappa}+1}}\lesssim \norm{\left(\tilde{L}_2+\tilde{L}_3\right)\tilde{u}}_{H^{0,\frac{1}{2\kappa}}} + \norm{\vec{\tdu}}_{H^{0,\frac{1}{2\kappa}}}.
    \end{equation*}

    To bound the full linearized $4$-velocity, we use the orthogonality to write $\tdu^0 = \frac{u^i}{u^0}\tdu_i$, so
    \begin{equation*}
        \norm{\tilde{u}^0}_{H^{2,\frac{1}{2\kappa}+1}} = \norm{\frac{u^i}{u^0}\tdu_i}_{H^{2,\frac{1}{2\kappa}+1}}\lesssim \norm{\left(\tilde{L}_2+\tilde{L}_3\right)\tilde{u}}_{H^{0,\frac{1}{2\kappa}}} + \norm{\tilde{u}}_{H^{0,\frac{1}{2\kappa}}},
    \end{equation*}
    which finishes the proof.
\end{proof}

\vspace{0.5em}
\blue{\begin{remark}\label{EllipticImplicitConstant}
    In the above elliptic estimates, the implicit constants are in fact depending on $\norm{(r,u,\pi)}_{L^\infty_t W^{1,\infty}_x}$, i.e., the elliptic estimates can be written as 
    \begin{equation*}
        \begin{aligned}
            \norm{\tilde{r}}_{H^{2,\frac{1}{2\kappa}-\frac{1}{2}}}&\lesssim \norm{\tilde{L}_1\tilde{r}}_{H^{0,\frac{1}{2\kappa}-\frac{1}{2}}} + \scrC(\norm{(r,u,\pi)}_{L^\infty_t W^{1,\infty}_x})\norm{\tilde{r}}_{H^{0,\frac{1}{2\kappa}-\frac{1}{2}}},\\
            \norm{\tilde{u}}_{H^{2,\frac{1}{2\kappa}}}&\lesssim  \norm{\left(\tilde{L}_2+\tilde{L}_3\right)\tilde{u}}_{H^{0,\frac{1}{2\kappa}}} + \scrC(\norm{(r,u,\pi)}_{L^\infty_t W^{1,\infty}_x})\norm{\tilde{u}}_{H^{0,\frac{1}{2\kappa}}}.
        \end{aligned}
    \end{equation*}
\end{remark}}
\vspace{0.5em}
\begin{remark}\label{LowerOrderEllipticShift}
    The operator $\hat{L}_1$ and $\hat{L}_2+\hat{L}_3$ satisfy a similar basic elliptic estimate as in \eqref{BasicREllipitc} and \eqref{BasicUElliptic} because they differ from $\tilde{L}_1$ and $\tilde{L}_2 + \tilde{L}_3$ only by lower-order terms. 
\end{remark}

\vspace{0.5em}
\blue{\begin{corollary}\label{InvertibleEllipticEstimate}
   Let $(r,u,\pi)$ be a sufficiently smooth solution to \eqref{The_nonlinear_system} satisfying the physical vacuum boundary condition \eqref{Physical_vacuum_boundary_condition}. There exists a sufficiently large $\lambda>0$ such that for $\tilde{r}$ sufficiently smooth, we have
   \begin{equation*}
       \norm{\tilde{r}}_{H^{2,\frac{1}{2\kappa}-\frac{1}{2}}}\lesssim \norm{(\hatL_1-\lambda I) \tilde{r}}_{H^{0,\frac{1}{2\kappa}-\frac{1}{2}}},
   \end{equation*}
   and for a sufficiently smooth $\tilde{u}$ orthogonal to $u$, we have
   \begin{equation*}
       \norm{\tilde{u}}_{H^{2,\frac{1}{2\kappa}}}\lesssim \norm{\left(\hatL_2 + \hatL_3-\lambda I\right)\tilde{u}}_{H^{0,\frac{1}{2\kappa}}}.
   \end{equation*}
\end{corollary}
}
\begin{proof}
    By the elliptic estimate \eqref{BasicREllipitc} and the non-negativity of the operator $\hatL_1$ and $\hatL_2+\hatL_3$ in \lref{SelfAdjoint1} and \lref{SelfAdjoint2} (See also \rref{LowerOrderEllipticShift}), we have
    \begin{equation*}
        \begin{aligned}
            \left((\hat{L}_1 - \lambda I)\tdr,(\hat{L}_1 -\lambda I)\tdr\right)_{H^{0,\frac{1}{2\kappa}-\frac{1}{2}}} &= \norm{\hat{L}_1\tilde{r}}_{H^{0,\frac{1}{2\kappa}-\frac{1}{2}}}^2 - 2\lambda\left(\hat{L}_1\tdr, \tdr\right)_{H^{0,\frac{1}{2\kappa}-\frac{1}{2}}} + \lambda^2\norm{\tilde{r}}_{H^{0,\frac{1}{2\kappa}-\frac{1}{2}}}^2\\
            &\gtrsim \norm{\hat{L}_1\tilde{r}}_{H^{0,\frac{1}{2\kappa}-\frac{1}{2}}}^2 + (\lambda^2 - \scrC)\norm{\tilde{r}}_{H^{0,\frac{1}{2\kappa}-\frac{1}{2}}}^2\\
            &\gtrsim \norm{\tdr}_{H^{2,\frac{1}{2\kappa}+\frac{1}{2}}}^2,  
        \end{aligned}
    \end{equation*}
    if $\lambda>0$ is sufficiently large. A similar proof works for the velocity case.
\end{proof}

In the following corollary, we can identify the domain of $\hat{L}_1$ and $\hat{L}_2+\hat{L}_3$:
\begin{corollary}\label{DomainOfWeightedEllipticOperators}
    \begin{equation*}
    \begin{aligned}
        \mathcal{D}(\hat{L}_1) &= H^{2, \frac{1}{2\kappa}+\frac{1}{2}}\\
        \mathcal{D}(\hat{L}_2+\hat{L}_3) &= H^{2,\frac{1}{2\kappa}+1}_\perp,
    \end{aligned}
\end{equation*}
\end{corollary}

\begin{proof}
    To show: $\mathcal{D}(\hat{L}_1) = H^{2, \frac{1}{2\kappa}+\frac{1}{2}}$, it is sufficient to show $\mathcal{D}(\hat{L}_1) \subset H^{2, \frac{1}{2\kappa}+\frac{1}{2}}$. Take $\phi\in\mathcal{D}(\hatL_1)\subset L^2(r^{\frac{1}{\kappa}-1})$.

    Let $\sigma = \frac{1}{\kappa}-1$. Then, $r^\frac{\sigma}{2}\hatL_1\phi\in L^2$, we take $\phi_n\in C^\infty_c(\Omega)$ such that $\phi_n\to r^\frac{\sigma}{2}\hatL_1\phi$ in $L^2$. We define $\psi_n\in C^\infty_c(\Omega)$ by
    \begin{equation*}
        \left\{ \begin{aligned}
        & \hatL_1\psi_n = \frac{\phi_n}{r^\frac{\sigma}{2} }\\
        & \psi_n|_{\partial \mathrm{supp}(\phi_n)} = 0
        \end{aligned}\right.
    \end{equation*}
    By the elliptic estimates, we have
    \begin{equation*}
        \norm{\psi_n-\psi_m}_{H^{2,\frac{\sigma}{2}+1}}\lesssim \norm{\hatL_1(\psi_n-\psi_m)}_{H^{0,\frac{\sigma}{2}}} = \norm{\frac{\phi_n-\phi_m}{r^\frac{\sigma}{2}}}_{H^{0,\frac{\sigma}{2}}}=\norm{\phi_n-\phi_m}_{L^2}\to 0
    \end{equation*}
    as $n,m\to\infty$. Thus, $\psi_n\to\psi\in H^{2,\frac{\sigma}{2}+1}$. So, 
    \begin{equation*}
        r^\frac{\sigma}{2}\hatL_1\psi = \lim_{n\to\infty}r^\frac{\sigma}{2}\hatL_1\psi_n = \lim_n \phi_n = r^\frac{\sigma}{2}\hatL_1\phi,
    \end{equation*}
    which implies $\hatL_1(\psi-\phi) = 0$, where $\psi-\phi\in H^{0,\frac{\sigma}{2}}$. Consider this under the a subdomain $\Omega'\subset\Omega$, where $r$ is strictly positive on $\p\Omega'$. Since $0$ is $C^\infty$, $\psi-\phi$ is $C^\infty$ on $\Omega'$ by the standard elliptic theory. By the weighted elliptic estimate,
    \begin{equation*}
        \norm{\psi-\phi}_{H^{2,\frac{\sigma}{2}+1}}\lesssim\norm{\hatL_1(\psi-\phi)}_{H^{0,\frac{\sigma}{2}}}=0.
    \end{equation*}
    So, $\phi = \psi\in H^{2,\frac{1}{2\kappa}+\frac{1}{2}}$. $\mathcal{D}(\hat{L}_2+\hat{L}_3) = H^{2,\frac{1}{2\kappa}+1}$ can be proven similarly.
\end{proof}

We have the following high-order elliptic estimates.
\vspace{0.5em}
\begin{lemma}\label{HighOrderEllipticEstimates}
        Let $(r,u,\pi)$ be a sufficiently smooth solution to \eqref{The_nonlinear_system} satisfying the physical vacuum boundary condition \eqref{Physical_vacuum_boundary_condition}. For $\tilde{r}$ sufficiently smooth, we have
    \begin{equation*}
        \norm{\tilde{r}}_{H^{2l,\frac{1}{2\kappa}+l-\frac{1}{2}}}\lesssim \norm{\tilde{L}_1\tilde{r}}_{H^{2l-2,\frac{1}{2\kappa}+l-\frac{3}{2}}} + \norm{\tilde{r}}_{H^{0,\frac{1}{2\kappa}}}
    \end{equation*}
    For a sufficiently smooth $\tilde{u}$ orthogonal to $u$, we have
    \begin{equation*}
        \norm{\tilde{u}}_{H^{2l,l+\frac{1}{2\kappa}}}\lesssim \norm{\left(\tilde{L}_2+\tilde{L}_3\right)\tilde{u}}_{H^{2l-2,l-1+\frac{1}{2\kappa}}} + \norm{\tilde{u}}_{H^{0,\frac{1}{2\kappa}}}.
    \end{equation*}
\end{lemma}

\begin{proof}
    We will prove the case for \( \tdr \) only, as the proof for \( \tdu \) follows similarly. We pick a coordinate $e_3=\frac{(\nabla r)(x_0)}{\abs{(\nabla r)(x_0)}}$ and $e_1,e_2$ an orthonormal basis of the tangent space of the free boundary at $x_0$. We have 
    \begin{equation*}
        \abs{\p_3r -1},\abs{\p_{i'}r} < \epsilon,
    \end{equation*}
    where $\epsilon>0$ is sufficiently small and $i'=1,2$. Let $\sigma = \frac{1}{2\kappa}-\frac{1}{2}$. It is sufficient to show that
    $$\norm{\tilde{r}}_{H^{2l,\frac{1}{2\kappa}+l-\frac{1}{2}}}\lesssim \norm{\tilde{L}_1\tilde{r}}_{H^{2l-2,\frac{1}{2\kappa}+l-\frac{3}{2}}} + \norm{\tilde{r}}_{H^{2l-2,l-1+\frac{1}{2\kappa}}}.$$
    To show this, we show that for any $l\in\mathbb{N}$ and any integer $0\leq a\leq l$,
    \begin{equation}\label{HigherOrderEllipticInduction}
        \norm{r^a\p^{l+a}\tdr}_{H^{0,\sigma}}\lesssim \norm{\tdL_1\tdr}_{H^{2l-2,l-1+\sigma}}+\epsilon\norm{\tdr}_{H^{2l,l+\sigma}}+\norm{\tdr}_{H^{2l-2,l-1+\sigma}}.
    \end{equation}

    We prove this statement by induction on $a$. Commuting $r^a\p^{l-1+a}$ and $\tdL_1$, we obtain $$r^a\p^{l-1+a}\tdL_1\tdr=\tdL_1(r^a\p^{l-1+a})+\comm{r^a\p^{l-1+a}}{\tdL_1}\tdr,$$
    where
    \begin{equation*}
        \begin{aligned}
            \comm{r^a \p^{l-1+a}}{\tdL_1} \tdr &= (\kappa+1)(l-1+a) r^a H^{ij} \p r \, \p_i \p_j \p^{l-2+a} \tdr \\
            &\quad + (\kappa+1) r^a H^{ij} \sum_{\gamma=2}^{l-1+a} \begin{pmatrix} l-1+a \\ \gamma \end{pmatrix} \p^\gamma r \, \p^{l-1+a-\gamma} \p_i \p_j \tdr \\
            &\quad + \frac{\kappa+1}{\kappa} r^a H^{ij} \sum_{\gamma=1}^{l-1+a} \begin{pmatrix} l-1+a \\ \gamma \end{pmatrix} \p^\gamma \p_i r \, \p^{l-1+a-\gamma} \p_j \tdr \\
            &\quad + (\kappa+1) r^a \sum_{\gamma=1}^{l-1+a} \begin{pmatrix} l-1+a \\ \gamma \end{pmatrix} \p^\gamma H^{ij} \, \p^{l-1+a-\gamma} \left(r \p_i \p_j \tdr + \frac{1}{\kappa} \p_i r \p_j \tdr \right) \\
            &\quad - 2a (\kappa+1) r^a H^{ij} \p_i r \, \p_j \p^{l-1+a} \tdr \\
            &\quad - (\kappa+1) a (a-1) r^{a-1} H^{ij} \p_i r \p_j r \, \p^{l-1+a} \tdr \\
            &\quad - (\kappa+1) a r^a H^{ij} \p_i \p_j r \, \p^{l-1+a} \tdr - \frac{1}{\kappa} (\kappa+1) a r^{a-1} H^{ij} \p_i r \p_j r \, \p^{l-1+a} \tdr \\
            &\simeq r^a \p^{l+a} \tdr + r^{a-1} \p^{l-1+a} \tdr.
        \end{aligned}
    \end{equation*}

    By the induction hypothesis,
    \begin{align*}
        \norm{r^a\p^{l-1+a}\tdL_1\tdr}_{H^{0,\sigma}}&=\norm{\tdL_1(r^a\p^{l-1+a})+\comm{r^a\p^{l-1+a}}{\tdL_1}\tdr}_{H^{0,\sigma}}\\
        &\gtrsim \norm{\tdL_1(r^a\p^{l-1+a})}_{H^{0,\sigma}}-\norm{r^a\p^{l+a}\tdr}_{H^{0,\sigma}}-\norm{r^{a-1}\p^{l-1+a}\tdr}_{H^{0,\sigma}}\\
        &\gtrsim \norm{\tdL_1(r^a\p^{l-1+a})}_{H^{0,\sigma}}-\norm{\tdL_1\tdr}_{H^{2l-2,l-1+\sigma}}-\epsilon\norm{\tdr}_{H^{2l,l+\sigma}}\\
        &\quad -\norm{\tdr}_{H^{2l-2,l-1+\sigma}},
    \end{align*}
    which implies
    \begin{equation*}
        \norm{\tdL_1(r^a\p^{l-1+a})}_{H^{0,\sigma}}\lesssim\norm{\tdL_1\tdr}_{H^{2l-2,l-1+\sigma}}+\epsilon\norm{\tdr}_{H^{2l,l+\sigma}}+\norm{\tdr}_{H^{2l-2,l-1+\sigma}}.
    \end{equation*}

    By the basic elliptic estimate \eqref{BasicREllipitc}, we find
    \begin{equation*}
        \norm{r^a\p^{l-1+a}\tdr}_{H^{2,1+\sigma}}\lesssim \norm{\tdL_1\tdr}_{H^{2l-2,l-1+\sigma}} + \epsilon\norm{\tdr}_{H^{2l,l+\sigma}}+\norm{\tdr}_{H^{2l-2,l-1+\sigma}}.
    \end{equation*}

    Expanding $\norm{r^a\p^{l-1+a}\tdr}_{H^{2,1+\sigma}}$ and by another use of the induction hypothesis, we find
    \begin{align*}
        \norm{r^a\p^{l-1+a}\tdr}_{H^{2,1+\sigma}}&\gtrsim \norm{r\p^2(r^a\p^{l-1+a}\tdr)}_{H^{0,\sigma}}\\
        &\sim\norm{r^{a+1}\p^{l+1+a}\tdr+r^a\p r\p^{l+a}\tdr+r^a\p^2 r\p^{l-1+a}+r^{a-1}\p r\p r\p^{l-1+a}\tdr}_{H^{0,\sigma}}\\
        &\gtrsim\norm{r^{a+1}\p^{l+1+a}\tdr}_{H^{0,\sigma}}-\norm{r^a\p^{l+a}\tdr}_{H^{0,\sigma}}-\norm{r^{a-1}\p^{l-1+a}\tdr}_{H^{0,\sigma}}\\
        &\gtrsim \norm{r^{a+1}\p^{l+1+a}\tdr}_{H^{0,\sigma}}-\norm{\tdL_1\tdr}_{H^{2l-2,l-1+\sigma}}-\epsilon\norm{\tdr}_{H^{2l,l+\sigma}}-\norm{\tdr}_{H^{2l-2,l-1+\sigma}},
    \end{align*}
        which implies the $(a+1)$-th case of the induction,
    \begin{equation*}
        \norm{r^{a+1}\p^{l+1+a}\tdr}\lesssim \norm{\tdL_1\tdr}_{H^{2l-2,l-1+\sigma}} + \epsilon\norm{\tdr}_{H^{2l,l+\sigma}}+\norm{\tdr}_{H^{2l-2,l-1+\sigma}}
    \end{equation*}

    It remains to show the base case $a=0$. To show this, we instead show the following equivalent statement: for $b=0,1,...,l$,
    \begin{equation}
        \norm{\p_n^b\p'^{l-b}\tdr}_{H^{0,\sigma}}\lesssim\norm{\tdL_1\tdr}_{H^{2l-2,l-1+\sigma}} + \epsilon\norm{\tdr}_{H^{2l,l+\sigma}} + \norm{\tdr}_{H^{2l-2,l-1+\sigma}},
    \end{equation}
    where $\p'$ represents the tangential derivatives. We commute $\p_n^b\p^{l-1-b}$ and $\tdL_1$,
    \begin{equation*}
        \begin{aligned}
            \p_n^{b} \p'^{l-1-b} \tdL_1 \tdr &= \tdL_1 \left(\p_n^{b} \p'^{l-1-b} \tdr \right) + \comm{\p_n^{b} \p'^{l-1-b}}{\tdL_1} \tdr \\
            &\simeq \tdL_1^b \left(\p_n^{b} \p'^{l-1-b} \tdr \right) + (\kappa+1)(l - b - 1) \p' r \, H^{ij} \, \p_i \p_j \p_n^{b} \p'^{l-b-2} \tdr \\
            &\quad + (\kappa+1) b \p_n r \, H^{i' j} \, \p_{i'} \p_j \p_n^{b-1} \p'^{l-b-1} \tdr 
            - (\kappa+1) b H^{i' j} \, \p_{i'} r \, \p_j \p_n^b \p'^{l-b-1} \tdr.
        \end{aligned}
    \end{equation*}

    Here, $\tdL_1^b\tdr = (\kappa+1)H^{ij}(r\p_i\p_j\tdr + (\frac{1}{\kappa}+b)\p_i r\p_j \tdr)$ has a similar structure to $\tdL_1^b$ and a review of the proof of \lref{EllipticEstimates} shows that $\tdL_1^b$ satisfies a similar elliptic bound. So, by induction and the smallness of $\p' r$, we have
    $$\norm{\p_n^b\p'^{l-b-1}\tdr}_{H^{1,\sigma}}\lesssim\norm{\p_n^b\p'^{l-b-1}\tdr}_{H^{2,1+\sigma}}\lesssim \norm{\tdL_1\tdr}_{H^{2l-2,l-1+\sigma}} + \epsilon\norm{\tdr}_{H^{2l,l+\sigma}} + \norm{\tdr}_{H^{2l-2,l-1+\sigma}},$$
    which implies the $(b+1)$-th case of the induction
    \begin{equation}
        \norm{\p_n^{b+1}\p'^{l-b-1}\tdr}_{H^{0,\sigma}}\lesssim \norm{\tdL_1\tdr}_{H^{2l-2,l-1+\sigma}} + \epsilon\norm{\tdr}_{H^{2l,l+\sigma}} + \norm{\tdr}_{H^{2l-2,l-1+\sigma}}.
    \end{equation}

    For the base case $b=0$, we commute $\p'^{l-1}$ and $\tdL_1$
    \begin{align*}
        \p'^{l-1}\tdL_1\tdr &= \tdL_1\p'^{l-1}\tdr + \comm{\p'^{l-1}}{\tdL_1}\tdr \\
        &\simeq \tdL_1\p'^{l-1}\tdr + (\kappa+1)(l-1)H^{ij}\p' r\p_i\p_j\p'^{l-2}\tdr.
    \end{align*}

    By another application of \lref{EllipticEstimates}, we have the base case $b=0$. Set $a=l$ in \eqref{HigherOrderEllipticInduction}, we have
    $$\norm{\tdr}_{H^{2l,l+\sigma}}\lesssim \norm{\tdL_1\tdr}_{H^{2l-2,l-1+\sigma}}+\epsilon\norm{\tdr}_{H^{2l,l+\sigma}}+\norm{\tdr}_{H^{2l-2,l-1+\sigma}}.$$
    We soak the term $\epsilon\norm{\tdr}_{H^{2l,l+\sigma}}$ to the LHS, which finishes the proof.
\end{proof}

\vspace{0.5em}
\blue{\begin{remark}
    Similar to \rref{EllipticImplicitConstant}, in the above higher-order elliptic estimates, the implicit constants are depending on $\norm{r,u,\pi}_{L^\infty_tW^{2l+1,\infty}_x}$, i.e., the higher-order elliptic estimates can be written as 
    \begin{align*}
        \norm{\tilde{r}}_{H^{2l,l+\frac{1}{2\kappa}-\frac{1}{2}}}&\lesssim \norm{\tilde{L}_1\tilde{r}}_{H^{2l-2,l-1+\frac{1}{2\kappa}-\frac{1}{2}}} + \scrC(\norm{(r,u,\pi)}_{L^\infty_t W^{2l+1,\infty}_x})\norm{\tilde{r}}_{H^{0,\frac{1}{2\kappa}-\frac{1}{2}}},\\
            \norm{\tilde{u}}_{H^{2,\frac{1}{2\kappa}}}&\lesssim  \norm{\left(\tilde{L}_2+\tilde{L}_3\right)\tilde{u}}_{H^{2l,l+\frac{1}{2\kappa}}} + \scrC(\norm{(r,u,\pi)}_{L^\infty_t W^{2l+1,\infty}_x})\norm{\tilde{u}}_{H^{2l,l+\frac{1}{2\kappa}}}.
    \end{align*}
\end{remark}}

\blue{We now have the higher-order case of  \cref{InvertibleEllipticEstimate}:
\vspace{0.5em}
\begin{corollary}\label{InvertibleHigherOrderEllipticEstimates}
    Let $(r,u,\pi)$ be a sufficiently smooth solution to \eqref{The_nonlinear_system} satisfying the physical vacuum boundary condition \eqref{Physical_vacuum_boundary_condition}. There exists a sufficiently large $\lambda>0$ such that for $\tilde{r}$ sufficiently smooth, we have
   \begin{equation*}
       \norm{\tilde{r}}_{H^{2l,\frac{1}{2\kappa}+l-\frac{1}{2}}}\lesssim \norm{(\hatL_1- \lambda I) \tilde{r}}_{H^{2l-2,\frac{1}{2\kappa}+l-\frac{3}{2}}},
   \end{equation*}
    and for a sufficiently smooth $\tilde{u}$ orthogonal to $u$, we have
   \begin{equation*}
       \norm{\tilde{u}}_{H^{2l,l+\frac{1}{2\kappa}}}\lesssim \norm{\left(\hatL_2 + \hatL_3- \lambda I\right)\tilde{u}}_{H^{2l-2,l-1\frac{1}{2\kappa}}}.
   \end{equation*}
\end{corollary}

This corollary implies that $\hatL_1- \lambda I$ and $\hatL_2 + \hatL_3 - \lambda I$ are injective mapping from $H^{2l,\frac{1}{2\kappa}+l-\frac{1}{2}}$ to $H^{2l-2,\frac{1}{2\kappa}+l-\frac{3}{2}}$ and mapping from $H^{2l,\frac{1}{2\kappa}+l}$ to $H^{2l-2,\frac{1}{2\kappa}+l-1}$. Since injective self-adjoint operators have dense ranges, \lref{SelfAdjoint1} and \lref{SelfAdjoint2} imply $\hatL_1- \lambda I$ and $\hatL_2 + \hatL_3 - \lambda I$ are isomorphisms.

}

\subsection{Local Well-posedness of the Linearized Equations}
We have already established the basic energy estimate in \tref{BasicEnergyEstimateTheorem}. We now proceed to derive the higher-order estimate. The following theorem provides the higher-order energy estimate for the linearized system \eqref{The_linearized_system}.

\blue{In the higher-order estimate, we commute the weighted elliptic operators with the equations instead of the material derivatives. Typically, in higher-order estimates for this type of problem, the standard approach is to commute the equations with two material derivatives  $D_t^2$, followed by an elliptic estimate . However, here we use weighted elliptic operators because  $D_t^2 \tilde{r} \simeq \tilde{L}_1 \tilde{r}$ and  $D_t^2 \tilde{u} \simeq \tilde{L}_2 \tilde{u}$. As a result, within our framework, commuting the material derivatives and the weighted elliptic operators are similar. In this sense, our proof is more in line with \cite{Jang_Masmoudi_2014}, where commutation with weighted elliptic operators is employed, rather than with \cite{disconzi2022relativistic} where commuting with material derivatives is used.}
\vspace{0.5em}
\blue{\begin{theorem}\label{HigherOrderEnergyEstimateTheorem}
    Let $(r,u,\pi)$ be a sufficiently smooth solution defined on $[0,T]$ satisfying the physical vacuum boundary condition \eqref{RealPhysicalVacuumBoundary}. There exists a time existence $T^* = T^*(\norm{(r,u,\pi)}_{L^\infty_t W^{1,\infty}_x})>0$ such that if $(\tdr,\tdu,\tdp)$ is a smooth solution, with $\tdu$ being orthogonal to $u$, to the Cauchy problem \eqref{The_linearized_system}, then it satisfies the following estimate,
    \begin{equation}\label{HigherOrderEnergyEstimate}
        \begin{aligned}
            \norm{(\tdr(t),\tdu(t),\tdp(t))}_{\mathcal{H}^{2l}}&\lesssim \scrC(\norm{(r,u,\pi)}_{L^\infty_tW^{2l+1,\infty}_x})\\
            &\quad\times\left(\norm{(\tdr_0,\tdu_0,\tdp_0)}_{\mathcal{H}^{2l}}+\int^t_0\norm{(\tdf(s),\tdg(s),\tdh(s))}_{\widetilde{\mathcal{H}}^{2l}}ds\right)
        \end{aligned}
    \end{equation}
\end{theorem}
\begin{proof}
    We choose the time existence $T^* = T^*(\norm{(r,u,\pi)}_{L^\infty_t W^{1,\infty}_x})>0$ as given in \lref{BasicEllInftyControl}. We will prove the energy estimate by induction. Assuming the following energy estimate for $l$
    \begin{equation}\label{HigherOrderInductionHypothesis}
        \begin{aligned}
            \norm{(\tdr(t),\tdu(t),\tdp(t))}_{\mathcal{H}^{2l}}&\lesssim \scrC(\norm{(r,u,\pi)}_{L^\infty_tW^{2l+1,\infty}_x})\\
            &\quad\times\left(\norm{(\tdr_0,\tdu_0,\tdp_0)}_{\mathcal{H}^{2l}}+\int^t_0\norm{(\tdf(s),\tdg(s),\tdh(s))}_{\widetilde{\mathcal{H}}^{2l}}ds\right),         
        \end{aligned}
    \end{equation}
    we will show the $(l+1)$-th case.

    We commute the weighted elliptic operators defined in Section \ref{SS:WeightedEllipticOperator} with the linearized system $\eqref{The_linearized_system}$. To be more specific, we commute $\tdL_1$ with \eqref{trEquation}, $\tdL_2+\tdL_3$ with \eqref{tuEquation} and $\tdL_1$ with \eqref{tpEquation} and obtain the following, 
    \begin{align}
        &\Dt\tdL_1\tdr + \kappa ra_1\p_\mu\left[(\tdL_2+\tdL_3)\tdu\right]^\mu + \p_\mu r \left[(\tdL_2+\tdL_3)\tdu\right]^\mu + V_1\tdL_1\tdr + r^2Z_1\tdL_1\tdp = \tdL_1 \tdf + \mathcal{f}\label{HigherOrderTrequation}\\
        &\Dt\left[(\tdL_2+\tdL_3)\tdu\right]^\alpha + (1+\frac{1}{\kappa})a_2\Delta^{\alpha\mu}\p_\mu\tdL_1\tdr + (2+\frac{1}{\kappa})ra_3\Delta^{\alpha\mu}\p_\mu\tdL_1\tdp+(V_2)^\alpha \tdL_1\tdr\nonumber\\ 
        &\quad +(W_2)^\alpha_\beta\left[(\tdL_2+\tdL_3)\tdu\right]^\beta + (Z_2)^\alpha \tdL_1\tdp = \left[(\tdL_2+\tdL_3)\tdg\right]^\alpha + \mathcal{g}^\alpha\\
        &\Dt\tdL_1\tdp + ra_4\p_\mu\left[(\tdL_2+\tdL_3)\tdu\right]^\mu + V_3\tdL_1\tdr + (W_3)_\alpha\left[(\tdL_2+\tdL_3)\tdu\right]^\alpha \nonumber \\
        &\quad +Z_3\tdL_1\tdp = \tdL_1\tdh +\mathcal{h}\\
        \label{fCommutator}
        \mathcal{f} &= \underline{\comm{\Dt}{\hatL_1} \tdr} + \underline{ \kappa r a_1 \p_\mu \left[\left(\hatL_2 + \hatL_3\right) \tdu\right]^\mu - \hatL_1 \left(\kappa r a_1 \p_\mu \tdu^\mu\right)}\nonumber \\
        &\quad + \underline{\p_\mu r \left[\left(\hatL_2 + \hatL_3\right) \tdu\right]^\mu - \hatL_1 \left(\p_\mu r \tdu^\mu\right)} - \left(\hatL_1 V_1\right) \tdr - \hatL_1 \left(r^2 Z_1\right) \tdp, \\
        \label{gcommutator}
        \mathcal{g} &= \underline{\comm{\Dt}{\left(\hatL_2 + \hatL_3\right)} \tdu} + \underline{\left(1 + \frac{1}{\kappa}\right) a_2 \Delta^{\bullet \mu} \p_\mu \left(\hatL_1 \tdr\right)}\nonumber \\
        &\quad \underline{- \left(\hatL_2 + \hatL_3\right) \left[\left(1 + \frac{1}{\kappa}\right) a_2 \Delta^{\bullet \mu} \p_\mu \tdr\right]} + \left(2 + \frac{1}{\kappa}\right) r a_3 \Delta^{\bullet \mu} \p_\mu \left(\hatL_1 \tdp\right) \nonumber\\
        &\quad - \left(\hatL_2 + \hatL_3\right) \left[\left(2 + \frac{1}{\kappa}\right) r a_3 \Delta^{\bullet \mu} \p_\mu \tdp\right]  + V_2 \hatL_1 \tdr - \left(\hatL_2 + \hatL_3\right) \left(V_2 \tdr\right) \nonumber\\
        &\quad+ \left(W_2\right)_\beta \left[\left(\hatL_2 + \hatL_3\right) \tdu\right]^\beta - \left(\hatL_2 + \hatL_3\right) \left(\left(W_2\right)_\beta \tdu^\beta\right) + Z_2 \left(\hatL_1 \tdp\right) \nonumber\\
        &\quad - \left(\hatL_2 + \hatL_3\right) (Z_2 \tdp), \\
        \label{hCommutator}
        \mathcal{h} &= \underline{\comm{\Dt}{\hatL_1} \tdp} + \underline{r a_4 \p_\mu \left[\left(\hatL_2 + \hatL_3\right) \tdu\right]^\mu - \hatL_1 \left(r a_4 \p_\mu \tdu^\mu\right)} \nonumber\\
        &\quad - \hatL_1 (r V_3) \tdr + (W_3)_\alpha \left[\left(\hatL_2 + \hatL_3\right) \tdu\right]^\alpha - \hatL_1 \left((W_3)_\alpha \tdu^\alpha\right) - \hatL_1 Z_3 \tdp.
    \end{align}
    We observe that the above underlined terms may initially appear to be higher-order terms. However, some of them contain commutator terms involving $\Dt$, and since \eqref{rEquation} allows us to trade $\Dt r$ for $r$ to raise the power of $r$, when $\Dt$ acts on $r$, these terms are actually lower-order. The other underlined terms are also lower-order because of a cancellation at the top order. Here, we provide a detailed computation for the terms in $\mathcal{f}$, while those $\mathcal{g}$ and $\mathcal{h}$ follow the same method.

    We begin by analyzing all the terms on the RHS of the first equation \eqref{HigherOrderTrequation}
    $$\tdL_1\tdf + \mathcal{f}.$$

    Using the bookkeeping scheme \nref{Bookkeep}, the terms $\tdL_1\tdf$ can be simplified as $\tdL_1\tdf\simeq r\p^2\tdf + \p\tdf$ and thus,
    \begin{equation}
        \norm{\tdL_1\tdf}_{H^{2l,l+\frac{1}{2\kappa}-\frac{1}{2}}}\lesssim \scrC(\norm{(r,u,\pi)}_{L^\infty_t W^{2l+3,\infty}})\norm{\tdf}_{H^{2l+2,l+\frac{1}{2\kappa}+\frac{1}{2}}}.
    \end{equation}
    Now, we examine the commutator terms,
    $$\mathcal{f} = (i)+(ii)+(iii)+(iv)+(v).$$

    Again, using the bookkeeping scheme \nref{Bookkeep}, the equation \eqref{rEquation} to trade $\Dt r$ for $r$, and \rref{TimeInSpace} and \rref{LinearizedTimeInSpace} to rewrite the time derivatives in terms of spatial derivatives,
    \begin{align*}
        (i)&=\comm{\Dt}{\tdL_1}\tdr\\
        &=\Dt\left((\kappa+1)H^{ij}(r\p_i\p_j\tdr+\frac{1}{\kappa}\p_i r\p_j\tdr)\right)-(\kappa+1)H^{ij}\left(r\p_i\p_j\Dt\tdr +\frac{1}{\kappa}\p_i r\p_j\Dt \tdr\right)\\
        &=(\kappa+1)\Dt H^{ij}(r\p_i\p_j\tdr + \frac{1}{\kappa}\p_ir\p_j\tdr)+(\kappa+1)H^{ij}\underbrace{\Dt(r\p_i\p_j\tdr)}_{\Dt r\p_i\p_j\tdr + r\Dt\p_i\p_j\tdr}\\
        &\quad +(1+\frac{1}{\kappa})H^{ij}\underbrace{\Dt(\p_i r\p_j\tdr)}_{\p_i r\Dt\p_j\tdr + \Dt\p_i r\p_j\tdr}-(\kappa+1)r\p_i\p_j\Dt\tdr - (1+\frac{1}{\kappa})\p_i r\p_j\Dt\tdr\\
        &\simeq (\kappa+1)rH^{ij}\underbrace{(\Dt\p_i\p_j\tdr - \p_i\p_j\Dt\tdr)}_{-2\p_i u^\mu\p_j\p_\mu\tdr -\p_i\p_j u^\mu\p_\mu\tdr} + (1+\frac{1}{\kappa})H^{ij}\p_i r\underbrace{(\Dt\p_j\tdr - \p_j\Dt\tdr)}_{\p_j u^\mu\p_\mu \tdr}\\
        &\simeq r\p^2\tdr + \p\tdr.
    \end{align*}
    As a result, we have the following estimate,
    \begin{equation*}
        \norm{\comm{\Dt}{\tdL_1}\tdr}_{H^{2l,l+\frac{1}{2\kappa}-\frac{1}{2}}}\lesssim \scrC(\norm{(r,u,\pi)}_{L^\infty_t W^{2l+3,\infty}})\norm{(\tdr,\tdu,\tdp)}_{\mathcal{H}^{2l+2}}.
    \end{equation*}

    For term $(ii)$, recalling the definition of $\tdL_1$ and $\tdL_2+\tdL_3$, we have
    \begin{align*}
        (ii) &= \kappa r a_1 \p_\mu\left[(\tdL_2+\tdL_3)\tdu\right]^\mu - \tdL_1(\kappa r a_1\p_\mu\tdu^\mu)\\
        &=\kappa r a_1\p_\mu\left[(\kappa+1)A^{i\mu}H^{jk}\left(r\p_j\p_k\tdu_i + \p_i r\p_j\tdu_k + (1+\frac{1}{\kappa})\p_j r\p_k\tdu_i - \p_j r\p_i\tdu_k\right)\right]\\
        &\quad - (\kappa +1)H^{ij}\left(r\p_i\p_j(\kappa r a_1\p_\mu\tdu^\mu)+\frac{1}{\kappa}\p_i r \p_j(\kappa ra_1\p_\mu\tdu^\mu)\right)\\
        &\simeq \kappa(\kappa+1)ra_1A^{i\mu}H^{jk}\left(r\p_j\p_k\p_\mu \tdu_i + \p_\mu r\p_j\p_k\tdu_i+\p_ir\p_j\p_\mu\tdu_k+(1+\frac{1}{\kappa})\p_j r\p_k\p_\mu\tdu_i-\p_j r\p_i\p_\mu\tdu_k\right)\\
        &\quad - \kappa(\kappa+1)ra_1H^{ij}(r\p_i\p_j\p_\mu\tdu^\mu + 2\p_ir\p_j\p_\mu\tdu^\mu +\frac{1}{\kappa}\p_i r\p_j\p_\mu\tdu^\mu)-(\kappa+1)a_1H^{ij}\p_i r\p_j\p_\mu\tdu^\mu\\
        &\simeq -(\kappa+1)H^{ij}H^{kl}\p_i r\p_j r\p_k\tdu_l.
    \end{align*}
    Here, we note that $(ii)\simeq (\kappa+1)H^{ij}H^{kl}\p_i r\p_j r\p_k\tdu_l$ is not a lower-order term and, therefore, cannot be controlled.

    Now, we compute $(iii)$,
    \begin{align*}
        (iii)&=\p_\mu r\left[(\tdL_2+\tdL_3)\tdu\right]^\mu - \tdL_1(\p_\mu r\tdu^\mu)\\
        &=\p_\mu r (\kappa+1)A^{i\mu}H^{jk}\left(r\p_j\p_k\tdu_i+\p_i r\p_j\tdu_k +(1+\frac{1}{\kappa})\p_jr\p_k\tdu_i-\p_j r\p_i\tdu_K\right)\\
        &\quad-(\kappa+1)H^{ij}\left(r\p_i\p_j(\p_\mu r\tdu^\mu) + \frac{1}{\kappa}\p_i\p_j(\p_\mu r\tdu^\mu)\right)\\
        &\simeq (\kappa+1)H^{ij}H^{kl}\p_j r\left(r\p_l\p_k\tdu_i + \p_i r\p_l\tdu_k + (1+\frac{1}{\kappa})\p_l r\p_k\tdu_i - \p_l r\p_i\tdu_k\right)\\
        &\quad - (\kappa+1)H^{ij}H^{kl}\left(r\p_k r\p_i\p_j\tdu_l+\frac{1}{\kappa}\p_r \p_k r\p_j\tdu_l\right)\\
        &\simeq (\kappa+1)H^{ij}H^{kl}\p_i r\p_j\p_l\tdu_k
    \end{align*}
    Here, we note that $(iii)\simeq (\kappa+1)H^{ij}H^{kl}\p_i r\p_j r\p_l\tdu_k$ is, once again, not a lower-order term and, therefore, cannot be controlled. However, we observe that $(iii)$ actually cancels out with $(ii)$, leading to
    \begin{equation*}
        (ii)+(iii)\simeq 0.
    \end{equation*}

    As a result, we have the following estimate,
    \begin{equation*}
        \norm{(ii)+(iii)}_{H^{2l,l+\frac{1}{2\kappa}-\frac{1}{2}}}\lesssim\scrC(\norm{(r,u,\pi)}_{L^\infty_t W^{2l+3,\infty}})\norm{(\tdr,\tdu,\tdp)}_{\mathcal{H}^{2l+2}}.
    \end{equation*}

    For the terms $(iv) = \tdL_1 V_1 \tdr$ and $(v)=\tdL_1(r^2Z_1)\tdp$, it is clear that
    \begin{align*}
        \norm{\tdL_1 V_1 \tdr}_{H^{2l,l+\frac{1}{2\kappa}-\frac{1}{2}}}&\lesssim\scrC(\norm{(r,u,\pi)}_{L^\infty_t W^{2l+3,\infty}})\norm{(\tdr,\tdu,\tdp)}_{\mathcal{H}^{2l+2}},\\
        \norm{\tdL_1(r^2Z_1)\tdp}_{H^{2l,l+\frac{1}{2\kappa}-\frac{1}{2}}}&\lesssim \scrC(\norm{(r,u,\pi)}_{L^\infty_t W^{2l+3,\infty}})\norm{(\tdr,\tdu,\tdp)}_{\mathcal{H}^{2l+2}}.
    \end{align*}

    Combining the estimates of all the terms of $\mathcal{f}$, we find
    \begin{equation}\label{festimate}
        \norm{\mathcal{f}}_{H^{2l,l+\frac{1}{2\kappa}-\frac{1}{2}}}\lesssim \scrC(\norm{(r,u,\pi)}_{L^\infty_t W^{2l+3,\infty}})\norm{(\tdr,\tdu,\tdp)}_{\mathcal{H}^{2l+2}}.
    \end{equation}

    Using a similar argument as above, we also have
    \begin{align}
        \norm{(\tdL_2+\tdL_3)\tdg}_{H^{2l,l+\frac{1}{2\kappa}}}&\lesssim\scrC(\norm{(r,u,\pi)}_{L^\infty_t W^{2l+3,\infty}})\norm{\tdg}_{H^{2l+2,l+1+\frac{1}{2\kappa}}},\\
        \label{gestimate}
        \norm{\mathcal{g}}_{H^{2l,l+\frac{1}{2\kappa}}}&\lesssim \scrC(\norm{(r,u,\pi)}_{L^\infty_t W^{2l+3,\infty}})\norm{(\tdr,\tdu,\tdp)}_{\mathcal{H}^{2l+2}},\\
        \norm{\tdL_1\tdh}_{H^{2l,l+\frac{1}{2\kappa}}}&\lesssim \scrC(\norm{(r,u,\pi)}_{L^\infty_t W^{2l+3,\infty}})\norm{\tdh}_{H^{2l+2,l+1+\frac{1}{2\kappa}}},\\
        \label{hestimate}
        \norm{\mathcal{h}}_{H^{2l,l+\frac{1}{2\kappa}}}&\lesssim \scrC(\norm{(r,u,\pi)}_{L^\infty_t W^{2l+3,\infty}})\norm{(\tdr,\tdu,\tdp)}_{\mathcal{H}^{2l+2}}.
    \end{align}

    Using the induction hypothesis \eqref{HigherOrderInductionHypothesis}, we have
    \begin{align*}
        \norm{(\tdL_1\tdr(t),(\tdL_2+\tdL_3)\tdu(t),\tdL_1\tdp(t))}_{\mathcal{H}^{2l}}&\lesssim \norm{(\tdL_1\tdr_0,(\tdL_2+\tdL_3)\tdu_0,\tdL_1\tdp_0)}_{\mathcal{H}^{2l}}\\
        &\quad +\int^t_0\norm{(\tdL_1\tdf + \mathcal{f},(\tdL_2+\tdL_3)\tdg+\mathcal{g},\tdL_1\tdh + \mathcal{h})}_{\widetilde{\mathcal{H}}^{2l}}ds\\
        &\lesssim \scrC(\norm{(r,u,\pi)}_{L^\infty_t W^{2l+3,\infty}})\\
        &\quad \times\left(\norm{(\tdr_0,\tdu_0,\tdp_0)}_{\mathcal{H}^{2l+2}} + \int_0^t\norm{(\tdf,\tdg,\tdh)}_{\widetilde{\mathcal{H}}^{2l+2}}ds\right)\\
        &\quad +\scrC(\norm{(r,u,\pi)}_{L^\infty_t W^{2l+3,\infty}})\\
        &\quad \times \int_0^t\norm{(\tdr(t),\tdu(t),\tdp(t))}_{\mathcal{H}^{2l+2}}ds.
    \end{align*}

    Then, by the \lref{HighOrderEllipticEstimates} and a simple rearrangement of the inequality, we have
    \begin{align*}
        \norm{(\tdr(t),\tdu(t),\tdp(t))}_{\mathcal{H}^{2l+2}}&\lesssim \scrC(\norm{(r,u,\pi)}_{L^\infty_t W^{2l+3,\infty}})\left(\norm{(\tdr_0,\tdu_0,\tdp_0)}_{\mathcal{H}^{2l+2}} + \int_0^t\norm{(\tdf,\tdg,\tdh)}_{\widetilde{\mathcal{H}}^{2l+2}}ds\right)\\
        &\quad +\scrC(\norm{(r,u,\pi)}_{L^\infty_t W^{2l+3,\infty}})\int_0^t\norm{(\tdr(t),\tdu(t),\tdp(t))}_{\mathcal{H}^{2l+2}}ds.
    \end{align*}

    By the standard Gronwall inequality, we obtain the (l+1)-th case of the induction. The base case has already been established, as given by the basic energy estimate \tref{BasicEnergyEstimateTheorem}.
\end{proof}}

Here, we establish the local well-posedness of the linearized equation. The idea of the proof is similar to the proof of Proposition 5.6 in \cite{Jang_Masmoudi_2014}.

\vspace{0.5em}
\blue{\begin{theorem}
    Let $(r,u,\pi)$ be a sufficiently regular solution to the system \eqref{The_nonlinear_system} defined on $[0,T]$ such that it satisfies the physical vacuum boundary condition \eqref{RealPhysicalVacuumBoundary}. There exists $T^*=T^*(\norm{(r,u,\pi)}_{W^{1,\infty}})>0$ such that for any $(\tilde{f},\tilde{g},\tilde{h})\in L^1_t\widetilde{\mathcal{H}}^{2l}$, there exists a unique solution $(\tdr,\tdu,\tdp)\in\mathcal{H}^{2l}$ to the Cauchy problem \eqref{The_linearized_system} satisfying the following estimate,
    \begin{equation}
    \begin{aligned}
        \norm{(\tdr(t),\tdu(t),\tdp(t))}_{{\mathcal{H}}^{2l}}&\lesssim \mathscr{C}(\norm{(r,u,\pi)}_{W^{2l+1,\infty}})\\
        &\quad \times\left(\norm{(\tilde{r}_0,\tilde{u}_0,\tilde{\pi}_0)}_{{\mathcal{H}}^{2l}}+\int_0^t\norm{(\tdf(s),\tdg(s),\tdh(s))}_{\widetilde{\mathcal{H}}^{2l}}\,\mathrm{d}s\right)
    \end{aligned}
    \end{equation}
    for any $t\in [0,T^*]$.
\end{theorem}}
\begin{proof}
    We will construct the solution by using a duality argument. We first consider the zero initial data case, i.e., $(\tilde{r}_0,\tilde{u}_0,\tilde{\pi}_0)=(0,0,0)$. Define the set of all test functions as
    \blue{    \begin{equation}
        \mathcal{A}:= \left\{(\tilde{\phi},\tilde{\psi},\tilde{\chi})\in C^\infty_c(\mathscr{D}_{T^*})\mid u^\mu\tilde{\psi}_\mu = 0,(\tilde{\phi},\tilde{\psi},\tilde{\chi})|_{t=T^*}=0\right\},
    \end{equation}}
    which is dense in $\mathcal{H}$. Let $\mathcal{L}$ denote the linear operator defined as in \eqref{WeightedLinearizedOperator}. We consider the weak formulation of the Cauchy problem \eqref{The_linearized_system}, i.e., we say that $(\tdr,\tdu,\tdp)$ solves \eqref{The_linearized_system} weakly if 
    \begin{equation}\label{WeakFormulation}
        \int_0^{T^*}\int_{\Omega_t}\begin{pmatrix}\tdr\\\tdu\\\tdp\end{pmatrix}\cdot\mathcal{L}^*\begin{pmatrix}\tilde{\phi}\\\tilde{\psi}\\\tilde{\chi}\end{pmatrix}dxdt=\int_0^{T^*}\int_{\Omega_t}\begin{pmatrix}r^{\frac{1}{\kappa}-1}\tdf \\ \frac{\kappa^2}{\kappa + 1}r^\frac{1}{\kappa}a_1^2 G^{\alpha\bullet}\tdg_\alpha\\ \frac{(2\kappa+1)\kappa}{\kappa+1}r^\frac{1}{\kappa}\frac{a_1a_0^{2+\frac{2}{\kappa}}}{b(r)}\tdh\end{pmatrix}\cdot\begin{pmatrix}\tilde{\phi}\\\tilde{\psi}\\\tilde{\chi}\end{pmatrix}dxdt,
    \end{equation}
    for any $\begin{pmatrix}\tilde{\phi}\\\tilde{\psi}\\\tilde{\chi}\end{pmatrix}\in\mathcal{A}$, where the "dot product" is defined as follows:
    \begin{equation*}
        \begin{pmatrix}\tilde{\phi} \\ \tilde{\psi}\\ \tilde{\chi}\end{pmatrix}\cdot\begin{pmatrix}\bar{\phi} \\ \bar{\psi}\\ \bar{\chi}\end{pmatrix} = \tilde{\phi}\bar{\phi} + \tilde{\psi}^\alpha\bar{\psi}_\alpha + \tilde{\chi}\bar{\chi},
    \end{equation*}
    and $\mathcal{L}^*$ is the formal adjoint of $\mathcal{L}$.

    A similar method as in the proof of the basic energy estimate \eqref{BasicEnergyEstimate} applying to the adjoint operator $\mathcal{L}^*$, along with the subsequent argument for the time-reverse estimate, yields:
    \begin{equation*}
        \norm{\begin{pmatrix}\tilde{\phi} \\ \tilde{\psi}\\ \tilde{\chi}\end{pmatrix}}_{L^\infty_t\mathcal{H}}\lesssim \mathscr{C}(\norm{(r,u,\pi)}_{L^\infty_tW^{1,\infty}_x})\norm{\mathcal{L}^*\begin{pmatrix}\tilde{\phi}\\\tilde{\psi}\\\tilde{\chi}\end{pmatrix}}_{L^1_t\mathcal{H}^*}.
    \end{equation*}

    Given $(\tdf,\tdg,\tdh)\in L^1_t\widetilde{\mathcal{H}}$, we define a map $J_0:\mathcal{L}^*(\mathcal{A})\to\R$ by
    \begin{equation*}
        J_0\left(\mathcal{L}^*\begin{pmatrix}\tilde{\phi}\\\tilde{\psi}\\\tilde{\chi}\end{pmatrix}\right) = \int_0^T\int_{\Omega_t}\begin{pmatrix}r^{\frac{1}{\kappa}-1}\tdf \\ \frac{\kappa^2}{\kappa + 1}r^\frac{1}{\kappa}a_1^2 G^{\alpha\beta}\tdg_\alpha\\ \frac{(2\kappa+1)\kappa}{\kappa+1}r^\frac{1}{\kappa}\frac{a_1a_0^{2+\frac{2}{\kappa}}}{b(r)}\tdh\end{pmatrix}\cdot\begin{pmatrix}\tilde{\phi}\\\tilde{\psi}\\\tilde{\chi}\end{pmatrix}\,dxdt.
    \end{equation*}
    Then, the map $J_0$ extends linearly to a map defined on the span $span\,\mathcal{L}^*(\mathcal{A})$ of $\mathcal{L}^*(\mathcal{A})$, which, with an abuse of notation that causes no confusion, denoted again by $J_0$. Since $\mathcal{L}^*\begin{pmatrix}\tilde{\phi}\\\tilde{\psi}\\\tilde{\chi}\end{pmatrix}\in{L_t^1\mathcal{H^*}}$, we have $\mathcal{L}^*(\mathcal{A})\subset {L_t^1\mathcal{H^*}}$, which implies $span\,\mathcal{L}^*(\mathcal{A})$ is a subspace of ${L_t^1\mathcal{H^*}}$. So, $J_0$ is a linear functional defined on a subspace of ${L_t^1\mathcal{H^*}}$.

    We show that $J_0$ is bounded by the following,
    \begin{equation*}
        \begin{aligned}
            \abs{J_0\left(\mathcal{L}^*\begin{pmatrix}\tilde{\phi}\\\tilde{\psi}\\\tilde{\chi}\end{pmatrix}\right)}
            &\lesssim\mathscr{C}(\norm{(r,u,\pi)}_{L^\infty_tW^{1,\infty}_x})\norm{\begin{pmatrix}\tdf \\ \tdg \\ \tdh\end{pmatrix}}_{L^1_t\widetilde{\mathcal{H}}}\norm{\begin{pmatrix}\tilde{\phi}\\\tilde{\psi}\\\tilde{\chi}\end{pmatrix}}_{L^\infty_t\mathcal{H}}\\
            &\lesssim \mathscr{C}(\norm{(r,u,\pi)}_{L^\infty_tW^{1,\infty}_x})\norm{\begin{pmatrix}\tdf \\ \tdg \\ \tdh\end{pmatrix}}_{L^1_t\widetilde{\mathcal{H}}}\norm{\mathcal{L}^*\begin{pmatrix}\tilde{\phi}\\\tilde{\psi}\\\tilde{\chi}\end{pmatrix}}_{L_t^1\mathcal{H^*}}.
        \end{aligned}
    \end{equation*}
    By Hahn-Banach Theorem, \blue{we can extend $J_0$ to $J\in \left(L^1_t(\mathcal{H}^*)\right)^*$, i.e., J is a bounded linear functional defined on $L^1_t\mathcal{H}^*$. By an application of the Riesz representation theorem, there exists a unique $(\tdr,\tdu,\tdp)\in L^\infty_t(\mathcal{H})$ such that
    $$J\left(\begin{pmatrix}
        \tilde{\Phi}\\\tilde{\Psi}\\\tilde{\mathrm{X}}
    \end{pmatrix}\right)=\int_0^{T^*}\int_{\Omega_t}\begin{pmatrix}\tdr\\\tdu\\\tdp\end{pmatrix}\cdot\begin{pmatrix}
        \tilde{\Phi}\\\tilde{\Psi}\\\tilde{\mathrm{X}}
    \end{pmatrix}dxdt,$$
    for every $\begin{pmatrix}
        \tilde{\Phi}\\\tilde{\Psi}\\\tilde{\mathrm{X}}
    \end{pmatrix}\in L^1_t(\mathcal{H}^*)$. We restrict $J$ back to the set $\mathcal{L}^*(\mathcal{A})$, i.e., we take $\begin{pmatrix}
        \tilde{\Phi}\\\tilde{\Psi}\\\tilde{\mathrm{X}}
    \end{pmatrix}= \mathcal{L}^*\begin{pmatrix}\tilde{\phi}\\\tilde{\psi}\\\tilde{\chi}\end{pmatrix}$ to get a weak solution to \eqref{WeakFormulation}.}

   \blue{ To construct more regular solution, it is sufficient to prove the case where $l=1$. The construction of solutions with higher regularity $l>1$ then follows inductively by using the same approach as for $l=1$.

    We apply $\hatL_1-\lambda I$ to \eqref{trEquation}, $\hatL_2+\hatL_3-\lambda I$ to \eqref{tuEquation} and $\hatL_1-\lambda I$ to \eqref{tpEquation}, we have
    \begin{subequations}\label{BeforeCommutation}
    \begin{align}
        \label{EllipticTrEquation}
        & (\hatL_1-\lambda I)\left(u^\mu \partial_\mu \tilde{r} + \kappa r a_1 \partial_\mu \tilde{u}^\mu 
        + \partial_\mu r \tilde{u}^\mu + V_1 \tilde{r} + r^2 Z_1 \tilde{\pi} - \tdf\right)=0, \\
        \label{EllipticTuEquation}
        & (\hatL_2+\hatL_3-\lambda I)\biggl(u^\mu \partial_\mu \tilde{u} + \left(1 + \frac{1}{\kappa}\right) a_2 \Delta^{\bullet\mu} \partial_\mu \tilde{r} 
        + \left(2 + \frac{1}{\kappa}\right) r a_3 \Delta^{\bullet\mu} \partial_\mu \tilde{\pi} \nonumber \\
        & \quad + (V_2)^\bullet \tilde{r} + (W_2)^\bullet_\beta \tilde{u}^\beta + (Z_2)^\bullet \tilde{\pi} - \tdg^\bullet\biggr)=0, \\
        \label{EllipticTpEquation}
        & (\hatL_1-\lambda I)\left(u^\mu \partial_\mu \tilde{\pi} + r a_4 \partial_\mu \tilde{u}^\mu 
        + V_3 \tilde{r} + (W_3)_\alpha \tilde{u}^\alpha + Z_3 \tilde{\pi} - \tdh\right) = 0.
        \end{align}
\end{subequations}
    We commute these invertible weighted elliptic operators with the linearized system, we obtain
    \begin{subequations}\label{AfterCommutation}
        \begin{align}
            &\Dt\tdL_1\tdr + \kappa ra_1\p_\mu\left[(\tdL_2+\tdL_3)\tdu\right]^\mu + \p_\mu r \left[(\tdL_2+\tdL_3)\tdu\right]^\mu + V_1\tdL_1\tdr + r^2Z_1\tdL_1\tdp = \tdL_1 \tdf + \mathcal{f}\\
        &\Dt\left[(\tdL_2+\tdL_3)\tdu\right]^\alpha + (1+\frac{1}{\kappa})a_2\Delta^{\alpha\mu}\p_\mu\tdL_1\tdr + (2+\frac{1}{\kappa})ra_3\Delta^{\alpha\mu}\p_\mu\tdL_1\tdp+(V_2)^\alpha \tdL_1\tdr\nonumber\\ 
        &\quad +(W_2)^\alpha_\beta\left[(\tdL_2+\tdL_3)\tdu\right]^\beta + (Z_2)^\alpha \tdL_1\tdp = \left[(\tdL_2+\tdL_3)\tdg\right]^\alpha + \mathcal{g}^\alpha\\
        &\Dt\tdL_1\tdp + ra_4\p_\mu\left[(\tdL_2+\tdL_3)\tdu\right]^\mu + V_3\tdL_1\tdr + (W_3)_\alpha\left[(\tdL_2+\tdL_3)\tdu\right]^\alpha \nonumber \\
        &\quad +Z_3\tdL_1\tdp = \tdL_1\tdh +\mathcal{h}.
        \end{align}
    \end{subequations}
    Here, $\mathcal{f}=\mathcal{f}[\tdr,\tdu,\tdp]$, $\mathcal{g}=\mathcal{g}[\tdr,\tdu,\tdp]$ and $\mathcal{h}=\mathcal{h}[\tdr,\tdu,\tdp]$ are very similar to those defined as in \eqref{fCommutator}, \eqref{gcommutator} and \eqref{hCommutator}. The only difference is that here we use $\hatL_1-\lambda I$ and $\hatL_2+\hatL_3-\lambda I$ instead of $\tdL_1$ and $\tdL_2+\tdL_3$. By a similar estimate as in the proof of \tref{HigherOrderEnergyEstimateTheorem}, we have a similar bound for $\mathcal{f},\mathcal{g}$ and $\mathcal{h}$ as in \eqref{festimate}, \eqref{gestimate} and \eqref{hestimate},
    \begin{align*}
        \norm{\mathcal{f}[\tdr,\tdu,\tdp]}_{H^{0,\frac{1}{2\kappa}-\frac{1}{2}}} &\lesssim \scrC(\norm{(r,u,\pi)}_{L^\infty_tW^{3,\infty}_x})\norm{(\tdr,\tdu,\tdp)}_{\mathcal{H}^{2}},\\
            \norm{\mathcal{g}[\tdr,\tdu,\tdp]}_{H^{0,\frac{1}{2\kappa}}} &\lesssim \scrC(\norm{(r,u,\pi)}_{L^\infty_tW^{3,\infty}_x})\norm{(\tdr,\tdu,\tdp)}_{\mathcal{H}^{2}},\\
            \norm{\mathcal{h}[\tdr,\tdu,\tdp]}_{H^{0,\frac{1}{2\kappa}}} &\lesssim \scrC(\norm{(r,u,\pi)}_{L^\infty_tW^{3,\infty}_x})\norm{(\tdr,\tdu,\tdp)}_{\mathcal{H}^{2}}.
    \end{align*}
}

    We also note that $(\tdr,\tdu,\tdp)$ solves \eqref{AfterCommutation} if and only if it solves \eqref{BeforeCommutation}.

    Now, we define $\begin{pmatrix}
        \hat{r}^0 \\ \hat{u}^0 \\ \hat{\pi}^0
    \end{pmatrix}\in \mathcal{H}$ by solving the Cauchy problem \eqref{The_linearized_system} with the source terms replaced by $\begin{pmatrix}
        (\hatL_1-\lambda I)\tdf \\ (\hatL_2+\hatL_3-\lambda I)\tdg \\ (\hatL_1-\lambda I) \tdh
    \end{pmatrix}\in \mathcal{H}$, then we have
    \begin{align*}
      \norm{\begin{pmatrix}
        \hat{r}^0 \\ \hat{u}^0 \\ \hat{\pi}^0
    \end{pmatrix}}_{L^\infty_t\mathcal{H}}&\lesssim \scrC(\norm{(r,u,\pi)}_{L^1_tW^{1,\infty}_x})\norm{\begin{pmatrix}
        (\hatL_1-\lambda I)\tdf \\ (\hatL_2+\hatL_3-\lambda I)\tdg \\ (\hatL_1-\lambda I) \tdh
    \end{pmatrix}}_{L^\infty_t\mathcal{H}}\\
    &\lesssim \scrC(\norm{(r,u,\pi)}_{L^\infty_tW^{3,\infty}_x})T^{**}\norm{\begin{pmatrix}
        \tdf \\ \tdg \\ \tdh
    \end{pmatrix}}_{L^\infty_t\mathcal{H}^2}.
    \end{align*}

    Now, for each $i$, we define $\begin{pmatrix}
        \hat{r}^i \\ \hat{u}^i \\ \hat{\pi}^i
    \end{pmatrix}\in\mathcal{H}$ inductively as a solution to \eqref{The_linearized_system} with the source terms replaced by
    \begin{equation*}
        \begin{pmatrix}
            \mathcal{f}[(\hatL_1-\lambda I)^{-1}\hat{r}^{i-1},(\hatL_2+\hatL_3-\lambda I)^{-1}\hat{u}^{i-1},(\hatL_1-\lambda I)^{-1}\hat{\pi}^{i-1}]\\
            \mathcal{g}[(\hatL_1-\lambda I)^{-1}\hat{r}^{i-1},(\hatL_2+\hatL_3-\lambda I)^{-1}\hat{u}^{i-1},(\hatL_1-\lambda I)^{-1}\hat{\pi}^{i-1}]\\
            \mathcal{h}[(\hatL_1-\lambda I)^{-1}\hat{r}^{i-1},(\hatL_2+\hatL_3-\lambda I)^{-1}\hat{u}^{i-1},(\hatL_1-\lambda I)^{-1}\hat{\pi}^{i-1}]
        \end{pmatrix}\in\mathcal{H}.
    \end{equation*}

    Then, we have the estimate,
    \begin{align*}
        \norm{\begin{pmatrix}
        \hat{r}^i \\ \hat{u}^i \\ \hat{\pi}^i
    \end{pmatrix}}_{L^\infty_t\mathcal{H}}&\lesssim \norm{\begin{pmatrix}
            \mathcal{f}[(\hatL_1-\lambda I)^{-1}\hat{r}^{i-1},(\hatL_2+\hatL_3-\lambda I)^{-1}\hat{u}^{i-1},(\hatL_1-\lambda I)^{-1}\hat{\pi}^{i-1}]\\
            \mathcal{g}[(\hatL_1-\lambda I)^{-1}\hat{r}^{i-1},(\hatL_2+\hatL_3-\lambda I)^{-1}\hat{u}^{i-1},(\hatL_1-\lambda I)^{-1}\hat{\pi}^{i-1}]\\
            \mathcal{h}[(\hatL_1-\lambda I)^{-1}\hat{r}^{i-1},(\hatL_2+\hatL_3-\lambda I)^{-1}\hat{u}^{i-1},(\hatL_1-\lambda I)^{-1}\hat{\pi}^{i-1}]
        \end{pmatrix}}_{L^1_t\mathcal{H}}\\
        &\lesssim\scrC(\norm{(r,u,\pi)}_{L^\infty_tW^{3,\infty}_x})T^{**}\norm{\begin{pmatrix}
            (\hatL_1-\lambda I)^{-1}\hat{r}^{i-1} \\ (\hatL_2+\hatL_3-\lambda I)^{-1}\hat{u}^{i-1} \\(\hatL_1-\lambda I)^{-1}\hat{\pi}^{i-1}
        \end{pmatrix}}_{L^\infty_t\mathcal{H}^2}\\
        &\lesssim \scrC(\norm{(r,u,\pi)}_{L^\infty_tW^{3,\infty}_x})T^{**}\norm{\begin{pmatrix}
            \hat{r}^{i-1} \\ \hat{u}^{i-1} \\ \hat{\pi}^{i-1}
        \end{pmatrix}}_{L^\infty_t\mathcal{H}}\\
        &\lesssim \cdots\\
        &\lesssim \left(\scrC(\norm{(r,u,\pi)}_{L^\infty_tW^{3,\infty}_x})T^{**}\right)^{i+1}\norm{\begin{pmatrix}
            \tdf \\ \tdg \\ \tdh
        \end{pmatrix}}_{L^1_t\mathcal{H}^2}.
    \end{align*}

    If $T^{**}$ is chosen small enough, then $\begin{pmatrix}
        \hat{r} \\ \hat{u} \\ \hat{\pi}
    \end{pmatrix} = \sum_{i=0}^\infty \begin{pmatrix}
        \hat{r}^i \\ \hat{u}^i \\ \hat{\pi}^i
    \end{pmatrix}\in L^\infty_t\mathcal{H}$ converges. Define $\tdr := (\hatL_1-\lambda I)^{-1}\hat{r}, \tdu := (\hatL_2+\hatL_3-\lambda I)^{-1}\hat{u}$ and $\tdp := (\hatL_1-\lambda I)^{-1}\hat{\pi}$, then $(\tdr,\tdu,\tdp)\in L^\infty_t\mathcal{H}^2$ satisfies \eqref{AfterCommutation}, which implies it satisfies \eqref{BeforeCommutation}. By invertibility of $\hatL_1-\lambda I$ and $\hatL_2+\hatL_3-\lambda I$, $(\tdr,\tdu,\tdp)$ solves \eqref{The_linearized_system} with more regularity.

    \blue{However, the small time $T^{**}$ we chose may potentially  be smaller than $T^*$. We will show that if that is the case, the solution can be extended up to $T^*$.

    We consider the non-zero data case, i.e., when $(\tdr_0,\tdu_0,\tdp_0)\neq 0$. By density, let $(\tdr_{0,n},\tdu_{0,n},\tdp_{0,n})\in C^\infty_c(\Omega_0)\to (\tdr_0,\tdu_0,\tdp_0)$ in $\mathcal{H}^2$. We first solve for the solutions $(\tdr_n,\tdu_n,\tdp_n)$ to \eqref{The_linearized_system} with the initial data prescribed by $(\tdr_{0,n},\tdu_{0,n},\tdp_{0,n})$. Define $\tds_n:=\tdr_n-\tdr_{0,n}$, $\tdv_n:= \tdu_n-\tdu_{0,n}$ and $\tdz_n := \tdp_n-\tdp_{0,n}$, then it is equivalent to solve the following Cauchy problem for $(\tds_n,\tdv_n,\tdz_n)$:
    \begin{align*}
        & u^\mu \partial_\mu \tds_n + \kappa r a_1 \partial_\mu \tdv_n^\mu 
        + \partial_\mu r \tdv_n^\mu + V_1 \tds_n + r^2 Z_1 \tdz_n = \tdf-u^i\p_i\tdr_{0,n}-\kappa ra_1\p_i\tdu_{0,n}^i \\
        &\quad- \p_\mu r \tdu_{0,n}^\mu - V_1\tdr_{0,n}-r^2Z_1\tdp_{0,n}, \\
        & u^\mu \partial_\mu \tdv_n^\alpha + \left(1 + \frac{1}{\kappa}\right) a_2 \Delta^{\alpha\mu} \partial_\mu \tds_n
        + \left(2 + \frac{1}{\kappa}\right) r a_3 \Delta^{\alpha\mu} \partial_\mu \tdz_n \\
        & \quad + (V_2)^\alpha \tds_n + (W_2)^\alpha_\beta \tdv_n^\beta + (Z_2)^\alpha \tdz_n = \tdg^\alpha-u^i\p_i\tdu_{0,n}^\alpha-\left(1 + \frac{1}{\kappa}\right) a_2 \Delta^{\alpha i}\p_i\tdr_{0,n}\\
        &\quad -\left(2 + \frac{1}{\kappa}\right) r a_3 \Delta^{\alpha i} \partial_i\tdp_{0,n}-(V_2)^\alpha\tdr_{0,n}-(W_2)_\beta^\alpha \tdu_{0,n}^\beta-(Z_2)^\alpha\tdp_{0,n}, \\
        & u^\mu \partial_\mu \tdz_n + r a_4 \partial_\mu \tdv_n^\mu 
        + V_3 \tds_n + (W_3)_\alpha \tdv_n^\alpha + Z_3 \tdz_n = \tdh-u^i\p_i\tdp_{0,n}-ra_4\p_i\tdu_{0,n}^i \\
        &\quad -V_3\tdr_{0,n}-(W_3)_\alpha\tdu_{0,n}^\alpha -Z_3\tdp_{0,n} , \\
        & u^\mu (\tdv_n)_\mu = 0, \\
        & (\tds_n, \tdv_n, \tdz_n) \big|_{t=0} = 0.
        \end{align*}

        The above system is solvable up to a nonzero time $T^{**}$ with our zero data case, as the source terms in the RHS belong to $L^1_t\mathcal{H}^{2}$. By \tref{HigherOrderEnergyEstimateTheorem}, we can show that $(\tdr_n,\tdu_n,\tdp_n)\in L^\infty_t\mathcal{H}^2$ is Cauchy and thus $(\tdr_n,\tdu_n,\tdp_n)\to(\tdr,\tdu,\tdp)$ in $L^\infty_t\mathcal{H}^2$. Notice that $(\tdr,\tdu,\tdp)$ solves the Cauchy problem \eqref{The_linearized_system} with nonzero data.

        Since we can solve the nonzero data case, we can continue the solution up to $T^*$ with the energy estimate in \tref{HigherOrderEnergyEstimateTheorem} by a continuation argument, thereby completing the proof.}
\end{proof}
\blue{\begin{remark}
    We observe that the estimate for the linearized equation display a genuine loss of derivatives in the coefficients. To overcome the regularity loss and establish a solution to the full nonlinear problem, we will employ a Nash-Moser iteration scheme tailored to this free boundary problem in the upcoming work \cite{Zhong-in_prep}.
\end{remark}}

\section*{Acknowledgements}
RZ gratefully acknowledges support from NSF grant DMS-2107701.

\end{document}